\documentclass[11pt,reqno]{amsart}

\usepackage{amssymb,mathtools,bm}
\usepackage{enumitem}
\usepackage[expansion=false]{microtype}
\usepackage{url}
\usepackage{aliascnt}
\usepackage{booktabs,array}
\usepackage{needspace}
\usepackage[hidelinks]{hyperref}
\usepackage[capitalize,nameinlink]{cleveref}
\crefname{equation}{equation}{equations}
\Crefname{equation}{Equation}{Equations}
\crefname{section}{Section}{Sections}
\crefname{subsection}{Section}{Sections}
\crefname{appendix}{Appendix}{Appendices}

\theoremstyle{plain}
\newtheorem{theorem}{Theorem}[section]
\newaliascnt{proposition}{theorem}
\newtheorem{proposition}[proposition]{Proposition}
\aliascntresetthe{proposition}
\crefname{proposition}{Proposition}{Propositions}
\newaliascnt{lemma}{theorem}
\newtheorem{lemma}[lemma]{Lemma}
\aliascntresetthe{lemma}
\crefname{lemma}{Lemma}{Lemmas}
\newaliascnt{corollary}{theorem}
\newtheorem{corollary}[corollary]{Corollary}
\aliascntresetthe{corollary}
\crefname{corollary}{Corollary}{Corollaries}
\theoremstyle{definition}
\newaliascnt{remark}{theorem}
\newtheorem{remark}[remark]{Remark}
\aliascntresetthe{remark}
\crefname{remark}{Remark}{Remarks}
\newaliascnt{example}{theorem}
\newtheorem{example}[example]{Example}
\aliascntresetthe{example}
\crefname{example}{Example}{Examples}
\newaliascnt{definition}{theorem}
\newtheorem{definition}[definition]{Definition}
\aliascntresetthe{definition}
\crefname{definition}{Definition}{Definitions}
\newaliascnt{question}{theorem}
\newtheorem{question}[question]{Question}
\aliascntresetthe{question}
\crefname{question}{Question}{Questions}
\numberwithin{equation}{section}

\newcommand{\E}{\mathbb{E}}
\newcommand{\Pp}{\mathbb{P}}
\newcommand{\R}{\mathbb{R}}
\newcommand{\C}{\mathbb{C}}
\newcommand{\1}{\mathbf{1}}
\newcommand{\Tr}{\operatorname{tr}}
\newcommand{\Var}{\operatorname{Var}}
\newcommand{\Cov}{\operatorname{Cov}}
\newcommand{\ESD}{\operatorname{ESD}}

\newcommand{\dd}{\,\mathrm{d}}
\newcommand{\eps}{\varepsilon}
\newcommand{\Imc}{\operatorname{Im}}
\newcommand{\Rec}{\operatorname{Re}}
\newcommand{\diag}{\operatorname{diag}}
\newcommand{\supp}{\operatorname{supp}}
\newcommand{\law}{\operatorname{Law}}
\newcommand{\spec}{\operatorname{spec}}
\newcommand{\rank}{\operatorname{rank}}
\newcommand{\ran}{\operatorname{ran}}
\newcommand{\dK}{d_{\mathrm K}}
\newcommand{\MP}{\mathrm{MP}}
\newcommand{\xrightarrowp}{\xrightarrow{\ \Pp\ }}

\title[Radial MP laws and projection rigidity]{Radial Marchenko--Pastur laws:\\ projection characterizations and rigidity}
\author{Xiaohui Xie}
\address{University of California, Irvine, California 92697, USA}
\email{xhx@uci.edu}
\date{September 2026}

\begin{document}

\begin{abstract}
Let $R_p=\|x_p\|^2/p\Rightarrow\nu$ and $p/N\to c\in(0,\infty)$.
We prove that the radial quadratic-form condition (RQC), which requires
energy in deterministic subspaces to follow the parent radius, holds
if and only if every deterministic rank-$\lfloor\alpha p\rfloor$ projected
covariance has empirical spectral distribution converging to
$\mu_{c\alpha,\nu}$, for one fixed $\alpha\in(0,1)$.
The law $\nu$ may have arbitrary tails: no moments, conditional isotropy,
or independence between radius and direction are required.
We also identify the radius law from these spectra when their common limit
has finite second moment and subspaces of vanishing relative dimension
carry vanishing normalized energy, without moment assumptions on the
original vectors.
A finite-sample inverse bounds quadratic-form defects by bounded tests
of expected projected spectra.
The converse combines angular symmetrization and fantope rigidity with a
determinant comparison between radial column deletion and spectral trimming.
Cancellation of the deleted radial factors makes the comparison error
depend on the deleted fraction, not its magnitudes.
\end{abstract}

\maketitle
\raggedbottom

\section{Introduction}
\label{sec:intro}

Let $x_1,\ldots,x_N$ be independent copies of a random vector $x_p\in\R^p$,
and let
\[
 S_p=\frac1N\sum_{j=1}^Nx_jx_j^T,\qquad p/N\to c\in(0,\infty).
\]
The Marchenko--Pastur theorem \cite{MarchenkoPastur1967} and its extensions
to dependent coordinates \cite{BaiZhou2008,PajorPastur2009,Yaskov2015}
relate the empirical spectral distribution (ESD) of $S_p$ to concentration
of quadratic forms. The usual sufficient condition, for every bounded deterministic matrix
sequence $(A_p)$, is
\begin{equation}\label{eq:yaskov-A}
 \frac{x_p^TA_px_p-\Tr A_p}{p}\xrightarrowp0,
 \qquad \sup_p\|A_p\|<\infty.
\end{equation}
At $A_p=I_p$ this forces $R_p=\|x_p\|^2/p\to1$ in probability.
A nonconcentrating radius therefore requires a different condition.
Subtracting the radial mode gives
\begin{equation}\label{eq:RQC-intro}
 Q_p(A_p):=\frac{x_p^TA_px_p}{p}-R_p\frac{\Tr A_p}{p}
 \xrightarrowp0.
\end{equation}
We call this the radial quadratic-form condition (RQC). For a proportional
projection it requires its normalized energy to agree with the parent
radius, while allowing that radius to fluctuate. If the limiting radius
has no atom at zero, RQC is equivalent to \eqref{eq:yaskov-A} for
$x_p/\sqrt{R_p}$; see \cref{prop:self-normalized}.

When $R_p\Rightarrow\nu$, RQC yields the radial MP law $\mu_{c,\nu}$,
whose Stieltjes transform $s(z)=\int(\lambda-z)^{-1}\mu_{c,\nu}(\dd\lambda)$
satisfies, on $\C_+=\{z\in\C:\Imc z>0\}$,
\begin{equation}\label{eq:RMP-intro}
 \frac1{s(z)}+z=\int_0^\infty\frac{r}{1+crs(z)}\,\nu(\dd r),
 \qquad z\in\C_+.
\end{equation}
The equation is familiar from elliptical and weighted covariance models
\cite{ElKaroui2009,PajorPastur2009,SilversteinBai1995}.
Its qualitative sufficiency under an endogenous radius is already covered
by Yaskov's random-profile universality theorem
\cite[Theorem~2.2]{Yaskov2014}, with profile $R_pI_p$.
Our direct moment-free proof supplies the fixed-point law and the
truncation estimates needed below.

\subsection{The inverse question}

Does radial MP spectral convergence force RQC?
A single full ESD does not: the block model in \cref{ex:block} has
the correct radial MP law and fails RQC. The relevant spectral data are
the covariances of projected samples. For a rank-$q$ orthogonal projection
$P$, let $C_P$ have orthonormal rows spanning its range and put
\[
 S_P=\frac1N\sum_{j=1}^N(C_Px_j)(C_Px_j)^T.
\]
Our principal result, \cref{thm:necessity-free}, states that, whenever
$R_p\Rightarrow\nu$, for any fixed $\alpha\in(0,1)$,
\[
 \mathrm{(RQC)}
 \quad\Longleftrightarrow\quad
 \left\{\begin{gathered}
 \ESD(S_{P_p})\Rightarrow\mu_{c\alpha,\nu}\\
 \text{for every deterministic rank-}\lfloor\alpha p\rfloor
 \text{ sequence }P_p.
 \end{gathered}\right.
\]
All orientations of this one rank are required; neither other ranks nor
the full spectral limit need be assumed. The radius law $\nu$ may have
arbitrary tails. In particular no logarithmic moment is needed.
The radius and direction may be functions of the same randomness,
and the conditional direction need not be isotropic.

One rank suffices because the convex hull of its projections contains
a neighborhood of $(q/p)I$ in the affine space of matrices with trace $q$.
Strict scalar concavity on this neighborhood controls all centered
quadratic forms. We use this standard fantope geometry to prove
dimension-free, one-sided rigidity bounds for arbitrary random positive
semidefinite operators (\cref{thm:bounded-energy}).
The converse extracts this energy comparison from spectra.

\subsection{Determinant comparison and heavy tails}

The expected logarithmic determinant is concave in the common law of
i.i.d.\ columns (\cref{lem:log-law-concavity}). Thus angular symmetrization
increases it. Combining this comparison for a projection and its
complementary energy with strict Jensen inequalities suggests a converse,
but the separate logarithmic expectations may be infinite.

\Cref{lem:frame-trimming}
resolves this problem before taking expectations. Delete the $k$ columns
above a radial threshold $K$, and remove the largest $k$ eigenvalues
from the original spectral logarithm. Interlacing gives one comparison.
For the reverse comparison we use a Gaussian reference carrying the
specified radial law. A fixed cutoff $M$, independent of $K$, isolates
a small collection of high-radius columns. Their direction matrix is
tall, with singular values bounded below with high probability.
A normalized Schur complement cancels the factors $1+r_j$ contributed
by the deleted radii. The remaining error is bounded by a constant times
$k$, irrespective of the radial magnitudes.

After normalization, this gives an error of order $\nu((K,\infty))$.
All logarithmic expectations in the ensuing Jensen argument concern
bounded, radially deleted variables. Letting $K\to\infty$ therefore
identifies the complementary energies for every probability law $\nu$.
\Cref{thm:finite-spectral-inverse} strengthens the argument to a
finite-dimensional inverse bound using capped logarithmic tests of expected
spectra. Its reference carries the actual radius law.
A universal bounded scalar substitute for the logarithm cannot retain
column-law concavity; \cref{prop:bounded-spectral-obstruction} gives
a counterexample even under exact conditional isotropy.

\subsection{Recovering the radius}

The characterization assumes that the actual radius converges to the
specified law. Spectra alone cannot remove this hypothesis.
In \cref{ex:spike}, energy of order one in a single deterministic direction
changes the limiting radius but leaves every proportional projected ESD
unchanged. We exclude this phenomenon by requiring that subspaces of
vanishing relative rank carry vanishing normalized energy in probability;
the precise condition is \eqref{eq:no-small-subspace-energy}.

Our second main result, \cref{thm:radius-identification}, shows that under
this condition, any common one-rank projected spectral limit $\mu$ with
finite second moment determines a unique radius law $\eta$:
\[
 R_p\Rightarrow\eta,\qquad \mathrm{(RQC)},\qquad \mu=\mu_{c\alpha,\eta}.
\]
In particular, a prescribed limit $\mu_{c\alpha,\nu}$ identifies $\nu$
when $\int r^2\,\nu(\dd r)<\infty$. No moments of the actual
finite-dimensional radii are assumed. The proof first recovers population
isotropy and projected second moments under uniform integrability, then
uses truncated samples to establish the required moment bound for every
subsequential radius law.
Radius identification for infinite-second-moment targets remains open;
the characterization with a known radial limit is already moment-free.

\subsection{Context and organization}

Yaskov's projected MP characterization \cite{Yaskov2015b} concerns an
isotropic vector and a concentrating radius. Here the radial mode can
remain random, one rank fraction suffices, and arbitrary radial tails
are allowed. \Cref{thm:necessity} provides a complementary inverse
under approximate conditional isotropy.
The relation to random-profile universality and weighted covariance
theory is discussed in \cref{sec:related}.

The tensor companion \cite{Xie2026} proves the tensor-specific quadratic-form
estimates, lognormal radius limit, critical ESD, and analysis of that law.
Its sufficient radial principle assumes uniform $L^2$ control and uniform
integrability. We use those tensor conclusions only as an example; the
present paper concerns model-independent characterization, determinant
comparison, energy rigidity, and radius recovery.

The main statements are in \cref{sec:statements}.
\Cref{sec:quadratic,sec:sufficient} give the RQC reductions and sufficiency.
\Cref{sec:rigidity,sec:converse} develop projection rigidity and the
moment-free converse; \cref{sec:finite-inverse,sec:radius} treat its
finite-dimensional form and radius identification.
The conditional alternative and examples follow in
\cref{sec:necessity,sec:examples}.
\Cref{sec:population-statement,sec:extremes,sec:related} give the
population extension, limitations, and comparison with earlier work.
The population proof, a multiple-aspect bound, and the bounded-concavity
obstruction are collected in the appendices.

\section{Main statements}
\label{sec:statements}

\subsection{Setting and notation}

Throughout, $p=p_n\to\infty$ and $N=N_n$ are integers with
$c_n:=p/N\to c\in(0,\infty)$; all limits are as $n\to\infty$. For each $n$,
$x=x_p$ is a random vector in $\R^p$, $x_1,\dots,x_N$ are independent copies
of it, $R_j=\|x_j\|^2/p$, and
\[
  S=\frac1N\sum_{j=1}^Nx_jx_j^{\mathsf T},\qquad
  G(z)=(S-zI)^{-1},\qquad s_n(z)=\frac1p\Tr G(z),\qquad z\in\C_+.
\]
For a probability measure $\mu$ on $\R$ we write
$s_\mu(z)=\int(t-z)^{-1}\mu(\dd t)$; this maps $\C_+$ into $\C_+$. $\ESD(H)$
is the empirical spectral distribution of a Hermitian matrix $H$, and
$\dK$ the Kolmogorov distance between distribution functions. Weak convergence
of random measures ``in probability'' means convergence in probability of the
L\'evy distance to the (deterministic) limit. $\|A\|$ is the operator norm and
$\|A\|_1$ the trace norm. For a matrix $A\in\C^{p\times p}$ put
\begin{equation}
  Q_p(A)=\frac{x^{\mathsf T}Ax}{p}-R_p\frac{\Tr A}{p}.
  \label{eq:Qdef}
\end{equation}

\begin{definition}[radial quadratic-form condition]
\label{def:RQC}
The sequence $(x_p)$ satisfies (RQC) if for every deterministic sequence of
orthogonal projections $P_p\in\R^{p\times p}$,
\begin{equation}
  Q_p(P_p)\xrightarrowp0.
  \label{eq:RQC}
\end{equation}
\end{definition}

If $(R_p)$ is tight, as it is whenever $R_p\Rightarrow\nu$,
\cref{prop:rqc-equivalences} shows that \eqref{eq:RQC} is equivalent to the
same statement for every bounded sequence of complex matrices, and to the
uniform statement $\sup_{\|A\|\le1}\Pp(|Q_p(A)|>\eps)\to0$ for every $\eps>0$;
the last two are equivalent to each other without tightness, by
\cref{lem:uniform}(b).
It is the version with projections that is easiest to verify or refute in
examples.

The sample vectors in this paper are real. Complex test matrices are allowed
in $Q_p$; this does not assert the complex-sample analogue, whose Haar
moment identities would require the corresponding unitary formulas.

For reference, the main notation is collected below; local definitions
and hypotheses are given in the indicated sections. For a rank-$q$
projection, $\rho_P=\|Px_p\|^2/q$ denotes its normalized energy.
\begin{center}
\small
\begin{tabular}{@{}l>{\raggedright\arraybackslash}p{0.64\textwidth}@{}}
\toprule
Symbol & Meaning and location\\
\midrule
$R_p,\ Q_p(A)$ & Parent radius and centered quadratic form,
\eqref{eq:Qdef}.\\
$\mu_{c,\nu}$ & Radial MP law, \cref{prop:limit-law}.\\
$\beta=q/p,\ b$ & Rank fraction and $\min(\beta,1-\beta)$,
\cref{sec:rigidity}.\\
$\psi(r)=r/(1+r)$ & Bounded energy test.\\
$\Delta_\psi,\ D_\psi$ & One-sided and absolute projected-energy defects,
\cref{thm:bounded-energy}.\\
$\mathcal W$ & Weighted quadratic-form defect,
\cref{thm:bounded-energy}.\\
$M$ & Random positive semidefinite operator in \cref{sec:rigidity};
elsewhere a locally specified radial cutoff.\\
$F_{q,N,t}(\sigma)$ & Expected normalized log-determinant for i.i.d.\
columns with law $\sigma$, \cref{lem:log-law-concavity}.\\
$\mathcal J_{d,\eta}(t)$ & Limiting radial logarithmic transform,
\cref{lem:log-variational}.\\
$\mathcal I$ & High-radius column indices,
\cref{lem:frame-trimming}.\\
$\underline\varsigma,\overline\varsigma$ & Lower and upper normalized
frame bounds, \cref{lem:frame-trimming}.\\
$V_p$ & Complementary averaged energy,
\eqref{eq:log-complement-energy}.\\
$\mathcal T_{\delta,t},\mathcal T_{\delta,L}$ & Trimmed logarithmic
transform and its capped version, \cref{sec:converse,sec:finite-inverse}.\\
$\Psi(\eta)$ & Finite spherical log-determinant functional;
$\widetilde\Psi(B)$ is its fantope lift, \cref{sec:finite-inverse}.\\
$m_1,\ \Theta$ & First limiting spectral moment and spectral cap
in \cref{sec:radius}.\\
$\mathcal H_p,\rho_p$ & Auxiliary field and radius in a radial reduction,
\cref{def:ACI}.\\
$T_p,\ H$ & Population covariance and its limiting ESD,
\cref{thm:population}.\\
\bottomrule
\end{tabular}
\end{center}

\subsection{The limit law}

\begin{proposition}
\label{prop:limit-law}
Let $\nu$ be a probability measure on $[0,\infty)$ and $c>0$. For every
$z\in\C_+$ the equation
\begin{equation}
  \frac1s+z=\int_0^\infty\frac{r}{1+crs}\,\nu(\dd r)
  \label{eq:RMP}
\end{equation}
has at most one solution $s\in\C_+$; the integral converges absolutely for
every $s\in\C_+$. There is exactly one probability measure $\mu_{c,\nu}$ on
$[0,\infty)$ whose Stieltjes transform solves \eqref{eq:RMP} for all
$z\in\C_+$. It is the free compound Poisson law with rate $1/c$ and jump law
$\law(cR)$, $R\sim\nu$; $\mu_{c,\delta_1}$ is the MP law with ratio $c$;
$\mu_{c,\nu}(\{0\})=\big(1-\nu((0,\infty))/c\big)_+$; and
$\nu\mapsto\mu_{c,\nu}$ is injective.
\end{proposition}

\subsection{The one-rank characterization}

For an orthogonal projection $P$ of rank $q$ let $C_P\in\R^{q\times p}$ have
orthonormal rows spanning $\ran P$, so $C_P^{\mathsf T}C_P=P$ and
$C_PC_P^{\mathsf T}=I_q$, and let
\[
  S_P=\frac1N\sum_{j=1}^N(C_Px_j)(C_Px_j)^{\mathsf T}\in\R^{q\times q}.
\]

\begin{theorem}[moment-free one-rank characterization]
\label{thm:necessity-free}
Assume $R_p\Rightarrow\nu$, where $\nu$ is an arbitrary probability measure
on $[0,\infty)$.
Fix any $\alpha\in(0,1)$ and put $q_p=\lfloor\alpha p\rfloor$.
Then the following are equivalent:
\begin{enumerate}[label=(\roman*)]
\item $(x_p)$ satisfies (RQC).
\item For every deterministic rank-$q_p$ projection sequence $P_p$
(all orientations at this one rank),
\[
 \ESD(S_{P_p})\Rightarrow\mu_{c\alpha,\nu}\quad\text{in probability}.
\]
\end{enumerate}
No moments or isotropy assumptions are imposed on the vectors, and $\nu$
may have arbitrary tails. If $\{R_p^2\}$ is uniformly integrable, (ii) also implies
\[
 \sup_{\|A\|\le1}\E|Q_p(A)|^2\to0.
\]
\end{theorem}

The proof is in \cref{sec:converse}. It uses a Gaussian determinant comparison
between column deletion and spectral trimming, angular symmetrization, and a strict
concavity argument for complementary energies.

In particular, the theorem applies when $\int\log(1+r)\nu(\dd r)=\infty$.
For example, let $Y\ge0$ have survival function
$\Pp(Y>u)=(1+u)^{-1/2}$ for $u\ge0$, and put $R=e^Y$.
Then $\E\log(1+R)=\infty$ and $\E R^\delta=\infty$ for every $\delta>0$.

\subsection{Radius identification}

The preceding theorem assumes the limiting radius law. To recover it from
spectra, we exclude energy concentrated in small deterministic subspaces:
\begin{equation}\label{eq:no-small-subspace-energy}
 \lim_{\delta\downarrow0}\limsup_{p\to\infty}
 \sup_{\rank P\le\delta p}
 \Pp\left(\frac{x_p^TPx_p}{p}>\varepsilon\right)=0
 \quad\text{for every }\varepsilon>0.
\end{equation}
Under tightness of $R_p$, (RQC) implies this condition because
$p^{-1}x_p^TPx_p=Q_p(P)+R_p\rank(P)/p$.

\begin{theorem}[Radius identification from projected spectra]
\label{thm:radius-identification}
Assume \eqref{eq:no-small-subspace-energy}, fix $\alpha\in(0,1)$, and put
$q_p=\lfloor\alpha p\rfloor$ and $d=c\alpha$.
Suppose that for every deterministic sequence $P_p$ of orthogonal projections
with $\rank P_p=q_p$, $\ESD(S_{P_p})$ converges weakly in probability to
the same deterministic law $\mu$.
\begin{enumerate}[label=(\alph*)]
\item If $\int\lambda^2\,\mu(\dd\lambda)<\infty$, then there is a unique
probability law $\eta$ such that
\[
 R_p\Rightarrow\eta,\qquad \mathrm{(RQC)}\text{ holds},\qquad
 \mu=\mu_{d,\eta}.
\]
In particular, if the specified limit is $\mu_{d,\nu}$ and
$\int r^2\,\nu(\dd r)<\infty$, then $R_p\Rightarrow\nu$.
No moment or convergence assumption on the actual $R_p$ is imposed.
\item If $\{R_p^2\}$ is uniformly integrable, the same conclusions hold
without assuming a second moment of $\mu$, and RQC holds uniformly in
$L^2$ over contractions.
\end{enumerate}
\end{theorem}

The proof is in \cref{sec:radius}. The condition rules out the fixed-direction
energy in \cref{ex:spike}; it imposes no finite-dimensional moment bound.
It is also necessary for the stated conclusion: radial convergence gives
tightness, and RQC then implies \eqref{eq:no-small-subspace-energy}.
Thus, under a common deterministic one-rank projected spectral limit with
finite second moment, \eqref{eq:no-small-subspace-energy} is necessary and
sufficient for radial convergence and RQC. Radius identification for
arbitrary target tails remains open.

\subsection{The sufficient theorem}

\begin{theorem}
\label{thm:sufficient}
Assume $c_n\to c\in(0,\infty)$, $R_p\Rightarrow\nu$ for a probability measure
$\nu$ on $[0,\infty)$, and (RQC). Then there is a deterministic probability
measure $\mu$ on $[0,\infty)$ such that $\ESD(S)\Rightarrow\mu$ weakly in
probability, and for every $z\in\C_+$ its Stieltjes transform is the unique
solution in $\C_+$ of \eqref{eq:RMP}; thus $\mu=\mu_{c,\nu}$. If the matrices
for all $n$ are realized on one probability space and
$\sum_ne^{-ap_n}<\infty$ for every $a>0$, the convergence is almost sure.
\end{theorem}

No moment of $R_p$, no uniform integrability, no isotropy and no assumption on
the joint law of $R_p$ and $x/\|x\|$ is made. The theorem also holds with
$R_p$ replaced by an arbitrary nonnegative random variable $\rho_p$ in
\eqref{eq:Qdef} and in the hypothesis $\rho_p\Rightarrow\nu$, provided the
condition is assumed in its uniform form over all contractions: the reduction
from projections uses $x^{\mathsf T}Px\le pR_p$ and is not available for a
general $\rho_p$ (\cref{rem:general-radius}). This extension applies to
projected samples with the parent radius.

\section{Quadratic forms: what the radial condition says}
\label{sec:quadratic}

We record the projection reduction and uniformity needed to apply RQC to
independent resolvents. The deterministic-sequence selection argument is
standard; radial truncation allows the reduction without moments.

\subsection{Deterministic sequences are automatically uniform}

For $\eps>0$ and a class $\mathcal A_p\subset\C^{p\times p}$ define
\[
  \phi_p(\eps;\mathcal A_p)=\sup_{A\in\mathcal A_p}\Pp\big(|Q_p(A)|>\eps\big).
\]

\begin{lemma}
\label{lem:uniform}
Let $\mathcal A_p\subset\C^{p\times p}$ be arbitrary nonempty sets.
\begin{enumerate}[label=(\alph*)]
\item For fixed $p$ and $\eps$, $A\mapsto\Pp(|Q_p(A)|>\eps)$ is Borel on
$\C^{p\times p}$.
\item $Q_p(A_p)\to0$ in probability for every deterministic sequence with
$A_p\in\mathcal A_p$ if and only if $\phi_p(\eps;\mathcal A_p)\to0$ for every
$\eps>0$.
\item Suppose $\mathcal A_p$ is Borel and $\phi_p(\eps;\mathcal A_p)\to0$ for
every $\eps>0$, and let $B_p$ be a random matrix, independent of $x_p$, with
$B_p\in\mathcal A_p$ almost surely. Then $\Pp(|Q_p(B_p)|>\eps)\le\phi_p(\eps;\mathcal A_p)$; in
particular $Q_p(B_p)\to0$ in probability.
\end{enumerate}
\end{lemma}

\begin{proof}
(a) The map $(A,\omega)\mapsto Q_p(A)(\omega)$ is continuous in $A$ for each
$\omega$ and measurable in $\omega$ for each $A$, hence jointly measurable.
The set $\{|Q_p|>\eps\}$ is therefore product-measurable, and the
probability of its $A$-section is a Borel function of $A$ by the
measurability-of-sections part of Tonelli's theorem; no measure on
$\C^{p\times p}$ is involved.

(b) If the uniform statement fails there are $\eps,\delta>0$ and
$p_k\uparrow\infty$ with $\phi_{p_k}(\eps;\mathcal A_{p_k})>\delta$. Choose
$A_{p_k}\in\mathcal A_{p_k}$ with $\Pp(|Q_{p_k}(A_{p_k})|>\eps)>\delta$ and
complete the sequence with arbitrary elements of $\mathcal A_p$ for the other
$p$. Then $Q_p(A_p)\not\to0$ in probability. The converse is trivial.

(c) Independence makes the joint law of $(x_p,B_p)$ a product measure, so by
Fubini
$\Pp(|Q_p(B_p)|>\eps)=\E\,\psi(B_p)$ with
$\psi(A)=\Pp(|Q_p(A)|>\eps)$, which is measurable by (a) and bounded by
$\phi_p(\eps;\mathcal A_p)$ on $\mathcal A_p$.
\end{proof}

The same selection argument applies to any deterministic family
$f_p(P)$: convergence to zero along every sequence $P_p$ is equivalent
to $\sup_P|f_p(P)|\to0$. We use it for expectations of bounded spectral
tests and for tail probabilities. A version with $L^2$ bounds is identical.

\subsection{From projections to all bounded matrices}

\begin{proposition}
\label{prop:rqc-equivalences}
Assume $(R_p)$ is tight. The following are equivalent.
\begin{enumerate}[label=(\roman*)]
\item (RQC): $Q_p(P_p)\to0$ in probability for every deterministic sequence of
orthogonal projections.
\item $Q_p(A_p)\to0$ in probability for every deterministic sequence of complex
matrices with $\sup_p\|A_p\|<\infty$.
\item $\sup_{\|A\|\le1}\Pp(|Q_p(A)|>\eps)\to0$ for every $\eps>0$.
\end{enumerate}
\end{proposition}

\begin{proof}
(iii)$\Rightarrow$(ii)$\Rightarrow$(i) are trivial. Assume (i).

\emph{Step 1: uniformity over projections.} By \cref{lem:uniform}(b) with
$\mathcal A_p$ the set $\Pi_p$ of orthogonal projections,
$\sup_{P\in\Pi_p}\Pp(|Q_p(P)|>\eps)\to0$ for every $\eps$.

\emph{Step 2: truncated $L^1$ bound.} For $P\in\Pi_p$,
$0\le x^{\mathsf T}Px/p\le R_p$ and $0\le R_p\Tr P/p\le R_p$, so
$|Q_p(P)|\le R_p$. Hence for every $M,\eps>0$,
\[
  \eta_p(M):=\sup_{P\in\Pi_p}\E\big[|Q_p(P)|\1_{\{R_p\le M\}}\big]
  \le\eps+M\sup_{P\in\Pi_p}\Pp(|Q_p(P)|>\eps),
\]
and therefore $\eta_p(M)\to0$ for every fixed $M$.

\emph{Step 3: positive semidefinite contractions.} Let $D$ be real symmetric
positive semidefinite with $\|D\|\le1$ and eigenvalues
$1\ge\lambda_1\ge\dots\ge\lambda_p\ge0$, and let $E_k$ be the orthogonal
projection onto the span of eigenvectors of the first $k$ eigenvalues (any
choice in case of ties). With $\lambda_{p+1}=0$ and
$w_k=\lambda_k-\lambda_{k+1}\ge0$ one has $\sum_kw_k=\lambda_1\le1$ and
\[
  D=\sum_{k=1}^pw_kE_k .
\]
Since $A\mapsto Q_p(A)$ is linear, $Q_p(D)=\sum_kw_kQ_p(E_k)$ and
\[
  \E\big[|Q_p(D)|\1_{\{R_p\le M\}}\big]\le\sum_kw_k\,\E\big[|Q_p(E_k)|\1_{\{R_p\le M\}}\big]
  \le\eta_p(M).
\]
The projections $E_k$ depend on $D$, but the bound $\eta_p(M)$ does not. By
Markov's inequality,
\[
  \sup_{D}\Pp(|Q_p(D)|>\eps)\le\frac{\eta_p(M)}{\eps}+\Pp(R_p>M),
\]
the supremum over all positive semidefinite contractions. Let $p\to\infty$,
then $M\to\infty$, using tightness of $R_p$.

\emph{Step 4: general matrices.} A real symmetric contraction is $A_+-A_-$
with $A_\pm$ positive semidefinite contractions, and $Q_p$ is linear. For
complex $A$ with $\|A\|\le1$ write $A=H_1+iH_2$ with
$H_1=(A+A^*)/2$, $H_2=(A-A^*)/(2i)$ Hermitian, $\|H_k\|\le1$. For real $x$,
$x^{\mathsf T}H_kx=x^{\mathsf T}(\Rec H_k)x$, where $\Rec H_k=(H_k+\overline{H_k})/2$
is real symmetric with $\|\Rec H_k\|\le1$, and $\Tr H_k=\Tr\Rec H_k$ because a
Hermitian matrix has a real diagonal. Hence
$Q_p(A)=Q_p(\Rec H_1)+iQ_p(\Rec H_2)$, and
$\sup_{\|A\|\le1}\Pp(|Q_p(A)|>\eps)\le2\sup\Pp(|Q_p(A')|>\eps/2)$ with the
supremum over real symmetric contractions, which is at most
$4\sup_D\Pp(|Q_p(D)|>\eps/4)$ over positive semidefinite contractions. This
is (iii).
\end{proof}

The layer-cake step is the one used by Yaskov \cite[Prop.~2.1]{Yaskov2015b} to
pass from projections to matrices; there it is combined with isotropy and
Vitali's theorem to obtain an $L^1$ bound. The truncation at $R_p\le M$
replaces both.

\subsection{The self-normalized vector}

\begin{proposition}
\label{prop:self-normalized}
Assume $R_p\Rightarrow\nu$ with $\nu(\{0\})=0$, and let $u_p=x_p/\sqrt{R_p}$
on $\{R_p>0\}$ (and $u_p=0$ otherwise), so $\|u_p\|^2=p$ on $\{R_p>0\}$. Then
(RQC) holds for $(x_p)$ if and only if $u_p$ satisfies \eqref{eq:yaskov-A}.
\end{proposition}

\begin{proof}
Since $R_p$ is tight, \cref{prop:rqc-equivalences} lets us read (RQC) as
$Q_p(A_p)\to0$ for every bounded deterministic sequence, which is the form
comparable with \eqref{eq:yaskov-A}. On $\{R_p>0\}$,
$Q_p(A)=R_p\,(u_p^{\mathsf T}Au_p/p-\Tr A/p)$, and
$\Pp(R_p=0)\to\nu(\{0\})=0$. Since
$\limsup_p\Pp(R_p<\delta)\le\nu([0,\delta])\to0$ as $\delta\downarrow0$, and
$R_p$ is tight, $R_p$ is bounded above and below in probability, so the two
convergences are equivalent.
\end{proof}

Under the hypotheses of \cref{prop:self-normalized}, the normalized vector
$u_p$ satisfies the usual MP criterion. The radial spectral limit depends
only on the limiting law of $R_p$, even when $R_p$ and $u_p$ are dependent.
Results requiring independent weights do not apply to this coupling;
examples are given in \cref{sec:examples}.

\section{The sufficient theorem}
\label{sec:sufficient}

We give a direct radial proof of \cref{prop:limit-law,thm:sufficient}.
The qualitative replacement conclusion also follows from
\cite[Theorem~2.2]{Yaskov2014}. Here bounded resolvent summands and radial
column deletion supply the moment-free fixed-point argument.

\subsection{Deterministic bounds}

Write $G_H(z)=(H-zI)^{-1}$ for Hermitian $H$ and $z\in\C_+$.

\begin{lemma}
\label{lem:denominator}
Let $z=E+i\eta\in\C_+$, $t_1,\dots,t_k\ge0$, $w_1,\dots,w_k\ge0$, and
$u=\sum_jw_j/(t_j-z)$. Then $|1+u|\ge\eta/|z|$, and
\begin{equation}
  \Big|\frac{u}{1+u}\Big|=\Big|1-\frac1{1+u}\Big|\le1+\frac{|z|}{\eta}.
  \label{eq:bounded-summand}
\end{equation}
\end{lemma}

\begin{proof}
$\Imc\frac{z}{t-z}=\frac{t\eta}{(t-E)^2+\eta^2}\ge0$ for $t\ge0$, so
$\Imc[z(1+u)]\ge\eta$ and $|z||1+u|\ge\eta$. The second bound follows.
\end{proof}

\begin{lemma}
\label{lem:rank-one}
Let $H,H'$ be Hermitian $p\times p$ with $\rank(H-H')\le k$. Then
$\dK(\ESD(H),\ESD(H'))\le k/p$, and for $z=E+i\eta\in\C_+$,
$|\frac1p\Tr G_H(z)-\frac1p\Tr G_{H'}(z)|\le\pi k/(p\eta)$.
\end{lemma}

\begin{proof}
The first statement is the rank inequality
\cite[Theorem~A.43]{BaiSilverstein2010}. For the second, integrate by parts:
$|\int\frac{\dd(F_H-F_{H'})(t)}{t-z}|=|\int\frac{(F_H-F_{H'})(t)}{(t-z)^2}\dd t|
\le\frac kp\int_\R\frac{\dd t}{|t-z|^2}=\frac{\pi k}{p\eta}$.
\end{proof}

\subsection{The limit law}

\begin{proof}[Proof of \cref{prop:limit-law}]
\emph{Convergence of the integral.} For $s\in\C_+$ and $r\ge0$,
$|1+crs|\ge\Imc(1+crs)=cr\Imc s$; hence
\begin{equation}
  \Big|\frac{r}{1+crs}\Big|\le\frac1{c\,\Imc s}\qquad(r\ge0,\ s\in\C_+),
  \label{eq:integrand-bounded}
\end{equation}
so the integrand in \eqref{eq:RMP} is bounded and continuous in $r$ on
$[0,\infty)$, and \eqref{eq:RMP} makes sense for every probability measure
$\nu$.

\emph{Uniqueness.} Let $s=u+iy\in\C_+$ solve \eqref{eq:RMP} at $z=E+i\eta$.
Taking imaginary parts,
$-y/|s|^2+\eta=-cy\int r^2|1+crs|^{-2}\nu(\dd r)$, that is
\begin{equation}
  c|s|^2\int\frac{r^2}{|1+crs|^2}\nu(\dd r)=1-\frac{\eta|s|^2}{y}<1.
  \label{eq:stability}
\end{equation}
(The integral is finite: $r^2/|1+crs|^2\le1/(c\Imc s)^2$ by
\eqref{eq:integrand-bounded}.) If $s_1\ne s_2$ are two solutions,
subtracting the equations and dividing by $s_1-s_2$ gives
$\frac1{s_1s_2}=c\int\frac{r^2\,\nu(\dd r)}{(1+crs_1)(1+crs_2)}$, and
Cauchy--Schwarz bounds the modulus of the right side by
$|s_1|^{-1}|s_2|^{-1}$ times the geometric mean of the two quantities in
\eqref{eq:stability}, each less than one; contradiction.

\emph{Existence.} Let $R\sim\nu$ and $g\sim N(0,I_p)$ be independent and
$x=\sqrt R\,g$. Then $R_p=R\|g\|^2/p\Rightarrow\nu$, and for a projection $P$
of rank $q$,
\[
  Q_p(P)=R\Big(\frac{g^{\mathsf T}Pg-q}{p}-\frac qp\Big(\frac{\|g\|^2}{p}-1\Big)\Big).
\]
Both terms in the bracket have variance at most $2/p$, and $R$ is tight, so
(RQC) holds. \Cref{thm:sufficient}, whose proof below uses only the uniqueness
part of this proposition, yields a probability measure on $[0,\infty)$ whose
transform solves \eqref{eq:RMP} on $\C_+$; it is unique because a Stieltjes
transform determines the measure. This is $\mu_{c,\nu}$.

\emph{Structure.} With $\mathcal G(z)=-s(z)$ the Cauchy transform,
\eqref{eq:RMP} reads $z=1/\mathcal G+\mathcal R(\mathcal G)$ with
$\mathcal R(w)=\int r(1-crw)^{-1}\nu(\dd r)$, which is the $R$-transform of the
free compound Poisson law with rate $1/c$ and jump law $\law(cR)$; see
\cite[\S11.4]{Xie2026} for this identification on the half-plane, valid without
moments \cite{BercoviciVoiculescu1993}. For $\nu=\delta_1$ the equation is
$1/s+z=1/(1+cs)$, the MP equation.

\emph{The atom at zero.} Write $\rho=\nu((0,\infty))$ and $a=\mu_{c,\nu}(\{0\})$.
The transform $s$ of $\mu_{c,\nu}$ extends analytically to
$\C\setminus[0,\infty)$ by the same formula, with
$s(-\sigma)=\int(t+\sigma)^{-1}\mu_{c,\nu}(\dd t)\in(0,\sigma^{-1}]$ for
$\sigma>0$. Let $\Omega$ be the connected component of
$\{z\in\C\setminus[0,\infty):\Rec s(z)>0\}$ containing $(-\infty,0)$; being
open it meets $\C_+$ in a nonempty open set. On $\Omega$, $\Rec(1+crs)>0$ gives
$|r/(1+crs)|\le1/(c\Rec s)$ locally uniformly in $z$ and uniformly in $r\ge0$,
so the integral is analytic there, and $s\ne0$ makes $1/s+z$ analytic. The two
sides agree on $\Omega\cap\C_+$, hence on $\Omega$ by the identity theorem,
hence at $z=-\sigma$.
Multiplying \eqref{eq:RMP} by $s=s(-\sigma)$,
\begin{equation}
  1-\sigma s(-\sigma)=\frac1c\int_0^\infty\frac{crs(-\sigma)}{1+crs(-\sigma)}\,\nu(\dd r).
  \label{eq:atom-identity}
\end{equation}
As $\sigma\downarrow0$ the left side tends to $1-a$, because
$\sigma s(-\sigma)=\int\frac{\sigma}{t+\sigma}\mu_{c,\nu}(\dd t)\to\mu_{c,\nu}(\{0\})$
by dominated convergence. Also $\sigma\mapsto s(-\sigma)$ is decreasing, so
$s(-\sigma)$ increases as $\sigma\downarrow0$. If $a>0$ then
$s(-\sigma)\ge a/\sigma\to\infty$, so the integrand increases to
$\1_{\{r>0\}}$ and monotone convergence gives $1-a=\rho/c$. If $a=0$ then the
right side of \eqref{eq:atom-identity} tends to $1$, while its integrand is at
most $\1_{\{r>0\}}$, so $\rho/c\ge1$. In both cases
$a=(1-\rho/c)_+$.

\emph{Injectivity.} Set $\Phi_\nu(w)=\int r(1+crw)^{-1}\nu(\dd r)$, analytic
on $\C_+$ by \eqref{eq:integrand-bounded}. If $\mu_{c,\nu}=\mu_{c,\nu'}$ then
$\Phi_\nu=\Phi_{\nu'}$ on $\{s(z):z\in\C_+\}$, a nonempty open subset of
$\C_+$ ($s$ is a nonconstant analytic map, $s(z)\sim-1/z$ at infinity), hence
on $\C_+$. With $v=-1/(cw)\in\C_+$ for $w\in\C_+$,
$\frac{r}{1+crw}=\frac{vr}{v-r}=-v-\frac{v^2}{r-v}$, so
$\Phi_\nu(w)=-v-v^2s_\nu(v)$ determines $s_\nu$ on $\C_+$, hence $\nu$.
\end{proof}

\subsection{\texorpdfstring{Proof of \cref{thm:sufficient}}{Proof of the sufficient theorem}}

By \cref{prop:rqc-equivalences}, (RQC) gives
\begin{equation}
  \phi_p(\eps):=\sup_{\|A\|\le1}\Pp(|Q_p(A)|>\eps)\longrightarrow0\qquad(\eps>0).
  \label{eq:phi}
\end{equation}
Write $J_M=\{j\le N:R_j>M\}$ and $\hat\nu_N=\frac1N\sum_j\delta_{R_j}$.

\emph{Step 0: the radii.} For bounded continuous $f$,
\[
 \Var\left(\frac1N\sum_jf(R_j)\right)\le\|f\|_\infty^2/N,\qquad
 \E f(R_p)\to\int f\dd\nu,
\]
so $\hat\nu_N\Rightarrow\nu$ in probability (countable convergence-determining
family). Likewise $|J_M|/N-\Pp(R_p>M)\to0$ in probability, and
$\limsup_p\Pp(R_p>M)\le\nu([M,\infty))$ by the Portmanteau theorem. Hence for
every $M$ and $\delta>0$,
\begin{equation}
  \Pp\Big(\frac{|J_M|}{N}>\nu([M,\infty))+\delta\Big)\longrightarrow0.
  \label{eq:JM}
\end{equation}

\emph{Step 1: tightness without moments.} Fix $M,L>0$ and let
$S^{(M)}=\frac1N\sum_{j\notin J_M}x_jx_j^{\mathsf T}$.
The failure of an untruncated termwise estimate is illustrated in
\cref{rem:truncation-needed}. Here
$\rank(S-S^{(M)})\le|J_M|$ and
\[
 \frac1p\Tr S^{(M)}=\frac1N\sum_{j\notin J_M}R_j\le M,
\]
so by \cref{lem:rank-one} and Markov's inequality applied to $\ESD(S^{(M)})$,
\begin{equation}
  \ESD(S)([L,\infty))\le\frac ML+\frac{|J_M|}{p}=\frac ML+\frac{N}{p}\frac{|J_M|}{N}.
  \label{eq:tight}
\end{equation}
Given $\delta>0$, choose $M$ with $\nu([M,\infty))<c\delta/8$, then $L$ with
$M/L<\delta/2$, and apply \eqref{eq:JM} with $\delta'=c\delta/8$: since
$N/p\to1/c$, the right side of \eqref{eq:tight} is at most
$\delta/2+(1/c+o(1))(c\delta/8+c\delta/8)<\delta$ with probability tending to
one. Hence $\Pp(\ESD(S)([L,\infty))>\delta)\to0$. Thus for
each rational $\delta$ there is $L_\delta$ with
\begin{equation}
  \Pp\big(\ESD(S)([L_\delta,\infty))>\delta\big)\longrightarrow0.
  \label{eq:tight-countable}
\end{equation}

\emph{Step 2: the approximate fixed point.} Fix $z=E+i\eta\in\C_+$. Let
$G_j=(S-\frac1Nx_jx_j^{\mathsf T}-z)^{-1}$, $s_n^{(j)}=\frac1p\Tr G_j$, and
$a_j=\frac1px_j^{\mathsf T}G_jx_j$. The identity $SG=I+zG$ and the
Sherman--Morrison formula $x_j^{\mathsf T}Gx_j=pa_j/(1+c_na_j)$ give the exact
identity
\begin{equation}
  1+zs_n(z)=\frac1N\sum_{j=1}^N\frac{a_j}{1+c_na_j}.
  \label{eq:exact}
\end{equation}
In an eigenbasis $(t_l,u_l)$ of $S-N^{-1}x_jx_j^{\mathsf T}$,
\[
 c_na_j=\frac1N\sum_l\frac{|u_l^*x_j|^2}{t_l-z},\qquad t_l\ge0.
\]
The same representation holds for $c_nrs_n^{(j)}$
and $c_nrs_n$ for $r\ge0$; all three are of the form in \cref{lem:denominator}.
Hence, with $K_z=(1+|z|/\eta)/\underline c$, $\underline c=\inf_nc_n>0$,
\begin{equation}
\begin{gathered}
  |1+c_na_j|^{-1},\ |1+c_nrs_n^{(j)}|^{-1},\ |1+c_nrs_n|^{-1}\le\frac{|z|}{\eta},\\
  \Big|\frac{a_j}{1+c_na_j}\Big|\le K_z,\qquad
  \Big|\frac{rw}{1+c_nrw}\Big|\le K_z
\end{gathered}
  \label{eq:denominators}
\end{equation}
for $w\in\{s_n,s_n^{(j)}\}$. For $a,b$ with nonvanishing denominators,
$\frac a{1+c_na}-\frac b{1+c_nb}=\frac{a-b}{(1+c_na)(1+c_nb)}$.

\emph{First replacement.} Put $b_j=R_js_n^{(j)}(z)$. The matrix $\eta G_j$ is
a function of $\{x_k\}_{k\ne j}$, hence independent of $x_j$, and
$\|\eta G_j\|\le1$; and $a_j-b_j=\eta^{-1}Q_p(\eta G_j)$ computed at $x_j$. By
\cref{lem:uniform}(c) and \eqref{eq:phi},
\begin{equation}
  \Pp(|a_j-b_j|>\eps)\le\phi_p(\eps\eta)\qquad\text{for every }j\text{ and }\eps>0.
  \label{eq:loo-prob}
\end{equation}
By \eqref{eq:denominators},
$d_j:=\big|\frac{a_j}{1+c_na_j}-\frac{b_j}{1+c_nb_j}\big|\le\min\{2K_z,(|z|/\eta)^2|a_j-b_j|\}$,
so for every $\eps>0$
\[
  \E\frac1N\sum_jd_j\le\Big(\frac{|z|}{\eta}\Big)^2\eps+2K_z\,\phi_p(\eps\eta)
  \longrightarrow\Big(\frac{|z|}{\eta}\Big)^2\eps,
\]
and $\frac1N\sum_jd_j\to0$ in $L^1$.

\emph{Second replacement.} Put $b_j'=R_js_n(z)$. By \cref{lem:rank-one},
$|s_n^{(j)}-s_n|\le\pi/(p\eta)$, so by \eqref{eq:denominators}
\[
  \Big|\frac{b_j}{1+c_nb_j}-\frac{b_j'}{1+c_nb_j'}\Big|
  \le\min\Big\{2K_z,\ R_j\Big(\frac{|z|}{\eta}\Big)^2\frac{\pi}{p\eta}\Big\},
\]
and the average over $j$ is at most
$2K_z|J_M|/N+M(|z|/\eta)^2\pi/(p\eta)$. Given $\eps>0$, choose $M$ and $\delta$
with $2K_z(\nu([M,\infty))+\delta)<\eps$; by \eqref{eq:JM} the average then
exceeds $2\eps$ with probability tending to zero. As $\eps>0$ was arbitrary the
average is $o_p(1)$. Together,
\begin{equation}
  1+zs_n(z)=s_n(z)\,\frac1N\sum_{j=1}^N\frac{R_j}{1+c_nR_js_n(z)}+o_p(1).
  \label{eq:AFP}
\end{equation}

\emph{Step 3: identification.} Fix a countable dense $\mathcal Z\subset\C_+$.
The statements $\hat\nu_N\Rightarrow\nu$ (through a countable family of test
functions), \eqref{eq:tight-countable} for rational $\delta$, and
\eqref{eq:AFP} for $z\in\mathcal Z$ are countably many convergences in
probability. Given any subsequence, a diagonal argument extracts a further
subsequence along which all of them hold almost surely. Fix such a sample
point. Along the subsequence, \eqref{eq:tight-countable} gives
$\limsup\ESD(S)([L_\delta,\infty))\le\delta$ for every rational $\delta$, so the
sequence $\ESD(S)$ is tight and every weak cluster point $\mu$ is a
probability measure on $[0,\infty)$; along a subsubsequence
$\ESD(S)\Rightarrow\mu$ and $s_n(z)\to s_\mu(z)\in\C_+$ for every $z$. Fix
$z\in\mathcal Z$, $g_n=c_ns_n(z)$, $g=cs_\mu(z)$; then $g_n\to g\in\C_+$ and
for $r\ge0$
\[
  \Big|\frac r{1+g_nr}-\frac r{1+gr}\Big|=\frac{|g_n-g|\,r^2}{|1+g_nr||1+gr|}
  \le\frac{|g_n-g|}{\Imc g_n\,\Imc g}\longrightarrow0
\]
uniformly in $r\in[0,\infty)$, because $|1+hr|\ge r\Imc h$. The limit
$r\mapsto r/(1+gr)$ is bounded and continuous by \eqref{eq:integrand-bounded},
and $\hat\nu_N\Rightarrow\nu$, so
$\frac1N\sum_jR_j/(1+c_nR_js_n(z))\to\int r(1+crs_\mu(z))^{-1}\nu(\dd r)$.
Passing to the limit in \eqref{eq:AFP} gives \eqref{eq:RMP} for $s_\mu$ at
every $z\in\mathcal Z$. By the uniqueness part of \cref{prop:limit-law},
$s_\mu$ is determined on
$\mathcal Z$; two cluster points have analytic transforms agreeing on a set
with an accumulation point in $\C_+$, hence everywhere, hence coincide. So
along the original subsequence a further subsequence converges almost surely to
one deterministic $\mu_{c,\nu}$ independent of the subsequence, which is
convergence in probability. Since the $\mu$ so obtained solves \eqref{eq:RMP}
on $\mathcal Z$ and both sides are analytic in $z$ (the right side through
$s_\mu$, by \eqref{eq:integrand-bounded}), it solves it on $\C_+$; this is the
existence part of \cref{prop:limit-law}.

\emph{Almost sure convergence.} If $\sum_ne^{-ap_n}<\infty$ for all $a>0$, the
L\'evy distance $\ell_n$ between $\ESD(S)$ and $\mu_{c,\nu}$ changes by at most
$2/p$ when one sample is changed: replacing $x_j$ perturbs $S$ by a matrix of
rank at most $2$, so $d_{\mathrm L}\le\dK\le2/p$ by \cref{lem:rank-one}. Hence
McDiarmid's
inequality gives $\Pp(\ell_n-\E\ell_n>t)\le\exp(-t^2p^2/(2N))$; since
$0\le\ell_n\le1$ and $\ell_n\to0$ in probability, $\E\ell_n\to0$, and
Borel--Cantelli finishes. No moment is involved.
\qed

\begin{remark}[a general radial variable]
\label{rem:general-radius}
Let $(y_j,\rho_j)_{j\le N}$ be i.i.d.\ pairs with $y_j\in\R^q$, $\rho_j\ge0$,
$q/N\to c'\in(0,\infty)$, $\rho_1\Rightarrow\nu$, and assume the
\emph{uniform} condition
\[
  \sup_{\|A\|\le1}\Pp\Big(\Big|\frac{y_1^{\mathsf T}Ay_1}{q}-\rho_1\frac{\Tr A}{q}\Big|>\eps\Big)\longrightarrow0
  \qquad(\eps>0).
\]
Then $\ESD(\frac1N\sum_jy_jy_j^{\mathsf T})\Rightarrow\mu_{c',\nu}$ in
probability. The proof is that of \cref{thm:sufficient} with two changes.
\Cref{prop:rqc-equivalences} is not available, because its Step 2 uses
$y^{\mathsf T}Py\le q\rho$, which is false for a general $\rho$; this is why
the uniform condition is assumed directly. Step 1 used
$\frac1N\sum_{j\notin J_M}\|y_j\|^2/q\le M$; instead apply the condition with
$A=I$ to get $\|y_1\|^2/q-\rho_1\to0$ in probability, and remove
$J_M'=\{j:\rho_j>M\text{ or }\|y_j\|^2/q>M+1\}$, after which
$\frac1N\sum_{j\notin J_M'}\|y_j\|^2/q\le M+1$ and the tightness bound of
Step~1 reads $(M+1)/L$; the proportion removed is at most
$\Pp(\rho_1>M)+\Pp(|\|y_1\|^2/q-\rho_1|>1)+o_p(1)$. Steps 2 and 3 use only that
the pairs are i.i.d.\ (so $(y_j,\rho_j)$ is independent of $G_j$) and
$\hat\nu_N\Rightarrow\nu$. This is the form in which the theorem is applied to
projections in \cref{sec:necessity}: the projected samples $C_Px_j$ come with
the parent radius $R_j$, which is a function of $x_j$ but not of $C_Px_j$.
\end{remark}

\subsection{Spectral replacement by the Gaussian model}

The sufficient theorem immediately gives comparison with a Gaussian
reference carrying the same limiting radial law.

\begin{corollary}[replacement]
\label{cor:replacement}
Assume the hypotheses of \cref{thm:sufficient} and let
$\tilde x_p=\sqrt{\rho_p}\,g_p$ where $g_p\sim N(0,I_p)$ and $\rho_p\ge0$ is
any random variable on the same probability space with $\rho_p\Rightarrow\nu$;
$\rho_p$ need not be independent of $g_p$. Let $\tilde S$ be the sample
covariance matrix of $N$ independent copies of $\tilde x_p$. Then
\[
  d_{\mathrm L}\big(\ESD(S),\ESD(\tilde S)\big)\xrightarrowp0 ,
\]
on any coupling of the two models; both limits are deterministic, so the
statement does not depend on which coupling is used.
\end{corollary}

\begin{proof}
Write $\tilde R_p=\|\tilde x_p\|^2/p=\rho_p\|g_p\|^2/p$. Since $\|g_p\|^2/p\to1$
in probability and $\rho_p\Rightarrow\nu$, Slutsky's theorem gives
$\tilde R_p\Rightarrow\nu$. For a real symmetric contraction $A$,
\[
  \tilde Q_p(A)=\frac{\tilde x^{\mathsf T}A\tilde x}{p}-\tilde R_p\frac{\Tr A}{p}
  =\rho_p\Big(\frac{g^{\mathsf T}Ag-\Tr A}{p}-\frac{\Tr A}{p}\Big(\frac{\|g\|^2}{p}-1\Big)\Big),
\]
and, for real symmetric $A$, $\Var(g^{\mathsf T}Ag)=2\Tr A^2\le2p$ and
$\Var(\|g\|^2)=2p$, so the bracket has second moment at most $8/p$ uniformly
over real symmetric contractions. With $\rho_p$ tight this gives, for every
$M>0$,
\[
  \sup_A\Pp\big(|\tilde Q_p(A)|>\eps\big)\le\Pp(\rho_p>M)+\frac{8M^2}{p\eps^2},
\]
the supremum over real symmetric contractions. Letting $p\to\infty$ and then
$M\to\infty$ gives $\tilde Q_p(P_p)\to0$ for every deterministic sequence of
orthogonal projections, which is (RQC) for $(\tilde x_p)$;
\cref{prop:rqc-equivalences} upgrades it to the uniform statement over all
complex contractions. No independence between $\rho_p$ and $g_p$ is used. \Cref{thm:sufficient} applies to both models with the same $\nu$, so
both empirical spectral distributions converge in probability to
$\mu_{c,\nu}$, and the triangle inequality for the L\'evy metric finishes.
\end{proof}

If $(x_p)$ satisfies (ACI) with radial law $\nu$, the projected form of the
comparison is again equivalent to the radial condition:
$d_{\mathrm L}(\ESD(S_P),\ESD(\tilde S_P))\to0$ in probability for every
deterministic projection sequence of proportional rank if and only if
$R_p\Rightarrow\nu$ and (RQC). Indeed $(\tilde x_p)$ satisfies (RQC) and
$\tilde R_p\Rightarrow\nu$, hence (RMP-$\Pi$) for $(c,\nu)$ by the implication
of \cref{thm:necessity} that needs no reduction, so
$\ESD(\tilde S_P)\Rightarrow\mu_{c\alpha,\nu}$ for every such $P$; the displayed
convergence is therefore (RMP-$\Pi$) for $(x_p)$, which under (ACI) is equivalent to
condition (i) of \cref{thm:necessity}.

\section{Projection-energy rigidity}
\label{sec:rigidity}

One projection rank suffices for a geometric reason: the convex hull of
rank-$q$ projections contains a neighborhood of $(q/p)I$ in the affine
space of matrices with trace $q$. A strict concavity gap in this
neighborhood controls every centered linear statistic of a random positive
semidefinite matrix. The argument does not require a rank-one matrix or
any moments.

\subsection{The convex hull of fixed-rank projections}

For $1\le q<p$, put $\beta=q/p$, $b=\min(\beta,1-\beta)$, and
\[
 \mathcal F_{p,q}=\{B=B^T:0\le B\le I,\ \Tr B=q\}.
\]
This set, called the fantope, is the convex hull of the rank-$q$
orthogonal projections \cite[Section~2]{VuChoLeiRohe2013}.
Indeed, after diagonalizing $B$, its eigenvalue vector belongs to
$\{\lambda\in[0,1]^p:\sum_i\lambda_i=q\}$. Its extreme points have
only zero and one entries: two fractional entries could be perturbed in
opposite directions, and the integer trace excludes exactly one fractional
entry. Thus $B$ is a finite convex combination of projections in its
eigenbasis.

If $A=A^T$ is a contraction, let $A_0=A-p^{-1}\Tr(A)I$ and $h=b/2$.
Since $\|A_0\|\le2$, both
\begin{equation}\label{eq:fantope-probes}
 B_\pm=\beta I\pm hA_0
\end{equation}
belong to $\mathcal F_{p,q}$. For a random positive semidefinite matrix
$M$, put $R=\Tr M/p$ and
\[
 Q_M(A)=\frac{\Tr(AM)}p-R\frac{\Tr A}p.
\]
The two corresponding energies are exactly
\begin{equation}\label{eq:fantope-energies}
 \frac{\Tr(B_\pm M)}q=R\pm\frac h\beta Q_M(A).
\end{equation}
The projections in each convex combination are deterministic. No
conditional expectation of an unbounded random matrix is needed.

\subsection{Second projected moments control all linear statistics}

\begin{lemma}[Quantitative rigidity from one projection rank]
\label{lem:energy-L2}
Let $M$ be a random real positive semidefinite $p\times p$ matrix,
$R=\Tr M/p$, and $\E R^2<\infty$. For $1\le q<p$, define
\[
 D_2(M,q)=\sup_{\rank P=q}
 \left|\E\left(\frac{\Tr(PM)}q\right)^2-\E R^2\right|.
\]
With $\beta=q/p$ and $b=\min(\beta,1-\beta)$,
\[
 \sup_{\|A\|\le1}\E|Q_M(A)|^2
 \le\frac{8\beta^2}{b^2}D_2(M,q).
\]
The supremum includes complex matrices; for real matrices the constant
can be replaced by $4\beta^2/b^2$.
\end{lemma}
\begin{proof}
The map $B\mapsto\E(\Tr(BM)/q)^2$ is convex. Its values on
$\mathcal F_{p,q}$ are therefore at most $\E R^2+D_2(M,q)$.
For a real symmetric contraction, average this bound at $B_+$ and $B_-$
in \eqref{eq:fantope-probes}. Equation~\eqref{eq:fantope-energies} gives
\[
 \E R^2+(h/\beta)^2\E|Q_M(A)|^2\le\E R^2+D_2(M,q).
\]
Symmetrization covers real contractions. For a complex contraction,
its transpose-symmetric part is still a contraction; its real and
imaginary parts are real symmetric contractions. The squared modulus
is the sum of the two squared real statistics, giving the stated factor.
\end{proof}

\subsection{A one-sided bounded certificate for positive operators}

Put $\psi(t)=t/(1+t)$. Unlike the preceding second moments, the
expectations below exist for every random $M\ge0$ with finite entries
almost surely, even when all positive moments of its trace are infinite.

\begin{theorem}[One-sided projected-energy rigidity]
\label{thm:bounded-energy}
Let $M$ be a random real positive semidefinite $p\times p$ matrix,
$R=\Tr M/p$, and $R_P=\Tr(PM)/q$. For $1\le q<p$, define
\begin{align*}
 \Delta_\psi(M,q)&=\sup_{\rank P=q}
       \{\E\psi(R)-\E\psi(R_P)\},\\
 D_\psi(M,q)&=\sup_{\rank P=q}|\E\psi(R_P)-\E\psi(R)|,\\
 \mathcal W(M)&=\sup_{\|A\|\le1}
       \E\frac{|Q_M(A)|^2}{(1+R)^3}.
\end{align*}
The matrices in these suprema are deterministic.
Set $\beta=q/p$ and $b=\min(\beta,1-\beta)$. Then
\begin{equation}\label{eq:one-sided-defects}
 0\le\Delta_\psi(M,q)\le D_\psi(M,q)
 \le\frac{1-b}{b}\Delta_\psi(M,q),
\end{equation}
and
\begin{equation}\label{eq:bounded-rigidity}
 \mathcal W(M)\le\frac{8\beta^2}{b^2}\Delta_\psi(M,q).
\end{equation}
For real test matrices the constant is $4\beta^2/b^2$.
Neither bound has a dimension remainder or a restriction on $\rank M$.
For $M=xx^T$ we write $D_\psi(x,q)$, $\Delta_\psi(x,q)$ and
$\mathcal W(x)$ for these quantities. In particular, if
$a\le q/p\le1-a$, then for every $K,t>0$,
\begin{equation}\label{eq:bounded-rigidity-tail}
 \sup_{\|A\|\le1}\Pp(|Q_p(A)|>t)
 \le\Pp(R>K)+\frac{C_a(1+K)^3}{t^2}\Delta_\psi(x,q).
\end{equation}
\end{theorem}

\begin{proof}
The functional
\[
 F(B)=\E\psi(\Tr(BM)/q)
\]
is concave on $\mathcal F_{p,q}$. The convex-hull identity shows that
$F(B)\ge\E\psi(R)-\Delta_\psi(M,q)$ throughout that set.
The choice $B=\beta I$ gives $\Delta_\psi(M,q)\ge0$.

For a rank-$q$ projection $P$, the matrix
$K_P=(\beta I-bP)/(1-b)$ belongs to $\mathcal F_{p,q}$ and
$\beta I=bP+(1-b)K_P$. Concavity and the lower bound on $F(K_P)$ give
\[
 \E\psi(R)\ge bF(P)+(1-b)F(K_P)
 \ge bF(P)+(1-b)\{\E\psi(R)-\Delta_\psi(M,q)\}.
\]
This bounds the positive deviation of $F(P)$ by
$(1-b)\Delta_\psi(M,q)/b$ and proves \eqref{eq:one-sided-defects}.

For a real symmetric contraction $A$, use $B_\pm$ from
\eqref{eq:fantope-probes} and write $u=(h/\beta)Q_M(A)$.
Their energies $R\pm u$ are nonnegative. The exact midpoint identity is
\[
 \psi(R)-\frac{\psi(R+u)+\psi(R-u)}2
 =\frac{u^2}{(1+R)((1+R)^2-u^2)}
 \ge\frac{u^2}{(1+R)^3}.
\]
Its expected left side is at most $\Delta_\psi(M,q)$, so
\[
 \E\frac{|Q_M(A)|^2}{(1+R)^3}
 \le\frac{4\beta^2}{b^2}\Delta_\psi(M,q).
\]
The real and complex reductions in \cref{lem:energy-L2} prove
\eqref{eq:bounded-rigidity}. The weighted integrands are bounded by $4R^2/(1+R)^3\le16/27$.
Markov's inequality on $\{R\le K\}$ proves
\eqref{eq:bounded-rigidity-tail}.
\end{proof}

The convex hull and the scalar Jensen identity are standard ingredients;
the latter also appears in \cite[Lemma~4.1]{Yaskov2015b}. Their use here
gives a one-sided, dimension-free certificate for random positive
operators. For a random density matrix $H\ge0$ with $\Tr H=1$, take
$M=pH$: then $R=1$ and the conclusion controls
$\Tr(AH)-\Tr A/p$. It concerns deterministic linear tests, not
operator-norm proximity to $I/p$ or rotational invariance.

Under tightness of $\Tr M_p/p$, with $q_p/p$ bounded away from zero and
one, the same argument shows that vanishing
$\Delta_\psi(M_p,q_p)$ is equivalent to uniform convergence in probability
of $Q_{M_p}(A)$ over contractions. Conversely, boundedness and the
Lipschitz property of $\psi$ give the projection defects from this
convergence. Tightness is used only to remove the weight in
\eqref{eq:bounded-rigidity}.

\begin{corollary}[one bounded statistic at one rank]\label{cor:one-transform}
Suppose $(R_p)$ is tight, fix any $\alpha\in(0,1)$, and set
$q_p=\lfloor\alpha p\rfloor$. The following are equivalent:
\begin{enumerate}[label=(\roman*)]
\item (RQC);
\item $\mathcal W(x_p)\to0$;
\item $D_\psi(x_p,q_p)\to0$;
\item for every deterministic rank-$q_p$ sequence $P_p$,
\[
 \E\psi(\|P_px_p\|^2/q_p)-\E\psi(R_p)\to0.
\]
\item $\Delta_\psi(x_p,q_p)\to0$.
\end{enumerate}
If $R_p\Rightarrow\nu$, the common expectation in (iv) is
$\int\psi(r)\nu(\dd r)$ in the limit.
\end{corollary}
\begin{proof}
The equivalence of (iii) and (iv) is the deterministic-sequence selection
argument of \cref{lem:uniform}. Equation~\eqref{eq:bounded-rigidity} gives
(iii)$\Rightarrow$(ii), and localization at $R_p\le B$, followed by
tightness, gives (ii)$\Rightarrow$(i).
Under (i), uniform convergence in probability and boundedness and uniform
continuity of $\psi$ give (iii), since $R_P-R_p=(p/q_p)Q_p(P)$.
Also, (i) gives (ii) directly from the bounded envelope
$|Q_p(A)|^2/(1+R_p)^3\le4R_p^2/(1+R_p)^3\le16/27$.
The equivalence of (iii) and (v) is \eqref{eq:one-sided-defects}.
\end{proof}

\begin{remark}[why tightness matters]\label{rem:bounded-tightness}
For even $p$, let $T=\diag(I_{p/2},2I_{p/2})$, $g\sim N(0,I_p)$,
and $x=\sqrt p\,T^{1/2}g$. Then $R=g^TTg$ is not tight; indeed $R/p\to3/2$.
For every rank-$q$ projection, the projected Gaussian second-moment matrix
is at least $pI_q$. Thus, for $q>2$,
\[
 \E(1+R_P)^{-1}\le\frac{q}{p(q-2)},\qquad
 \E(1+R)^{-1}\le\frac1{p-2}.
\]
Consequently $D_\psi(x,q)=O_a(p^{-1})$ uniformly in the projection.
For $P$ onto the first coordinate block, however,
\[
 Q_p(P)=\tfrac12\|g^{(1)}\|^2-\|g^{(2)}\|^2,\qquad Q_p(P)/p\to-1/4.
\]
Thus transformed-energy preservation does not imply
absolute RQC without radial tightness.
\end{remark}

\subsection{From the bounded energy statistic to spectral replacement}

The geometric certificate also gives a direct finite-dimensional comparison.
The column replacement step below is a quantitative specialization of the
method of \cite{Yaskov2014}; the input through $\Delta_\psi$ is supplied by
\cref{thm:bounded-energy}.

\begin{proposition}\label{prop:defect-replacement}
Let $x_1,\ldots,x_N$ be independent copies of $x\in\R^p$.
Set $R_j=\|x_j\|^2/p$, and independently take Haar unit vectors $u_j$.
Put $\widetilde x_j=\sqrt{pR_j}\,u_j$. Let $B_0\ge0$ be independent of all
these variables, with no bound on its operator norm, and put
\[
 m_x(-t)=\frac1p\Tr\left(B_0+\frac1N\sum_jx_jx_j^T+tI\right)^{-1}
\]
and define $m_{\widetilde x}(-t)$ similarly. Write $c_p=p/N$ and
$\tau(K)=\Pp(R>K)$. For $a\le q/p\le1-a$ and every $t,\delta,K>0$,
\begin{align}
 |\E m_x(-t)-\E m_{\widetilde x}(-t)|
 \le{}&\frac{2\delta}{t^2}
 +\frac{C_a(1+K)^3\Delta_\psi(x,q)}{c_p t\delta^2}
 +\frac{3\tau(K)}{c_p t}
 +\frac{2\sqrt2K}{t^2\sqrt{p+2}}.
 \label{eq:defect-replacement}
\end{align}
No moments or limiting radius law are required.
\end{proposition}

\begin{proof}
Telescope by replacing columns one at a time. At one step let
$G=(B+tI)^{-1}$, where $B$ contains $B_0$ and all other columns, and is
independent of the pair being replaced. For $v=x_j$ or $\widetilde x_j$, set
\[
 a_v=v^TGv/N,\quad b_v=v^TG^2v/N,\quad
 a_0=R_j\Tr G/N,\quad b_0=R_j\Tr G^2/N.
\]
Sherman--Morrison gives the normalized trace increment
$-b_v/[p(1+a_v)]$. Since $G^2\le G/t$, both $b_v/(1+a_v)$ and
$b_0/(1+a_0)$ belong to $[0,1/t]$, and
\[
 \left|\frac{b_v}{1+a_v}-\frac{b_0}{1+a_0}\right|
 \le |b_v-b_0|+t^{-1}|a_v-a_0|.
\]
Let $\Phi(\delta)=\sup_{\|A\|\le1}\Pp(|Q_p(A)|>\delta)$.
For the original column, simultaneous control of the two contractions
$tG,t^2G^2$ bounds the last display by $2c_p\delta/t^2$ outside an
event of probability at most $2\Phi(\delta)$.
The bounded ratio difference controls that exceptional event by
$2\Phi(\delta)/t$ in expectation.
For the spherical column,
\[
 \E\left|u^TAu-\frac{\Tr A}{p}\right|^2
 \le\frac2{p+2}
\]
for real symmetric contractions. On $R_j\le K$ this gives an expected
ratio error at most $2c_pK\sqrt{2/(p+2)}/t^2$; on its complement use $1/t$.
Divide by $p$ and sum the $N$ replacement steps to obtain
\[
 |\E m_x(-t)-\E m_{\widetilde x}(-t)|
 \le\frac{2\delta}{t^2}+\frac{2\Phi(\delta)}{c_p t}
 +\frac{2\sqrt2K}{t^2\sqrt{p+2}}+\frac{\tau(K)}{c_p t}.
\]
Apply \eqref{eq:bounded-rigidity-tail} at cutoff $K$ and absorb constants.
The current radius is never conditioned upon when using the original
quadratic-form bound.
\end{proof}

For tight radii and $c_p$ bounded above and away from zero,
$\Delta_\psi(x_p,q_p)\to0$ makes the right side vanish at every fixed $t>0$:
choose $K$ and $\delta$ first, then let $p\to\infty$, and finally
let $K\to\infty$ and $\delta\downarrow0$.
This comparison requires neither a deterministic limit for the background
nor identification of a common limiting ESD. It is a global resolvent
estimate on the negative real axis; it is not a local or anisotropic law.

\section{The moment-free one-rank converse}
\label{sec:converse}

This section gives the primary proof of the one-rank characterization.

\subsection{Angular symmetrization of the logarithmic determinant}

The spectral input for the converse will be a logarithmic determinant.
Its key property is concavity in the law of an independent column, with
the same law used in every column. This is different from the usual
concavity of the logarithmic determinant in a matrix argument.

\begin{lemma}[Concavity in the column law]\label{lem:log-law-concavity}
Fix $q,N\ge1$ and $t>0$. For a probability law $\sigma$ on $\R^q$ with
$\int\log(1+\|y\|^2)\sigma(\dd y)<\infty$, put
\[
 F_{q,N,t}(\sigma)=\frac1q\E\log\det\left(
 I+\frac1{Nt}\sum_{j=1}^N y_jy_j^T\right),
 \qquad y_j\overset{\mathrm{iid}}\sim\sigma.
\]
Then $F_{q,N,t}$ is finite and concave in $\sigma$. If $\sigma^\circ$ is
the law of $Oy$, where $O$ is an independent Haar orthogonal matrix, then
\begin{equation}\label{eq:log-angular}
 F_{q,N,t}(\sigma)\le F_{q,N,t}(\sigma^\circ).
\end{equation}
The same concavity inequality holds for probability mixtures of laws
whenever their mixture has a finite logarithmic moment.
\end{lemma}

\begin{proof}
For boundedly supported laws $\sigma_0,\sigma_1$, write
$\sigma_u=(1-u)\sigma_0+u\sigma_1$ and $\delta=\sigma_1-\sigma_0$.
For $N\ge2$, differentiation of the product measure gives
\[
 \frac{\dd^2}{\dd u^2}F_{q,N,t}(\sigma_u)
 =\frac{N(N-1)}q\E_B\iint
 \log\det(B+vv^T+ww^T)\,\delta_N(\dd v)\delta_N(\dd w),
\]
where $\delta_N$ is the pushforward under $y\mapsto y/\sqrt N$, and
$B=tI$ plus the other $N-2$ independent rank-one terms.
Set $a(v)=v^TB^{-1}v$ and
$u_B(v)=B^{-1/2}v/\sqrt{1+a(v)}$. The determinant lemma twice yields
\begin{align*}
 \log\det(B+vv^T+ww^T)
 ={}&\log\det B+\log(1+a(v))+\log(1+a(w))\\
 &+\log\bigl(1-\langle u_B(v),u_B(w)\rangle^2\bigr).
\end{align*}
Terms depending on at most one of $v,w$ integrate to zero.
In fact $\|u_B(v)\|^2=a(v)/(1+a(v))<1$.
If $A_*=\sup_v\|v\|^2/t$ over the two bounded supports, then
$a(v)\le A_*$. Cauchy--Schwarz in the $B^{-1}$ inner product gives
\[
 \begin{aligned}
 \langle u_B(v),u_B(w)\rangle^2
 &=\frac{(v^TB^{-1}w)^2}{(1+a(v))(1+a(w))}\\
 &\le\frac{a(v)a(w)}{(1+a(v))(1+a(w))}
 \le\frac{A_*^2}{(1+A_*)^2}<1.
 \end{aligned}
\]
Expanding the remaining logarithm therefore gives
the absolutely convergent expression
\[
 -\sum_{k\ge1}\frac1k
 \left\|\int u_B(v)^{\otimes 2k}\,\delta_N(\dd v)\right\|^2\le0.
\]
This proves concavity; for $N=1$ the functional is affine.

The bound
\begin{equation}\label{eq:log-column-bound}
 0\le F_{q,N,t}(\sigma)
 \le\frac Nq\int\log\left(1+\frac{\|y\|^2}{Nt}\right)\sigma(\dd y)
\end{equation}
follows by successive rank-one updates. For general laws, replace $y$
by $y\1_{\{\|y\|\le K\}}$ and let $K\to\infty$.
For each coupled sample the covariance increases in the PSD order, so
monotone convergence passes the concavity inequality to the limit;
\eqref{eq:log-column-bound} ensures finiteness.
For bounded laws, Jensen's inequality extends from finite mixtures to
general mixtures by continuity of the bounded product integral under
weak convergence; truncation then gives the stated general version.
Finally, $F_{q,N,t}(O\sigma)=F_{q,N,t}(\sigma)$ for every fixed $O$.
Applying mixture concavity to Haar measure proves \eqref{eq:log-angular}.
\end{proof}

\begin{lemma}[Logarithmic spectral uniform integrability]
\label{lem:log-spectral-UI}
Let $y_q\in\R^q$, let $q/N$ stay bounded above and away from zero,
and suppose $\{\log(1+\|y_q\|^2/q)\}$ is uniformly integrable.
For the covariance $S$ of $N$ independent copies and every $t>0$,
the random spectral integrals of $\log(1+\lambda/t)$ have uniformly
vanishing expected tails. Consequently, if
$\ESD(S)\Rightarrow\mu$ weakly in probability for a deterministic
probability measure $\mu$, then
\[
 \E\frac1q\log\det(I+S/t)
 \longrightarrow\int\log(1+\lambda/t)\,\mu(\dd\lambda)<\infty.
\]
The tail bounds hold uniformly over an additional index if the logarithmic
uniform integrability does. Convergence of expected logarithmic determinants
is uniform when weak ESD convergence is also uniform.
\end{lemma}

\begin{proof}
Put $r_q=\|y_q\|^2/q$ and $d_q=q/N$.
Split $S=A_K+B_K$ according as the column radius is at most $K$ or
exceeds $K$. Successive determinant updates give
\begin{equation}\label{eq:log-tail-columns}
 0\le\frac1q\log\det(I+S/t)-\frac1q\log\det(I+A_K/t)
 \le\frac1q\sum_{j:r_j>K}\log(1+d_qr_j/t).
\end{equation}
The expectation of the right side tends uniformly to zero as
$K\to\infty$. For fixed $K$, $\E\Tr A_K/q\le K$.
Write $h(\lambda)=\log(1+\lambda/t)$ and order the eigenvalues of
$S,A_K$ as $s_i\ge a_i$. Since
\[
 (h(s_i)-L)_+\le(h(a_i)-L)_++h(s_i)-h(a_i),
\]
the expected normalized sum of the left side is at most the
expectation in \eqref{eq:log-tail-columns} plus
\[
 K\sup_{\lambda:\,h(\lambda)>L}\frac{h(\lambda)}\lambda.
\]
First let $L\to\infty$, then $K\to\infty$. The displayed supremum
tends to zero, proving the tail claim. Apply weak convergence to
$h\wedge L$, take expectations using boundedness, and remove the
truncation. For an additional index, uniform logarithmic integrability
makes the expectation in \eqref{eq:log-tail-columns} small uniformly,
while the core bound $K$ is independent of that index. Thus the same
ordered limits give uniform spectral tails. For each fixed cap,
uniform weak convergence in probability implies uniform convergence of
expectations, since the capped test is bounded. Removing the cap using
these uniform tail bounds proves uniform convergence of the expected
logarithmic determinants; the limit measures inherit the same tail bound
by bounded truncation and passage to the limit.
\end{proof}

\subsection{A variational form of the radial logarithmic transform}

For $d,t>0$ and a radial law $\eta$ with
$\int\log(1+r)\eta(\dd r)<\infty$, define
\[
 \mathcal J_{d,\eta}(t)
 =\int\log(1+\lambda/t)\,\mu_{d,\eta}(\dd\lambda).
\]
Write $s_{d,\eta}=s_{\mu_{d,\eta}}$ for the Stieltjes transform of this law.

\begin{lemma}\label{lem:log-variational}
The preceding integral is finite, and
\begin{equation}\label{eq:log-variational}
 \mathcal J_{d,\eta}(t)=\inf_{s>0}
 \left\{-\log(ts)+\frac1d\int\log(1+drs)\,\eta(\dd r)-1+ts\right\}.
\end{equation}
The unique minimizer is $s=s_{d,\eta}(-t)$.
If $\eta_n\Rightarrow\eta$ and the logarithms of the corresponding
radii are uniformly integrable, then
$\mathcal J_{d,\eta_n}(t)\to\mathcal J_{d,\eta}(t)$.
If $Z,T\ge0$ have finite logarithmic moments and
$T=\E[Z\mid\mathcal G]<\infty$ almost surely, then
\begin{equation}\label{eq:log-conditional-Jensen}
 \mathcal J_{d,\law(T)}(t)\ge\mathcal J_{d,\law(Z)}(t).
\end{equation}
Here conditional expectation of a nonnegative variable is allowed
without assuming its unconditional first moment finite.
\end{lemma}

\begin{proof}
Let $V_\eta(t,s)$ be the expression in braces in
\eqref{eq:log-variational}. Its derivative with respect to $\log s$ is
\[
 -1+ts+\int\frac{rs}{1+drs}\,\eta(\dd r).
\]
This is strictly increasing from $-1$ to infinity as $s$ increases,
so the unique global minimizer solves the radial MP equation at $-t$.
Differentiation is justified since the radial integrands in this
derivative are bounded. At the minimizer the derivative of the value
with respect to $t$ is $s_{d,\eta}(-t)-1/t$.

To identify the value, use spherical columns
$y_q=\sqrt{q r}\,u_q$, with $r\sim\eta$ and $u_q$ uniform on the unit
sphere, independently. The sufficient theorem and
\cref{lem:log-spectral-UI}, together with
\eqref{eq:log-column-bound}, show that $\mathcal J_{d,\eta}(t)$ is finite
and bounded above by $d^{-1}\int\log(1+dr/t)\eta(\dd r)$.
Its derivative with respect to $t$ is also $s_{d,\eta}(-t)-1/t$.
Both functions tend to zero as $t\to\infty$: for the variational
value use
\[
 0\le\inf_s V_\eta(t,s)\le V_\eta(t,1/t)
 =\frac1d\int\log(1+dr/t)\,\eta(\dd r),
\]
and for the spectral value use the same upper bound. Logarithmic
dominated convergence now proves \eqref{eq:log-variational}.

For the asserted continuity, the minimizers belong to $(0,1/t]$.
They are bounded away from zero uniformly in $n$. Otherwise, tightness
and the bounded integrand in their scalar equation would make
its right side tend to zero instead of one. Weak convergence and
logarithmic uniform integrability give uniform convergence of
$V_{\eta_n}(t,s)$ on the resulting compact interval, proving continuity.
Finally, conditional Jensen gives
$\E\log(1+dsT)\ge\E\log(1+dsZ)$ for every $s>0$.
Taking infima proves \eqref{eq:log-conditional-Jensen}.
\end{proof}

\subsection{Proof of the converse at one projection rank}

Deleting a fraction of columns gives a lower bound by a correspondingly
trimmed spectral logarithm. For the Gaussian reference model, the next lemma
bounds the difference in the opposite direction independently of the
magnitudes of the deleted radii.
The lemma itself is deterministic. Deleting $k$ columns changes the matrix
by rank at most $k$, so interlacing gives one inequality. For the reverse
inequality, both the determinant increment and the upper spectral trimming
contain the large factors $\log(1+r_j)$ of the deleted columns.
A Schur-complement normalization cancels these factors, leaving an
$O(k)$ error controlled by frame bounds. In the Gaussian application those
bounds are uniform in the large radii, so the error requires no logarithmic
moment.

\begin{lemma}[A determinant bound for radial deletion]
\label{lem:frame-trimming}
Let $r_j\ge0$ and $t>0$, and let $G\in\R^{q\times N}$ have columns $g_j$. Put
\[
 S=\frac1N\sum_j r_jg_jg_j^T,\qquad
 S^K=\frac1N\sum_{r_j\le K}r_jg_jg_j^T.
\]
Fix $M>0$, write $\mathcal I=\{j:r_j>M\}$, and suppose $|\mathcal I|\le q$ and $G_{\mathcal I}$ has full
column rank when $\mathcal I\ne\varnothing$. Let
\[
 \begin{gathered}
 C_0=I+\frac1{Nt}\sum_{r_j\le M}r_jg_jg_j^T,\\
 \overline\varsigma=\max\{1,\|G\|^2/(Nt)\},\qquad
 \underline\varsigma=\min\{1,s_{\min}(G_{\mathcal I})^2/(Nt\|C_0\|)\}.
 \end{gathered}
\]
For $\mathcal I=\varnothing$ take $\underline\varsigma=1$. For $K\ge M$, let $k=\#\{j:r_j>K\}$ and
order the eigenvalues of $S$ decreasingly. Then
\begin{equation}\label{eq:frame-trimming}
 0\le \log\det(I+S^K/t)
       -\sum_{i=k+1}^q\log(1+\lambda_i(S)/t)
 \le k\log(\overline\varsigma/\underline\varsigma).
\end{equation}
In particular, the bound has no dependence on $\max_j r_j$.
\end{lemma}

\begin{proof}
The lower bound is eigenvalue interlacing under deletion of $k$ positive
rank-one terms. If $k=0$ the assertion is immediate.
Put $D_{\mathcal I}=\diag(r_j:j\in \mathcal I)$ and
\[
 \mathcal V=\frac1{Nt}G_{\mathcal I}^TC_0^{-1}G_{\mathcal I},\qquad
 Z=(I+D_{\mathcal I})^{-1/2}
   (I+D_{\mathcal I}^{1/2}\mathcal V D_{\mathcal I}^{1/2})
   (I+D_{\mathcal I})^{-1/2}.
\]
Since $C_0\ge I$, one has
\[
 \frac{s_{\min}(G_{\mathcal I})^2}{Nt\|C_0\|}I\le \mathcal V\le\frac{\|G\|^2}{Nt}I.
\]
Thus $\underline\varsigma I\le Z\le\overline\varsigma I$. Indeed $Z$ is the sum of
$(I+D_{\mathcal I})^{-1}$ and the congruence of $\mathcal V$ by
$D_{\mathcal I}^{1/2}(I+D_{\mathcal I})^{-1/2}$, whose squared coefficients sum to $I$ with
$(I+D_{\mathcal I})^{-1}$.

Let $J=\{j:r_j>K\}\subset \mathcal I$ and $E=\mathcal I\setminus J$. The determinant lemma,
first for $S$ and then for $S^K$, gives
\[
 \log\det(I+S/t)-\log\det(I+S^K/t)
 =\sum_{j\in J}\log(1+r_j)+\log\det Z-\log\det Z_{EE}.
\]
Use determinant one for an empty principal block. The last difference is
the log determinant of a $k\times k$ Schur complement of $Z$, which is at
least $k\log \underline\varsigma$: its inverse is a principal block of $Z^{-1}\le \underline\varsigma^{-1}I$.
On the other hand,
\[
 I_N+\frac1{Nt}D^{1/2}G^TGD^{1/2}\le I_N+\overline\varsigma D,\qquad D=\diag(r_1,\ldots,r_N).
\]
Since $k\le\min(q,N)$, its largest $k$ eigenvalues agree with those of
$I_q+S/t$. The set $J$ indexes the largest $k$ radii. Eigenvalue monotonicity
and $\overline\varsigma\ge1$ therefore give
\[
 \sum_{i=1}^k\log(1+\lambda_i(S)/t)
 \le\sum_{j\in J}\log(1+\overline\varsigma r_j)
 \le\sum_{j\in J}\log(1+r_j)+k\log \overline\varsigma.
\]
Subtracting the determinant increment cancels the factors
$\log(1+r_j)$ and proves the upper bound.
\end{proof}

For a probability measure $\mu$ on $[0,\infty)$, let $Q_\mu$ be its increasing
quantile function. For $0<\delta<1$ define the finite trimmed logarithm
\[
 \mathcal T_{\delta,t}(\mu)
   =\int_0^{1-\delta}\log(1+Q_\mu(u)/t)\,\dd u.
\]
It is continuous under weak convergence of $\mu$ and convergence of
$\delta$ within $(0,1)$. To see this, choose a continuity point $L$ with
$\mu([0,L])>1-\delta/2$. The retained quantiles are then bounded by $L$
eventually, and their almost-everywhere convergence gives the assertion.
This reasoning also applies to random measures converging in probability.

\begin{lemma}[A tail-probability bound from the projected spectra]
\label{lem:log-tail-transfer}
Assume $R_p\Rightarrow\nu$ and the projected spectral hypothesis (ii) of
\cref{thm:necessity-free}, with no radial moment assumption. Write $d=c\alpha$.
Choose $M>0$ such that $\nu((M,\infty))<d/4$.
For a continuity point $K\ge M$ of $\nu$, put
\[
 x_p^K=x_p\1_{\{R_p\le K\}},\qquad
 \nu_K=\law(R\1_{\{R\le K\}}),\quad R\sim\nu,\qquad
 \tau_K=\nu((K,\infty)).
\]
For each $t>0$ there is a finite constant $C=C(d,M,t)$, independent of $K$,
such that
\begin{equation}\label{eq:log-tail-transfer}
 \liminf_p\inf_{\rank P=q_p}F_{q_p,N,t}(\law(C_Px_p^K))
 \ge\mathcal J_{d,\nu_K}(t)-\frac{C\tau_K}{d}.
\end{equation}
\end{lemma}

\begin{proof}
First suppose $\tau_K>0$, so $0<\delta_K=\tau_K/d<1/4$.
The number $k_p$ of deleted columns satisfies $k_p/q_p\to\delta_K$ in
probability. Interlacing gives
\[
 F_{q_p,N,t}(\law(C_Px_p^K))
 \ge \E\frac1{q_p}\sum_{i=k_p+1}^{q_p}
           \log(1+\lambda_i(S_P)/t).
\]
Set the sum to zero when $k_p\ge q_p$. The spectral hypothesis and
continuity of the trimmed logarithm make the random statistic converge
in probability to $\mathcal T_{\delta_K,t}(\mu_{d,\nu})$.
Nonnegativity and Fatou give the same lower bound for the liminf of its
expectation. Sequence selection makes the conclusion uniform in $P$.

It remains to prove
\begin{equation}\label{eq:gaussian-trimmed-comparison}
 \mathcal J_{d,\nu_K}(t)
 \le\mathcal T_{\delta_K,t}(\mu_{d,\nu})+\frac{C\tau_K}{d}.
\end{equation}
Use a Gaussian reference with $q/N\to d$, independent $r_j\sim\nu$ and
standard Gaussian $G\in\R^{q\times N}$. Its full and radially deleted
covariances have limits $\mu_{d,\nu}$ and $\mu_{d,\nu_K}$ by the sufficient
theorem. Conditional on the radii, $G_{\mathcal I}$ for $\mathcal I=\{j:r_j>M\}$ is a standard
Gaussian matrix with $|\mathcal I|/N\to\nu((M,\infty))<d/4$.
The Gaussian singular-value bounds \cite[Theorem~II.13]{DavidsonSzarek2001} imply that,
on events $\mathcal E_p$ with probability tending to one,
\[
 |\mathcal I|<q/2,\qquad
 \|G\|^2/N\le B_d,\qquad
 s_{\min}(G_{\mathcal I})^2/N\ge A_d>0
\]
when $\mathcal I\ne\varnothing$, for fixed constants $A_d,B_d$.
Moreover $\|C_0\|\le1+MB_d/t$. The constants $\underline\varsigma,\overline\varsigma$ in
\cref{lem:frame-trimming} consequently satisfy
\[
 \underline\varsigma\ge\min\{1,A_d/(t+MB_d)\},\qquad
 \overline\varsigma\le\max\{1,B_d/t\}
\]
on $\mathcal E_p$. Their logarithmic ratio is bounded by a constant
$C(d,M,t)$ independent of $K$ and of all radial magnitudes.

For each fixed $K$ the normalized logarithm of the deleted reference
covariance, and also its trimmed full spectral logarithm by interlacing,
are bounded above by
\[
 \frac1{qt}\Tr S^K\le\frac{K}{qNt}\|G\|_{\mathrm F}^2.
\]
Here $\|\cdot\|_{\mathrm F}$ is the Frobenius norm.
Put $W_p=\|G\|_{\mathrm F}^2/(qN)$. Since its numerator is chi-squared
with $qN$ degrees of freedom, $\E W_p^2=1+2/(qN)$.
For either logarithmic statistic $X_p$ just bounded, Cauchy--Schwarz gives
\[
 \E[X_p\1_{\mathcal E_p^c}]
 \le (K/t)(\E W_p^2)^{1/2}\Pp(\mathcal E_p^c)^{1/2}=o(1)
\]
at fixed $K$. The same $L^2$ envelope makes the trimmed statistics uniformly
integrable. Their convergence in probability therefore gives convergence
in expectation to $\mathcal T_{\delta_K,t}(\mu_{d,\nu})$.

For the deleted reference, the sufficient theorem gives weak spectral
convergence. For $\lambda>L$,
\[
 \log(1+\lambda/t)-\log(1+L/t)
 =\log\left(1+\frac{\lambda-L}{t+L}\right)
 \le\frac{\lambda}{t+L};
\]
the positive part is zero for $\lambda\le L$.
Since $\E\Tr S^K/q=\E[r\1_{\{r\le K\}}]\le K$, capping the logarithm
leaves expected error at most $K/(t+L)$. First take $p\to\infty$ at fixed
$K,L$, then $L\to\infty$, to identify the expected deleted logarithm
with $\mathcal J_{d,\nu_K}(t)$.

Finally $k_p$ is binomial with parameters $N,\tau_K$, so
$\E(k_p/q)=(N/q)\tau_K\to\tau_K/d$.
Taking expectations in \cref{lem:frame-trimming} proves
\eqref{eq:gaussian-trimmed-comparison}.

If $\tau_K=0$, continuity of $\nu$ at $K$ and $R_p\Rightarrow\nu$ give
$\Pp(R_p>K)\to0$. Hence
$\E(k_p/q_p)=(N/q_p)\Pp(R_p>K)\to0$, and Markov's inequality gives
$k_p/q_p\to0$ in probability. The rank inequality therefore preserves
the projected spectral limit $\mu_{d,\nu}$ after deletion.
Lower semicontinuity for the nonnegative logarithm gives
\eqref{eq:log-tail-transfer} with zero error, since $\nu_K=\nu$ has bounded
support. Uniformity again follows by sequence selection.
\end{proof}

The converse proceeds through subsequential joint limits of the parent
and projected energies $(R_p,\rho_{P_p})$. An averaged complementary
projection gives an energy $V_p$ with
$R_p=\theta_p\rho_{P_p}+(1-\theta_p)V_p$.
At a fixed radial cutoff, angular symmetrization, column-law concavity,
and the deletion/trimming bound give lower logarithmic-transform bounds
for both projected energies. Comparing these bounds with the parent
transform controls a nonnegative strict Jensen gap by the radial tail
probability. Sending the cutoff to infinity forces $\rho=V=R$ in every
joint limit. Deterministic-sequence selection then makes projected-energy
convergence uniform, and projection-energy rigidity yields RQC.

\begin{proof}[Proof of Theorem~\ref{thm:necessity-free}]
Assume (ii), and write $d=c\alpha$, $\beta_p=q_p/p$.
Fix a deterministic projection sequence $P_p$. Along any subsequence we
can extract a further subsequence such that
\[
 (R_p,\rho_{P_p})\Rightarrow(R,\rho),\qquad
 \rho_P=\|Px_p\|^2/q_p,\quad R\sim\nu.
\]
Indeed $0\le\rho_P\le R_p/\beta_p$, so the pair is tight.

If $\beta_p\le1/2$, let $U_p$ be an independent Haar rank-$q_p$
subprojection of $P_p^\perp$. If $\beta_p>1/2$, let
$U_p=P_p^\perp+H_p$, with $H_p$ an independent Haar rank-$(2q_p-p)$
subprojection of $P_p$. Set
\[
 V_p=\E_{U_p}[\rho_{U_p}\mid x_p],\qquad
 \theta_p=\min(\beta_p,1-\beta_p).
\]
The Haar mean gives the exact identity
\begin{equation}\label{eq:log-complement-energy}
 R_p=\theta_p\rho_{P_p}+(1-\theta_p)V_p.
\end{equation}
Thus $V_p$ converges jointly to
$V=(R-\theta\rho)/(1-\theta)\ge0$, where
$\theta=\min(\alpha,1-\alpha)>0$.

Fix $t>0$, choose $M$ as in \cref{lem:log-tail-transfer}, and let
$C=C(d,M,t)$. Fix a continuity cutoff $K\ge M$ of $\nu$. Define
\[
 \begin{aligned}
 (R_p^K,\rho_{P_p}^K,V_p^K)
   &=(R_p,\rho_{P_p},V_p)\1_{\{R_p\le K\}},\\
 (R_K,\rho_K,V_K)&=(R,\rho,V)\1_{\{R\le K\}}.
 \end{aligned}
\]
The first triple converges weakly to the second and is uniformly bounded.
It is exactly the energy triple of $x_p^K$ and the same projections.
All logarithmic integrability and continuity arguments below therefore
concern bounded variables.

Let $\sigma_{q,\eta}^{\circ}$ denote the law of $\sqrt{qr}\,u_q$ with
independent $r\sim\eta$ and uniform spherical $u_q$.
Angular symmetrization, spherical convergence, and
\cref{lem:log-spectral-UI,lem:log-tail-transfer} give
\begin{equation}\label{eq:log-one-sided}
 \mathcal J_{d,\law(\rho_K)}(t)
 \ge\mathcal J_{d,\nu_K}(t)-C\tau_K/d.
\end{equation}
To obtain the complementary inequality, put
$\bar\eta_{p,K}=\law(\rho_{U_p}\1_{\{R_p\le K\}})$, averaging over
both $x_p$ and $U_p$. Column-law concavity and angular symmetrization give
\begin{align}
 F_{q_p,N,t}(\sigma_{q_p,\bar\eta_{p,K}}^\circ)
 &\ge\E_{U_p}F_{q_p,N,t}
   (\sigma_{q_p,\law(\rho_{U_p}\1_{\{R_p\le K\}}\mid U_p)}^\circ)
   \notag\\
 &\ge\E_{U_p}F_{q_p,N,t}(\law(C_{U_p}x_p^K\mid U_p)).
 \label{eq:log-probe-mixture}
\end{align}
The liminf of the last expression is at least
$\mathcal J_{d,\nu_K}(t)-C\tau_K/d$ by the uniform bound in
\cref{lem:log-tail-transfer}. Extract a further subsequence on which
$\bar\eta_{p,K}$ converges. Its support is uniformly bounded, so spherical
convergence and logarithmic uniform integrability identify the limit of
the left side. Conditional Jensen in \cref{lem:log-variational} gives
\[
 \mathcal J_{d,\law(V_p^K)}(t)\ge
 \mathcal J_{d,\bar\eta_{p,K}}(t).
\]
Continuity for bounded radial laws now yields
\begin{equation}\label{eq:log-complement-lower}
 \mathcal J_{d,\law(V_K)}(t)
 \ge\mathcal J_{d,\nu_K}(t)-C\tau_K/d.
\end{equation}

Let $s_K=s_{d,\nu_K}(-t)$. Evaluating the three variational formulas at
this parent minimizer and using $R_K=\theta\rho_K+(1-\theta)V_K$ gives
\begin{align}
 \mathcal J_{d,\nu_K}(t)
 \ge{}&\theta\mathcal J_{d,\law(\rho_K)}(t)
 +(1-\theta)\mathcal J_{d,\law(V_K)}(t)+\frac1d\E J_K,\notag\\
 J_K={}&\log(1+ds_KR_K)-\theta\log(1+ds_K\rho_K)\notag\\
       &\hspace{2em}-(1-\theta)\log(1+ds_KV_K)\ge0.
 \label{eq:log-strict-gap}
\end{align}
The lower bounds above imply the quantitative estimate
\begin{equation}\label{eq:truncated-gap-bound}
 \E J_K\le C\tau_K.
\end{equation}
The constant is independent of $K$, and $\tau_K\to0$ for every probability
law on $[0,\infty)$. Hence $\E J_K\to0$ along continuity cutoffs tending
to infinity, without a radial tail assumption.

No logarithmic moment is needed to pass the minimizers to the limit.
Their equation is
\[
 1=ts_K+\frac1d\E\psi(ds_KR_K),\qquad 0<s_K\le1/t.
\]
Tightness of $(R_K)$ and boundedness of $\psi$ prevent $s_K\to0$.
Every other subsequential limit solves the unique negative-real radial
equation for $\nu$, hence $s_K\to s_\nu=s_{d,\nu}(-t)>0$.
On the joint-limit probability space, $J_K$ therefore converges pointwise to
\[
 J=\log(1+ds_\nu R)-\theta\log(1+ds_\nu\rho)
                    -(1-\theta)\log(1+ds_\nu V)\ge0.
\]
Each value of $J$ is finite because $R,\rho,V$ are finite almost surely;
its expectation has not been presumed finite. Fatou's lemma along the
chosen cutoffs gives $\E J=0$. Strict scalar concavity proves
$\rho=V=R$ almost surely.

Every subsequential joint limit has equal projected and parent energies.
Thus $\rho_{P_p}-R_p\to0$ in probability for every rank-$q_p$ sequence.
Sequence selection makes this uniform over those projections. Boundedness
and the Lipschitz property of $\psi$ imply $D_\psi(x_p,q_p)\to0$.
The bounded projected-energy criterion gives (RQC).

If $\{R_p^2\}$ is uniformly integrable, the bound $|Q_p(A)|\le2R_p$
upgrades uniform convergence in probability to
$\sup_{\|A\|\le1}\E|Q_p(A)|^2\to0$.
Finally, (i)$\Rightarrow$(ii) follows from
\cref{thm:sufficient,rem:general-radius} applied to the projected samples
with their parent radius.
\end{proof}

\section{A finite-dimensional spectral inverse}
\label{sec:finite-inverse}

Our goal here is a finite-sample bound on quadratic-form defects in terms
of bounded tests of expected projected spectra. The quantitative estimate
uses a finite radial Jensen inequality in place of the limiting
variational formula.
Throughout this section, put
\[
 \beta=q/p,\qquad b=\min(\beta,1-\beta),\qquad d=q/N,
 \qquad 1\le q<p,\quad t>0,
\]
and abbreviate
\[
 \Psi(\eta)=F_{q,N,t}(\sigma^\circ_{q,\eta}),\qquad
 \sigma^\circ_{q,\eta}=\law(\sqrt{qr}\,u),
\]
where $r\sim\eta$ and the uniform unit vector $u\in\R^q$ are independent.
Only bounded radial laws are used as arguments of $\Psi$ below.

\begin{lemma}[Finite-sample strict radial Jensen inequality]
\label{lem:finite-radial-gap}
Let $0<b\le1/2$, let bounded nonnegative variables satisfy
$r=bz+(1-b)w$, and suppose
$\Pp(r>M)\le d/8$ for some $M>0$. Put
\[
 \kappa=\frac{\min(1,d/t)^2}{32(1+M/(bt))^2}.
\]
Then
\begin{align}
 &\Psi(\law(r))-b\Psi(\law(z))-(1-b)\Psi(\law(w))\notag\\
 &\hspace{2em}\ge\frac{b(1-b)\kappa}{2d}
           \E[\psi(z)-\psi(w)]^2 .
 \label{eq:finite-radial-gap}
\end{align}
In particular, the constant does not depend on the upper support bound.
For bounded $r\pm v\ge0$ with $\Pp(r>M)\le d/8$, the midpoint version is
\begin{equation}\label{eq:finite-midpoint-gap}
 \begin{gathered}
 \Psi(\law(r))-\frac{\Psi(\law(r+v))+\Psi(\law(r-v))}2
 \ge\frac{\kappa_*}{2d}\E\frac{v^2}{(1+r)^2},\\
 \kappa_*=\frac{\min(1,d/t)^2}{32(1+2M/t)^2}.
 \end{gathered}
\end{equation}
\end{lemma}

\begin{proof}
Let $Z$ select $z$ with probability $b$ and $w$ otherwise, independently
of the triple. Column-law concavity gives
$\Psi(\law(Z))\ge b\Psi(\law(z))+(1-b)\Psi(\law(w))$.
Replace the $N$ independent radial weights $r_j$ by $Z_j$ one at a time.
For each replacement, condition on all other columns and apply the
determinant lemma. Its expected loss, including the factor $1/q$, is
$q^{-1}\E J_a(z,w)$, where
\[
 \begin{split}
 J_a(z,w)&=\log(1+ar)-b\log(1+az)-(1-b)\log(1+aw),\\
 a&=(d/t)u^TB^{-1}u,\qquad
 B=I+(d/t)\sum_{i\ne j}v_i u_i u_i^T.
 \end{split}
\]
The background $B$ and the current direction $u$ are independent of the
current triple. Nonnegativity gives $z_i\le r_i/b$ and
$w_i\le r_i/(1-b)\le r_i/b$, since $b\le1/2$.
Each background weight, whether $r_i$ or $Z_i$, is therefore at most
$r_i/b$, and the underlying $r_i$ remain independent.
Scalar strong concavity gives the pointwise bound
\begin{equation}\label{eq:finite-scalar-gap}
 J_a(z,w)\ge
 \frac{b(1-b)a^2(z-w)^2}{2(1+a\max(z,w))^2}
 \ge\frac{b(1-b)}2\min(a,1)^2[\psi(z)-\psi(w)]^2.
\end{equation}

Let $h=\#\{i\ne j:r_i>M\}$. Markov's inequality gives
$\Pp(h>q/2)\le2(N-1)\Pp(r>M)/q\le1/4$.
On $h\le q/2$, compress $B$ to the orthogonal complement of those
high-radius directions, of dimension $m\ge q/2$.
The compressed trace is at most $q(1+M/(bt))$.
Interlacing and the arithmetic--harmonic mean inequality imply
\[
 \frac1q\Tr B^{-1}\ge
 \frac{m^2}{q^2(1+M/(bt))}\ge\frac1{4(1+M/(bt))}.
\]
Since $0\le u^TB^{-1}u\le1$,
$\min(a,1)\ge\min(d/t,1)u^TB^{-1}u$.
Conditional spherical averaging, Jensen's inequality, and
$\Pp(h\le q/2)\ge3/4$ give
\[
 \E\min(a,1)^2
 \ge\frac{3\min(d/t,1)^2}{64(1+M/(bt))^2}\ge\kappa.
\]
The bound is uniform over the replacement position.
This coefficient is independent of $(r,z,w)$. Taking expectations in
\eqref{eq:finite-scalar-gap} and summing the $N$ losses proves the claim.
For the midpoint assertion use mixing weight $1/2$ in the same argument,
so background weights are at most $2r_i$. Replace the scalar bound by
\[
 \begin{aligned}
 &\log(1+ar)-\frac{\log(1+a(r+v))+\log(1+a(r-v))}2\\
 &\qquad=-\frac12\log\left(1-\frac{a^2v^2}{(1+ar)^2}\right)
 \ge\frac{\min(a,1)^2v^2}{2(1+r)^2}.
 \end{aligned}
\]
The coefficient estimate is now $\kappa_*$, proving
\eqref{eq:finite-midpoint-gap}.
\end{proof}

\begin{corollary}[A finite logarithmic inverse]\label{cor:finite-log-inverse}
Let $R=\|x\|^2/p<\infty$ almost surely and choose $M>0$ with
$\tau_M:=\Pp(R>M)\le d/8$. For $K\ge M$, put
$x^K=x\1_{\{R\le K\}}$, $R_K=R\1_{\{R\le K\}}$, and
\[
 e_K=\Psi(\law(R_K))-
     \inf_{\rank P=q}F_{q,N,t}(\law(C_Px^K)).
\]
Then $e_K\ge0$ and, with $\kappa$ as in \cref{lem:finite-radial-gap},
\begin{equation}\label{eq:finite-log-inverse}
 \Delta_\psi(x,q)\le\tau_K+
          \sqrt{\frac{2d e_K}{b(1-b)\kappa}},
 \qquad \tau_K=\Pp(R>K).
\end{equation}
There is also the linear-error bound
\begin{equation}\label{eq:finite-log-weighted}
 \sup_{\|A\|\le1}\E\frac{|Q_x(A)|^2\1_{\{R\le K\}}}{(1+R)^2}
 \le\frac{16d\beta^2}{b^2\kappa_*}e_K,
 \qquad Q_x(A)=x^TAx/p-R\Tr A/p.
\end{equation}
The constant is halved for real matrices.
\end{corollary}

\begin{proof}
On the fantope $\mathcal F_{p,q}$ define
$\widetilde\Psi(B)=\Psi(\law((x^K)^TBx^K/q))$.
Couple the same independent $x_j^K,u_j$ in every argument. The matrix
inside its logarithmic determinant is affine in $B$, so $\widetilde\Psi$ is
concave. On projections, angular symmetrization gives
$\widetilde\Psi(P)\ge F_{q,N,t}(\law(C_Px^K))$.
The convex-hull identity therefore bounds $\widetilde\Psi$ below by the
infimum defining $e_K$ throughout the fantope. At $B=\beta I$ this proves
$e_K\ge0$.

For fixed $P$, the matrix $B_P=(\beta I-bP)/(1-b)$ also belongs to
the fantope. The coupled energies satisfy
$R_K=b\rho_P^K+(1-b)\rho_{B_P}^K$.
Apply \cref{lem:finite-radial-gap}. Both endpoint functionals are at least
the same infimum, so
\[
 \E[\psi(\rho_P^K)-\psi(R_K)]^2
 \le\E[\psi(\rho_P^K)-\psi(\rho_{B_P}^K)]^2
 \le\frac{2d e_K}{b(1-b)\kappa}.
\]
Here monotonicity of $\psi$ and the between-endpoints identity give the
first inequality. Cauchy--Schwarz and a bound of $\tau_K$ on the omitted
tail event prove \eqref{eq:finite-log-inverse}.

For a real symmetric contraction $A$, use the two fantope points
$B_\pm=\beta I\pm(b/2)(A-p^{-1}\Tr(A)I)$.
Their truncated energies are $R_K\pm v$, with
$v=(b/(2\beta))Q_x(A)\1_{\{R\le K\}}$.
Both values of $\widetilde\Psi$ are at least the same infimum.
Equation~\eqref{eq:finite-midpoint-gap} therefore proves
\eqref{eq:finite-log-weighted} with constant $8d\beta^2/(b^2\kappa_*)$.
Symmetrization and the real--imaginary decomposition used in
\cref{lem:energy-L2} give the stated real and complex versions.
\end{proof}

We next bound $e_K$ using only bounded tests of expected spectral measures.
Let
\[
 S^\circ=\frac qN\sum_{j=1}^N R_j u_j u_j^T,\qquad
 \bar\mu^\circ=\E\ESD(S^\circ),\qquad \bar\mu_P=\E\ESD(S_P),
\]
where the reference uses independent uniform directions and independent
radii with the \emph{actual} law of $R$. For $L>0$ put
$B_L=\log(1+L/t)$ and define the one-sided spectral discrepancy
\begin{equation}\label{eq:capped-spectral-discrepancy}
 \chi_L=\sup_{\rank P=q}\sup_{0\le v\le B_L}
 \left[\int\min\{\log(1+\lambda/t),v\}
             (\bar\mu^\circ-\bar\mu_P)(\dd\lambda)\right]_+.
\end{equation}
With $d_{\rm BL}$ defined by tests of sup norm and Lipschitz constant
at most one,
\[
 \chi_L\le\max(B_L,1/t)
       \sup_{\rank P=q}d_{\rm BL}(\bar\mu^\circ,\bar\mu_P).
\]
No expected logarithm of an untruncated covariance is used.

\begin{theorem}[Finite-dimensional spectral-to-energy inverse]
\label{thm:finite-spectral-inverse}
Under the notation and assumptions of \cref{cor:finite-log-inverse}, set
\[
 \begin{gathered}
 A_d=d/256,\qquad B_d=16(1+\sqrt d)^2,\qquad
 C=\log\frac{\max(1,B_d/t)}{\min(1,A_d/(t+MB_d))},\\
 \gamma=(1-1/\sqrt2)^2/8,\qquad
 \eta_{q,N}=\min\{1,(e/4)^{q/2}+e^{-N/2}
                       +e^{-\gamma q}+2Ne^{-q/8}\}.
 \end{gathered}
\]
For every $K\ge M$ and $L>0$ define
\begin{align}
 \mathcal E_{K,L}={}&\chi_L+
 B_L\left\{\frac{2\sqrt{N\tau_K(1-\tau_K)}}q+
       \frac2q+\sqrt{\frac{N\log(2(q+1))}{2q^2}}\right\}\notag\\
 &+\frac{C\tau_K}{d}+\frac{K\eta_{q,N}}t+\frac K{t+L}.
 \label{eq:finite-spectral-error}
\end{align}
Then
\begin{equation}\label{eq:finite-spectral-inverse}
 \Delta_\psi(x,q)\le\tau_K+
      \sqrt{\frac{2d\mathcal E_{K,L}}{b(1-b)\kappa}}.
\end{equation}
More directly, for any $0<U\le K$ and $\varepsilon>0$,
\begin{equation}\label{eq:finite-spectral-probability}
 \sup_{\|A\|\le1}\Pp(|Q_x(A)|>\varepsilon)
 \le\Pp(R>U)+
 \frac{16d\beta^2(1+U)^2}{b^2\kappa_*\varepsilon^2}\mathcal E_{K,L}.
\end{equation}
\end{theorem}

\begin{proof}
It suffices to prove $e_K\le\mathcal E_{K,L}$.
Represent the spherical reference directions by
$g_j=\sqrt q\,u_j=z_j\sqrt q/\|z_j\|$, with independent standard Gaussian
$z_j$. Let $\mathcal I=\{j:R_j>M\}$.
Chernoff's inequality gives
$\Pp(|\mathcal I|>q/2)\le(e/4)^{q/2}$ and, simultaneously for all $j$,
$\sqrt q/2\le\|z_j\|\le2\sqrt q$ except with probability
at most $2Ne^{-q/8}$.
The Gaussian singular-value bounds \cite{DavidsonSzarek2001} give
\[
 \|Z\|\le\sqrt q+2\sqrt N,\qquad
 s_{\min}(Z_{\mathcal I})\ge(1-1/\sqrt2)\sqrt q/2
\]
except with probabilities $e^{-N/2}$ and $e^{-\gamma q}$, respectively;
the second assertion is conditional on $\mathcal I$ with $0<|\mathcal I|\le q/2$.
The singular-value condition is omitted when $\mathcal I$ is empty.
Consequently, outside an event of probability at most $\eta_{q,N}$,
\[
 \|G\|^2/N\le B_d,\qquad s_{\min}(G_{\mathcal I})^2/N\ge A_d
\]
when needed. In \cref{lem:frame-trimming}, the resulting logarithmic
ratio is at most $C$, independently of $K$ and the radial magnitudes.
The radially deleted reference logarithm is at most $K/t$ samplewise,
because its normalized trace is at most $K$.
The bad event therefore costs at most $K\eta_{q,N}/t$.

Write $h_L(\lambda)=\min\{\log(1+\lambda/t),B_L\}$ and
\[
 \mathcal T_{\delta,L}(\mu)=\int_0^{1-\delta}h_L(Q_\mu(u))\,\dd u,
 \qquad \delta=\tau_K/d\le1/8.
\]
The definition includes $\delta=0$. If $k$ columns are deleted, take the
trim fraction to be $\min(k/q,1)$.
Interlacing bounds the portion clipped off the retained reference logarithm
by $K/(t+L)$, since
$[\log(1+\lambda/t)-B_L]_+\le\lambda/(t+L)$ and the retained eigenvalues
are bounded by those of the core.
The trimmed integral is $B_L$-Lipschitz in its trim fraction.
In each model $k\sim\mathrm{Bin}(N,\tau_K)$ and $\E(k/q)=\delta$.
Since clipping $k/q$ to $[0,1]$ can only decrease its distance from
$\delta\in[0,1]$, replacing the random fraction by $\delta$ costs at most
\[
 B_L\E|k/q-\delta|
 \le B_L\sqrt{\Var(k/q)}
 =\frac{B_L\sqrt{N\tau_K(1-\tau_K)}}q.
\]
Thus \cref{lem:frame-trimming} and original-model interlacing give
\begin{align*}
 e_K\le{}&
 \E \mathcal T_{\delta,L}(\ESD(S^\circ))
 -\inf_P\E \mathcal T_{\delta,L}(\ESD(S_P))\\
 &+\frac{2B_L\sqrt{N\tau_K(1-\tau_K)}}q
   +\frac{C\tau_K}{d}+\frac{K\eta_{q,N}}t+\frac K{t+L}.
\end{align*}

For any probability measure,
\begin{equation}\label{eq:trimmed-cap-variational}
 \mathcal T_{\delta,L}(\mu)=\max_{0\le v\le B_L}
 \left\{\int\min(h_L(\lambda),v)\,\mu(\dd\lambda)-\delta v\right\}.
\end{equation}
In particular this functional is convex, so the original-model expectation
is at least $\mathcal T_{\delta,L}(\bar\mu_P)$.
For the reference, replacing one column changes each capped spectral
integral by at most $B_L/q$: the two rank-one additions interlace the
same background. Write $X(v)$ for this integral minus its expectation.
The bounded-differences bound gives
$\E e^{sX(v)}\le\exp(s^2NB_L^2/(8q^2))$ for every real $s$.
At grid points $v_\ell=\ell B_L/q$, $0\le\ell\le q$, summing this bound
for both signs and using Jensen gives, for $s>0$,
\[
 \E\max_\ell|X(v_\ell)|
 \le\frac{\log(2(q+1))}{s}+\frac{sNB_L^2}{8q^2}.
\]
Minimizing in $s$ yields $B_L\sqrt{N\log(2(q+1))/(2q^2)}$.
Each integral and its expectation are $1$-Lipschitz in $v$, so $X$ is
$2$-Lipschitz. A grid point within $B_L/q$ adds at most $2B_L/q$.
The excess of the expected maximum in
\eqref{eq:trimmed-cap-variational} over the maximum of expectations
is at most $\E\sup_v|X(v)|$. Therefore
\[
 \E \mathcal T_{\delta,L}(\ESD(S^\circ))
 \le \mathcal T_{\delta,L}(\bar\mu^\circ)
       +B_L\left\{\frac2q+\sqrt{\frac{N\log(2(q+1))}{2q^2}}\right\}.
\]
Finally, \eqref{eq:trimmed-cap-variational} bounds
$\mathcal T_{\delta,L}(\bar\mu^\circ)-\mathcal T_{\delta,L}(\bar\mu_P)$ by $\chi_L$.
This proves $e_K\le\mathcal E_{K,L}$.
Equation~\eqref{eq:finite-log-weighted} and Markov's inequality on
$\{R\le U\}\subset\{R\le K\}$ prove
\eqref{eq:finite-spectral-probability}.
\end{proof}

\begin{remark}[Asymptotic consequence]
\label{rem:finite-second-proof}
Under \cref{thm:necessity-free}(ii) and $R_p\Rightarrow\nu$,
the reference has the same limit, so $\chi_L\to0$ for fixed $L$;
with a fixed continuity cutoff $M$ satisfying $\nu((M,\infty))<d/8$,
where $d=c\alpha$, the ordered limits $p\to\infty$, $L\to\infty$,
$K\to\infty$ in \eqref{eq:finite-spectral-inverse}
recover $\Delta_\psi\to0$ and RQC.
The reference still uses the actual radius law; this does not solve the
radius-identification question below.
\end{remark}

\begin{corollary}[One-sided spectral tests]\label{cor:one-sided-spectra}
Assume $R_p\Rightarrow\nu$, $p/N\to c>0$, and
$q_p=\lfloor\alpha p\rfloor$ for fixed $\alpha\in(0,1)$.
Fix any one $t>0$. Then (RQC) is equivalent to the following condition:
for every deterministic rank-$q_p$ projection sequence and every $v>0$,
\[
 \begin{gathered}
 \liminf_p\E\int f_{t,v}(\lambda)\,\ESD(S_{P_p})(\dd\lambda)
 \ge\int f_{t,v}(\lambda)\,\mu_{c\alpha,\nu}(\dd\lambda),\\
 f_{t,v}(\lambda)=\min\{\log(1+\lambda/t),v\}.
 \end{gathered}
\]
In particular, equality or prior convergence of the projected ESDs need
not be assumed in this criterion.
\end{corollary}

\begin{proof}
The sufficient theorem gives the forward implication.
For the reverse, the spherical reference has the stated radial MP limit.
Sequence selection and a finite grid of cap values, using the
$1$-Lipschitz dependence of $f_{t,v}$ on $v$, give $\chi_L\to0$ for
every fixed $L$. Apply \cref{thm:finite-spectral-inverse} with the
ordered limits in \cref{rem:finite-second-proof}.
\end{proof}

\section{Radius identification}
\label{sec:radius}

Condition~\eqref{eq:no-small-subspace-energy} itself implies asymptotic
tightness of $R_p$. For fixed $\delta>0$, let $U_p$ be an independent Haar
projection of rank $\lfloor\delta p\rfloor$. Conditional on any nonzero $x_p$,
$\|U_px_p\|^2/\|x_p\|^2\to\delta$ in probability, uniformly in the vector.
Consequently
\[
 (1-o(1))\Pp(R_p>2\varepsilon/\delta)
 \le\sup_{\rank P\le\delta p}\Pp(x_p^TPx_p/p>\varepsilon).
\]
Take the limsup in $p$ and then let $\delta\downarrow0$.
The proof of \cref{thm:radius-identification} has two stages. Under uniform
integrability of $R_p^2$, first projected spectral moments recover
$p^{-1}\|\Sigma_p-m_1I\|_1\to0$, where
$m_1=\int\lambda\,\mu(\dd\lambda)$; second spectral moments and Haar projection
identities then give uniform $L^2$ RQC. Injectivity of the radial MP map
identifies the radius law. For general actual radii, the tightness just
proved permits subsequential radial limits. The finite second moment of
the common spectral limit bounds the second moment of each such radial
limit. A slowly increasing radial cutoff produces uniformly integrable
squared radii while preserving all projected spectral limits, reducing
the argument to the first stage. RQC then transfers back in probability.
We first record two preparatory lemmas.

\begin{lemma}[Small-subspace and population energy]
\label{lem:population-energy}
Suppose $\{R_p\}$ is uniformly integrable, and let $\Sigma_p=\E x_px_p^T$.
Then \eqref{eq:no-small-subspace-energy} is equivalent to
\begin{equation}\label{eq:population-energy-UI}
 \lim_{L\to\infty}\limsup_p\frac1p
       \Tr\bigl(\Sigma_p\1_{\{\Sigma_p>L\}}\bigr)=0.
\end{equation}
\end{lemma}
\begin{proof}
Write $Z_P=x_p^TPx_p/p\le R_p$. For $\varepsilon,K>0$,
\[
 \E Z_P\le\varepsilon+K\Pp(Z_P>\varepsilon)
                    +\E[R_p\1_{\{R_p>K\}}].
\]
Take the supremum over $\rank P\le\delta p$, then the limsup in $p$.
First let $\delta\downarrow0$, then $K\to\infty$, then
$\varepsilon\downarrow0$. The small-subspace condition and uniform
integrability give vanishing expected energy in these subspaces.
With $C_0=\sup_p\E R_p<\infty$ and $E_{p,L}=\1_{\{\Sigma_p>L\}}$,
\[
 \rank E_{p,L}/p\le C_0/L,\qquad
 \Tr(\Sigma_pE_{p,L})/p=\E Z_{E_{p,L}}.
\]
This proves \eqref{eq:population-energy-UI}. For fixed $L$ the rank
bound is $O(p/L)$, not necessarily $o(p)$; the relative rank vanishes
in the ordered limit $p\to\infty$, $L\to\infty$.
Conversely, for $\rank P\le\delta p$,
\[
 \E Z_P=\Tr(P\Sigma_p)/p
 \le L\delta+\Tr(\Sigma_pE_{p,L})/p.
\]
Markov's inequality, followed by $\delta\downarrow0$ and $L\to\infty$,
gives \eqref{eq:no-small-subspace-energy}.
\end{proof}

\begin{lemma}[Slow radial truncation]
\label{lem:slow-truncation}
If $X_n\ge0$, $X_n\Rightarrow X$, and $\E X^2<\infty$, there are deterministic
cutoffs $K_n\to\infty$, each a continuity point of the law of $X$, such that
\[
 \widehat X_n=X_n\1_{\{X_n\le K_n\}}\Rightarrow X,
 \qquad \{\widehat X_n^2\}\text{ is uniformly integrable}.
\]
\end{lemma}
\begin{proof}
Choose continuity points $L_j\uparrow\infty$. At each fixed $j$,
$\E[X_n^2\1_{\{X_n\le L_j\}}]\to\E[X^2\1_{\{X\le L_j\}}]$.
Choose increasing integers $n_j$ so that the difference is at most
$1/j$ for all $n\ge n_j$, and set $K_n=L_j$ for
$n_j\le n<n_{j+1}$, with a fixed finite cutoff for earlier indices.
Tightness gives $\Pp(X_n>K_n)\to0$, hence $\widehat X_n\Rightarrow X$,
and the diagonal construction gives $\E\widehat X_n^2\to\E X^2$.
For $Y_n=\widehat X_n^2$ and $Y=X^2$, bounded truncation then yields
\[
 \E(Y_n-\Theta)_+=\E Y_n-\E(Y_n\wedge \Theta)\longrightarrow\E(Y-\Theta)_+.
\]
The bound $Y_n\1_{\{Y_n>2\Theta\}}\le2(Y_n-\Theta)_+$ proves asymptotic uniform
integrability. Each of the finitely many remaining variables is bounded,
so the whole family is uniformly integrable.
\end{proof}

\begin{proof}[Proof of \cref{thm:radius-identification}]
We first assume uniform integrability of $R_p^2$, then use the second
moment of $\mu$ to remove this assumption.
Write $\beta_p=q_p/p$ and $d_p=q_p/N\to d$.
For every bounded continuous spectral test, convergence in probability
implies convergence of its expectation. Sequence selection, as in
\cref{lem:uniform}(b), makes these deterministic expectation bounds
uniform over rank-$q_p$ orientations. Uniform moment bounds below then
permit unbounded tests. Haar projections are always independent of the
sample: we condition on their realized orientation and integrate a
uniform deterministic bound.

\emph{Population energy and the first spectral moment.}
Let $\Sigma_p=\E x_px_p^T$. By \cref{lem:population-energy},
\eqref{eq:population-energy-UI} holds.

For any vector $y\in\R^q$ with $\Lambda=\E yy^T$, independence gives
\begin{equation}\label{eq:inverse-spectral-second}
 \E\frac1q\Tr S_y^2
 =\frac{N-1}{Nq}\Tr\Lambda^2+\frac1{Nq}\E\|y\|^4.
\end{equation}
Set $E_{p,L}=\1_{\{\Sigma_p>L\}}$ and
$x_p^L=(I-E_{p,L})x_p$.
For fixed $L$, \eqref{eq:inverse-spectral-second} bounds the expected
second spectral moments of all projections of $x_p^L$, since its
covariance norm is at most $L$ and $\sup_p\E R_p^2<\infty$.
Put $a_{p,L}=p^{-1}\Tr(\Sigma_pE_{p,L})$.
The decomposition into low and high population components gives
\[
 \sup_P\E\frac1q\|S_P-S_P^L\|_1
 \le\beta_p^{-1}\{a_{p,L}+2\sqrt{\E R_p\,a_{p,L}}\}.
\]
The eigenvalue trace-norm inequality and
$(\lambda-\Theta)_+\le\lambda^2/\Theta$ give
\[
 \sup_P\E\int(\lambda-\Theta)_+\,\ESD(S_P)(\dd\lambda)
 \le \frac{C_L}{\Theta}
 +\beta_p^{-1}\{a_{p,L}+2\sqrt{\E R_p\,a_{p,L}}\},
\]
where $C_L$ is a uniform bound for the core's second spectral moment.
First let $\Theta\to\infty$ at fixed $L$, then $L\to\infty$.
Since $\lambda\1_{\{\lambda>2\Theta\}}\le2(\lambda-\Theta)_+$,
\eqref{eq:population-energy-UI} proves uniform integrability of
the first spectral moments. Consequently, with
$m_1=\int\lambda\,\mu(\dd\lambda)<\infty$,
\[
 \sup_{\rank P=q_p}\left|\frac{\Tr(P\Sigma_p)}{q_p}-m_1\right|\to0.
\]
The fantope probes \eqref{eq:fantope-probes}, together with their center
$\beta_pI$, imply
\begin{equation}\label{eq:identified-population-isotropy}
 \frac1p\|\Sigma_p-m_1I\|_1\to0.
\end{equation}
For example, if the preceding supremum is $\epsilon_p$, the trace-norm
quantity is at most $(1+2\beta_p/b_p)\epsilon_p$, where
$b_p=\min(\beta_p,1-\beta_p)$.
Indeed, the two probes $B_\pm=\beta_pI\pm(b_p/2)A_0$ have normalized
traces against $\Sigma_p$ within $\epsilon_p$ of $m_1$.
Subtracting their bounds gives
$|\Tr(A_0\Sigma_p)/p|\le2\beta_p\epsilon_p/b_p$.
The center gives $|\Tr\Sigma_p/p-m_1|\le\epsilon_p$; adding the scalar
part of a real symmetric contraction and using trace-norm duality
proves the displayed constant.

\emph{The second spectral moment and energy rigidity.}
Let $E_p=\1_{\{\Sigma_p>2m_1+1\}}$, $z_p=(I-E_p)x_p$, and
$\Gamma_p=\E z_pz_p^T$.
Equation~\eqref{eq:identified-population-isotropy} gives
\[
 \begin{gathered}
 \rank E_p=o(p),\quad e_p:=x_p^TE_px_p/p\to0\text{ in }L^1,\quad
 \|\Gamma_p\|\le2m_1+1,\\
 p^{-1}\|\Gamma_p-m_1I\|_1\to0.
 \end{gathered}
\]
Since $e_p\le R_p$, uniform integrability of $R_p^2$ gives
$\E e_p^2\to0$. The sample covariances of $z_p$ and $x_p$ differ by rank
at most $2\rank E_p$, so their projected spectral limits agree.

The second spectral moments for $z_p$ are uniformly integrable.
To verify this, write $z_p^K=z_p\1_{\{R_p\le K\}}$ and split each sample
covariance into its core and its positive tail $T_P^K$.
By \eqref{eq:inverse-spectral-second},
\[
 \E\frac1q\Tr(T_P^K)^2
 \le\frac{2m_1+1}{\beta_p}\E[R_p\1_{\{R_p>K\}}]
   +\frac{p/N}{\beta_p}\E[R_p^2\1_{\{R_p>K\}}],
\]
which vanishes uniformly as $K\to\infty$.
For fixed $K$, the core has uniformly bounded third spectral moments.
The exact formula is
\begin{align}
 \E\frac1q\Tr S_y^3
 ={}&\frac{(N-1)(N-2)}{N^2q}\Tr\Lambda^3\notag\\
 &+\frac{3(N-1)}{N^2q}\E[\|y\|^2y^T\Lambda y]
   +\frac1{N^2q}\E\|y\|^6 .
 \label{eq:inverse-spectral-third}
\end{align}
For $y=C_Pz_p^K$, use $\|\Lambda\|\le2m_1+1$ and $\|y\|^2\le pK$.
Order the full and core eigenvalues decreasingly as $\lambda_i\ge a_i$,
and write $b_i=\lambda_i-a_i\ge0$. For $a,b\ge0$ and $\Theta>0$,
\[
 (a+b)^2\1_{\{a+b>\Theta\}}
 \le 2a^2\1_{\{a>\Theta/2\}}+4b^2.
\]
Indeed, on $a>\Theta/2$ use $(a+b)^2\le2a^2+2b^2$; on the remaining
part of $\{a+b>\Theta\}$ one has $b>\Theta/2\ge a$.
Hoffman--Wielandt gives $\sum_i b_i^2\le\Tr(T_P^K)^2$. Consequently
\begin{equation}\label{eq:inverse-second-tail}
 \sup_P\E\frac1q\sum_i\lambda_i^2\1_{\{\lambda_i>\Theta\}}
 \le\frac{4C_K}{\Theta}
       +4\sup_P\E\frac1q\Tr(T_P^K)^2,
\end{equation}
where $C_K$ bounds the core's third spectral moment.
First let $\Theta\to\infty$ for fixed $K$, then $K\to\infty$.
This proves the required second-moment uniform integrability.

It follows that $m_{2,\mu}=\int\lambda^2\,\mu(\dd\lambda)<\infty$.
The bounded norm and trace-norm approximation of $\Gamma_p$ imply
$q^{-1}\Tr(C_P\Gamma_pC_P^T)^2\to m_1^2$ uniformly.
Equation~\eqref{eq:inverse-spectral-second} therefore yields
\[
 \sup_{\rank P=q_p}
 \left|\E\left(\frac{\|Pz_p\|^2}{q_p}\right)^2
                   -\frac{m_{2,\mu}-m_1^2}{d}\right|\to0.
\]
For an independent Haar rank-$q$ projection $U$,
\begin{equation}\label{eq:inverse-Haar-second}
 \E_U\left(\frac{\|Uz_p\|^2}{q}\right)^2
 =\frac{p(q+2)}{q(p+2)}
       \left(\frac{\|z_p\|^2}{p}\right)^2.
\end{equation}
Thus the parent second radial moment converges to the same value, and
\cref{lem:energy-L2} gives uniform $L^2$ RQC for $z_p$.
For every contraction,
\[
 |Q_{x_p}(A)-Q_{z_p}(A)|\le2\sqrt{R_pe_p}+e_p.
\]
Indeed,
$x_px_p^T-z_pz_p^T=(x_p-z_p)x_p^T+z_p(x_p-z_p)^T$.
Its trace norm divided by $p$ is at most
$\sqrt{e_p}(\sqrt{R_p}+\sqrt{R_p-e_p})\le2\sqrt{R_pe_p}$;
the centering term contributes at most $e_p$.
Its square tends to zero in probability and has a constant multiple of
$R_p^2$ as envelope, so RQC transfers in $L^2$ to $x_p$.
Every subsequential radius limit $\eta$ now satisfies $\mu=\mu_{d,\eta}$
by the sufficient theorem. Tightness and injectivity in
\cref{prop:limit-law} give a unique limit for $R_p$.
This proves the uniformly integrable case.

\emph{Remove moment assumptions on the actual radii.}
Assume now only \eqref{eq:no-small-subspace-energy} and $m_{2,\mu}<\infty$.
Tightness was proved at the start of this section. Fix a subsequence with
$R_p\Rightarrow\eta$ and a continuity cutoff $K$ of $\eta$.
Let $x_p^K=x_p\1_{\{R_p\le K\}}$ and
$\Sigma_{p,K}=\E x_p^K(x_p^K)^T$.
For $E_{p,K,L}=\1_{\{\Sigma_{p,K}>L\}}$, its rank is at most $pK/L$.
The bounded vector $x_p^K$ inherits the small-subspace condition, so
\eqref{eq:population-energy-UI} applies to it and
\[
 \lim_{L\to\infty}\limsup_p
 \E[(x_p^K)^TE_{p,K,L}x_p^K/p]=0.
\]
Fix $L>K/(1-\alpha)$. Since $\rank E_{p,K,L}\le pK/L$ and
$q_p=\lfloor\alpha p\rfloor$, its complement has dimension at least $q_p$.
Consider rank-$q_p$ projections $P$ inside this
complement. Their core vectors $C_Px_p^K$ have covariance norm at most
$L$ and squared norm at most $pK$.
Equation~\eqref{eq:inverse-spectral-third} bounds their third spectral
moments uniformly in $p$ and $P$ at fixed $K,L$; no uniformity in
$K$ or $L$ is needed.
Let $S_P^K$ be the covariance of these core vectors and $S_P$ the
original projected covariance on the same samples.
Since $0\le S_P^K\le S_P$, eigenvalue monotonicity gives
\[
 \E\int(\lambda^2\wedge \Theta)\,\ESD(S_P^K)(\dd\lambda)
 \le \E\int(\lambda^2\wedge \Theta)\,\ESD(S_P)(\dd\lambda).
\]
At fixed $K,L$, the omitted core tail is at most $C_{K,L}/\sqrt \Theta$
by its third-moment bound. Thus
\[
 \limsup_p\sup_{\substack{\rank P=q_p\\P\le I-E_{p,K,L}}}
 \E\int\lambda^2\,\ESD(S_P^K)(\dd\lambda)
 \le \int(\lambda^2\wedge \Theta)\,\mu(\dd\lambda)+C_{K,L}/\sqrt \Theta.
\]
Let $\Theta\to\infty$. This uses no moment convergence for the original
matrices. The nonnegative diagonal-column term in
\eqref{eq:inverse-spectral-second} is
\[
 \frac{\E\|C_Px_p^K\|^4}{Nq_p}
 =d_p\,\E\left(\frac{\|Px_p^K\|^2}{q_p}\right)^2,
\]
so
\[
 \limsup_p\sup_{\substack{\rank P=q_p\\P\le I-E_{p,K,L}}}
 \E\left(\frac{\|Px_p^K\|^2}{q_p}\right)^2\le\frac{m_{2,\mu}}{d}.
\]
Let $F_{p,K,L}=I-E_{p,K,L}$, of rank $m_p$, and take an independent
Haar rank-$q_p$ projection $U$ in its range.
For every vector, including zero, spherical second moments give
\[
 \E_U\left(\frac{\|Ux_p^K\|^2}{q_p}\right)^2
 =\frac{p^2(q_p+2)}{q_pm_p(m_p+2)}
       \left(\frac{\|F_{p,K,L}x_p^K\|^2}{p}\right)^2.
\]
The coefficient is
$(p/m_p)^2(1+2/q_p)/(1+2/m_p)\ge1$ because $q_p\le m_p\le p$.
Averaging the preceding uniform deterministic-projection bound therefore
bounds the expected squared complementary energy by $m_{2,\mu}/d$
in the limsup.

Write $r_p^K=\|x_p^K\|^2/p\le K$ and
$e_{p,K,L}=(x_p^K)^TE_{p,K,L}x_p^K/p$.
Since the complementary energy is $r_p^K-e_{p,K,L}$,
\[
 0\le (r_p^K)^2-(r_p^K-e_{p,K,L})^2
 =2r_p^Ke_{p,K,L}-e_{p,K,L}^2\le2Ke_{p,K,L}.
\]
Its expected upper bound vanishes in the ordered limit
$p\to\infty$, $L\to\infty$, at fixed $K$.
Moreover $\E(r_p^K)^2\to\int_{[0,K]}r^2\,\eta(\dd r)$ by weak
convergence at the continuity cutoff $K$.
At fixed $K$ the preceding bounds give the explicit chain
\[
 \begin{aligned}
 \int_{[0,K]}r^2\,\eta(\dd r)
 &=\lim_p\E(r_p^K)^2\\
 &\le \limsup_p\E(r_p^K-e_{p,K,L})^2
                  +2K\limsup_p\E e_{p,K,L}\\
 &\le m_{2,\mu}/d+2K\limsup_p\E e_{p,K,L}.
 \end{aligned}
\]
First send $L\to\infty$, still at fixed $K$, using the population-energy
tail bound. Then send $K\to\infty$ through continuity cutoffs.
Monotone convergence proves
\begin{equation}\label{eq:inverse-radius-moment-bootstrap}
 \int r^2\,\eta(\dd r)\le m_{2,\mu}/d<\infty.
\end{equation}

There are now deterministic cutoffs $K_p\to\infty$, increasing slowly
enough along this subsequence, such that
\[
 \widehat R_p=R_p\1_{\{R_p\le K_p\}}\Rightarrow\eta,\qquad
 \{\widehat R_p^2\}\text{ is uniformly integrable}.
\]
This is \cref{lem:slow-truncation} applied along the chosen subsequence.
The vectors $\widehat x_p=x_p\1_{\{R_p\le K_p\}}$ inherit the
small-subspace condition. Tightness gives $\Pp(R_p>K_p)\to0$, so
if $k_p$ columns are deleted, the rank inequality gives
$\dK(\ESD(S_P),\ESD(\widehat S_P))\le k_p/q_p$ for every $P$, and
$\E(k_p/q_p)=d_p^{-1}\Pp(R_p>K_p)\to0$.
Every projected ESD limit is therefore preserved.
The already proved case gives RQC and $\mu=\mu_{d,\eta}$ for
$\widehat x_p$. Since $\widehat x_p=x_p$ with probability tending to
one, RQC transfers in probability to $x_p$.
This works along every subsequential radius limit; injectivity again
gives full radius convergence and RQC.
Finally, apply the uniformly integrable case's spectral-moment calculation
to independent spherical columns with radius law $\nu$. For
$\int r^2\,\nu(\dd r)<\infty$ it gives second spectral moment
$(\int r\,\nu(\dd r))^2+d\int r^2\,\nu(\dd r)$.
Hence $\mu_{d,\nu}$ has finite second moment, so part (a) applies.
Injectivity of $\eta\mapsto\mu_{d,\eta}$ then gives $\eta=\nu$.
\end{proof}

The bootstrap uses finiteness of the \emph{limiting spectral} second moment,
not convergence of the original second spectral moments. Rare enormous
columns can prevent the latter. The first part of the proof also explains
the role of the geometric assumption: when the radii are uniformly
integrable, it is exactly the population-energy condition
\eqref{eq:population-energy-UI}. For infinite-second-moment target laws,
the bootstrap \eqref{eq:inverse-radius-moment-bootstrap} no longer supplies
the integrability needed by this argument.

\section{A conditional-isotropy alternative}
\label{sec:necessity}

\subsection{Radial reductions}

\begin{definition}[projected radial MP property]
\label{def:RMP-proj}
Let $\nu$ be a probability measure on $[0,\infty)$. $(x_p)$ has property
(RMP-$\Pi$) for $(c,\nu)$ if for every $\alpha\in(0,1]$ and every deterministic
sequence of orthogonal projections $P_p$ with $\rank P_p/p\to\alpha$,
$\ESD(S_{P_p})\Rightarrow\mu_{c\alpha,\nu}$ weakly in probability.
\end{definition}

The law in (RMP-$\Pi$) is $\mu_{c\alpha,\nu}$ with one and the same $\nu$ for
every $\alpha$; the projected samples have their own radii, and part of the
content of the definition is that these are not allowed to matter. Nothing is
assumed here about the law of $R_p$: that $R_p\Rightarrow\nu$ is part of the
conclusion of \cref{thm:necessity}.

The following alternative can recover the radius law through a conditional
second-moment hypothesis, using all proportional ranks.

\begin{definition}[radial reduction; approximate conditional isotropy]
\label{def:ACI}
A \emph{radial reduction} of $(x_p)$ is a pair $(\mathcal H_p,\rho_p)$ in which
$\mathcal H_p\subseteq\sigma(x_p)$ is a $\sigma$-field with
$\E[\|x_p\|^2\mid\mathcal H_p]<\infty$ almost surely and $\rho_p\ge0$ is
$\mathcal H_p$-measurable, such that the conditional second-moment matrix
$\Sigma_p=\E[x_px_p^{\mathsf T}\mid\mathcal H_p]$ satisfies
\begin{equation}
  \frac1p\big\|\Sigma_p-\rho_pI_p\big\|_1\xrightarrowp0.
  \label{eq:ACI}
\end{equation}
We say that $(x_p)$ satisfies (ACI) with radial law $\nu$ if it admits a radial
reduction with $\rho_p\Rightarrow\nu$. A radial reduction transfers to
independent copies: since $\mathcal H_p\subseteq\sigma(x_p)$, the family
$\mathcal B=\{B\in\mathcal B(\R^p):x_p^{-1}(B)\in\mathcal H_p\}$ is a
$\sigma$-field with $\mathcal H_p=x_p^{-1}(\mathcal B)$, and $\rho_p=g(x_p)$
for some $\mathcal B$-measurable $g\ge0$; for a copy $x_j$ of $x_p$ we write
$\mathcal H_j=x_j^{-1}(\mathcal B)$ and $\rho_j=g(x_j)$, so that
$\mathcal H_1,\dots,\mathcal H_N$ are independent.
\end{definition}

Two reductions matter. The \emph{canonical} one is $\mathcal H_p=\sigma(R_p)$,
$\rho_p=R_p$; there $\E[\|x_p\|^2\mid\mathcal H_p]=pR_p<\infty$ automatically,
$\Tr\Sigma_p=pR_p$, and \eqref{eq:ACI} says that $x_p$ is isotropic given its
own radius, in an approximate trace-norm sense. Exact conditional isotropy
$\Sigma_p=R_pI$ holds whenever the law of $x_p$ is invariant under a group of
orthogonal transformations that preserve $\|x\|$ and act irreducibly on $\R^p$
(a symmetric matrix commuting with such a group has a nonzero real eigenspace,
which is invariant and hence all of $\R^p$),
for instance under all sign changes and coordinate permutations; it holds for
the tensor model when the base law is symmetric (\cref{lem:tensor-ACI}), and for
elliptical laws. The \emph{trivial} reduction is
$\mathcal H_p=\{\emptyset,\Omega\}$, $\rho_p=1$, for which \eqref{eq:ACI} reads
$\frac1p\|\E x_px_p^{\mathsf T}-I\|_1\to0$: this is isotropy in the sense of
\cite{Yaskov2015b}, with $\nu=\delta_1$. Neither reduction subsumes the other
(\cref{rem:reductions}). The definition therefore allows any radial
reduction satisfying \eqref{eq:ACI}, rather than requiring either of these
two choices. The largest choice $\mathcal H_p=\sigma(x_p)$
is restrictive: there $\Sigma_p=x_px_p^{\mathsf T}$ has eigenvalues
$pR_p$ and $0$, so
$\frac1p\|\Sigma_p-\rho_pI\|_1=|R_p-\rho_p/p|+(1-1/p)\rho_p$, which tends to
$0$ only if $\rho_p\to0$ and $R_p\to0$; the reduction is then available only
when $\nu=\delta_0$, where $|Q_p(A)|\le2R_p$ makes (RQC) automatic and both
sides of \cref{thm:necessity} hold.

\begin{theorem}
\label{thm:necessity}
Let $c_n\to c\in(0,\infty)$, let $\nu$ be a probability measure on
$[0,\infty)$, and assume that $(x_p)$ satisfies (ACI) with radial law $\nu$
through a radial reduction $(\mathcal H_p,\rho_p)$. Then the following are
equivalent.
\begin{enumerate}[label=(\roman*)]
\item $R_p\Rightarrow\nu$ and (RQC).
\item (RMP-$\Pi$) for $(c,\nu)$.
\end{enumerate}
Moreover (ii) implies $R_p-\rho_p\to0$ in probability. The implication
(i)$\Rightarrow$(ii) uses no hypothesis on $(\mathcal H_p,\rho_p)$.
\end{theorem}

For the canonical reduction, the theorem says that for vectors that are approximately isotropic given their own radius, the radial
condition is necessary and sufficient for the radial MP law under all
proportional projections. Taking the trivial reduction and $\nu=\delta_1$
recovers the necessity half of Yaskov's characterization
\cite[Theorem~3.3]{Yaskov2015b}, in which $R_p\to1$ is a conclusion and not a
hypothesis; the converse half is recovered in probability, and almost surely
under the summability hypothesis of \cref{thm:sufficient} (\cref{sec:related}).

\begin{remark}[the two reductions are not comparable]
\label{rem:reductions}
For an elliptical vector with a nondegenerate radial law the canonical
reduction is exact while the trivial one is useless: if $\E R=\infty$ the
trivial reduction is not even admissible, since $\E[\|x_p\|^2]=\infty$; and if
$\E R<\infty$ then $\E xx^{\mathsf T}=(\E R)I$, so \eqref{eq:ACI} holds only
with the deterministic
$\rho_p=\E R$, whose weak limit $\delta_{\E R}$ is not $\nu$, and (RMP-$\Pi$)
for $(c,\delta_{\E R})$ is false. Conversely, let $p=2m$, let
$g^{(1)},g^{(2)}\sim N(0,I_m)$ be independent, let $\pi=(1+a_p)/2$ with
$p^{-1/2}\ll a_p\to0$, and let
\[
  x=\pi^{-1/2}\big(g^{(1)},0\big)\ \text{with probability }\pi,\qquad
  x=(1-\pi)^{-1/2}\big(0,g^{(2)}\big)\ \text{otherwise}.
\]
Then $\E xx^{\mathsf T}=I_p$ exactly, so the trivial reduction is exact with
$\nu=\delta_1$, and $R_p\to1$ in probability. The canonical reduction fails:
given the branch and the radius the direction is uniform on the sphere of the
active block, so $\Sigma_p=2R_p(\theta I_m\oplus(1-\theta)I_m)$ with
$\theta=\Pp(\text{first branch}\mid R_p)$, and
\[
 \frac1p\|\Sigma_p-R_pI\|_1=R_p|2\theta-1|.
\]
The two conditional laws of $R_p$
are shifted by $\asymp a_p$ against fluctuations of order $m^{-1/2}$, and
$a_p\sqrt m\to\infty$, so $\theta\to\1_{\{\text{first branch}\}}$ and
$\frac1p\|\Sigma_p-R_pI\|_1\to1$. Here (RQC) fails, for the projection onto the
first block, so \cref{thm:necessity} applied through the trivial reduction says
that (RMP-$\Pi$) for $(c,\delta_1)$ fails; the canonical reduction says
nothing.
\end{remark}

\Cref{ex:block} shows that (RMP-$\Pi$) cannot be weakened to the single
statement $\ESD(S)\Rightarrow\mu_{c,\nu}$: there are vectors satisfying (ACI)
for which the full matrix obeys the radial law and (RQC) fails.
When $R_p\Rightarrow\nu$, \cref{prop:law-level} shows that (RQC) is equivalent
to the statement that every proportional projection of $x$ has radial law
$\nu$. Theorem~\ref{thm:necessity-free} identifies these projected radius
laws from the spectral hypothesis without (ACI).

\subsection{The conditional argument}

For isotropic $x_p$, Yaskov's characterization
\cite[Theorem~3.3]{Yaskov2015b} makes \eqref{eq:yaskov-A} equivalent to the MP law
holding for $C_Px_p$ for every sequence of matrices $C_P$ with orthonormal
rows and proportional rank. The proof of necessity has two parts. The first,
\cite[Theorem~2.1]{Yaskov2016}, shows that the MP law for the full matrix forces
$\|x\|^2/p\to1$; applied to $C_Px$, this gives $x^{\mathsf T}Px/p-\rank P/p\to0$
for every proportional projection, and the layer-cake argument of
\cref{prop:rqc-equivalences} does the rest. In the radial setting the analogue
of the first part is vacuous for the full matrix ($\|x\|^2/p=R_p$ by
definition) and is the entire content for projections: one must show that the
radial law for $S_P$, with the parent $\nu$, forces $\|C_Px\|^2/\rank P$ to be
close to $R_p$, not merely to have the same limit law.

The proof of \cite[Theorem~2.1]{Yaskov2016} uses a strict Jensen inequality.
At a negative spectral parameter $z=-\eps$, the leave-one-out quadratic form
$Z$ of a fresh sample is nonnegative, the summands $Z/(1+Z)$ of the exact
identity are bounded, and the MP law fixes the limit of $\E[Z/(1+Z)]$. The
Gaussian model has the same limit and satisfies $Z\approx\E[Z\mid\text{matrix}]$,
so the two limits of $\E f(Z)$ and $\E f(\E[Z\mid\text{matrix}])$ coincide, and
the gap identity for $f(t)=t/(1+t)$ then forces $Z-\E[Z\mid\text{matrix}]\to0$.
Isotropy enters once: to compute $\E[Z\mid\text{matrix}]=\Tr(\text{matrix})$.
When the radius fluctuates, one conditions on a sub-field $\mathcal H$ of the
sample and the conditional mean of $Z$ becomes
$\Tr(M\,\E[xx^{\mathsf T}\mid\mathcal H])$; the argument closes when
this is close to $\rho\Tr M$ for an $\mathcal H$-measurable $\rho$.
Accordingly, \cref{thm:necessity} assumes the existence of a suitable radial
reduction, without fixing $\mathcal H=\sigma(R_p)$: Yaskov's isotropy is the case
$\mathcal H$ trivial, $\rho=1$, and \cref{rem:reductions} shows that it is not
a special case of conditioning on the radius. The different argument of
\cref{thm:necessity-free} removes (ACI) for arbitrary limiting radial laws,
provided the actual parent radius converges to the specified target law.

\Needspace{5\baselineskip}
\subsection{Conditional isotropy in the tensor model}

\begin{lemma}
\label{lem:tensor-ACI}
Let $1\le\ell\le n$ and $x_I=\prod_{i\in I}X_i$ for
$I\in\binom{[n]}{\ell}$, with $X_1,\dots,X_n$ i.i.d.
and $X\overset d=-X$. Then $\E[xx^{\mathsf T}\mid R_p]=R_pI$.
\end{lemma}

\begin{proof}
$R_p=\binom{n}{\ell}^{-1}\sum_I\prod_{i\in I}X_i^2$ is a symmetric function of
$(X_1^2,\dots,X_n^2)$. For $I\ne J$ pick $i\in I\triangle J$; the map
$X_i\mapsto-X_i$ preserves the joint law, preserves $R_p$, and changes the sign
of $x_Ix_J$, so $\E[x_Ix_J\mid R_p]=0$. For $I=J$, exchangeability of
$(X_1^2,\dots,X_n^2)$ and the symmetry of $R_p$ show that
$\E[x_I^2\mid R_p]$ does not depend on $I$; summing over $I$ gives
$\binom n\ell\E[x_I^2\mid R_p]=\E[pR_p\mid R_p]=pR_p$.
\end{proof}

The same argument applies to any law invariant under a group of signed
permutations (or any group of orthogonal maps preserving $\|x\|$) acting
irreducibly: the conditional second-moment matrix given $R_p$ commutes with the
group and is symmetric, so each of its real eigenspaces is invariant and
therefore all of $\R^p$; it is a multiple of $I$, with trace $pR_p$. For a nonsymmetric base
law the off-diagonal conditional means $\E[x_Ix_J\mid R_p]$ need not vanish,
and we do not know whether (ACI) holds.

\subsection{\texorpdfstring{Proof of \cref{thm:necessity}}{Proof of the necessity theorem}}

\emph{(i)$\Rightarrow$(ii).} Assume $R_p\Rightarrow\nu$ and (RQC). Tightness of
$(R_p)$ makes \cref{prop:rqc-equivalences} available, so \eqref{eq:phi} holds.
Let $P=P_p$ have rank $q$, $q/p\to\alpha\in(0,1]$, and $y=C_Px$. For a $q\times q$ matrix $A$ with $\|A\|\le1$,
$y^{\mathsf T}Ay=x^{\mathsf T}(C_P^{\mathsf T}AC_P)x$, $\|C_P^{\mathsf T}AC_P\|\le1$ and
$\Tr(C_P^{\mathsf T}AC_P)=\Tr A$, so
\[
  \frac{y^{\mathsf T}Ay}{q}-R_p\frac{\Tr A}{q}=\frac pq\,Q_p(C_P^{\mathsf T}AC_P)\xrightarrowp0
\]
uniformly over $\|A\|\le1$: the probability that the left side exceeds
$\eps$ is at most $\phi_p(\eps q/p)$, where $\phi_p$ is the uniform modulus of
the parent in \eqref{eq:phi}. The pairs $(y_j,R_j)=(C_Px_j,\|x_j\|^2/p)$ are
i.i.d.\ and $R_j\Rightarrow\nu$, so \cref{rem:general-radius} applies with
$q/N\to c\alpha$ and gives $\ESD(S_P)\Rightarrow\mu_{c\alpha,\nu}$.

\emph{(ii)$\Rightarrow$(i).} Fix a projection sequence with
$q=\rank P_p$, $q/p\to\alpha\in(0,1]$; the case $\alpha=0$ and non-convergent
rank ratios are treated at the end. We show
\begin{equation}
  \frac{\|C_Px\|^2}{q}-R_p\xrightarrowp0.
  \label{eq:N1}
\end{equation}
Let $y_1,\dots,y_N,y_{N+1}$ be the projections of $N+1$ independent copies of
$x$, with $\rho_1,\dots,\rho_{N+1}$ the corresponding copies of $\rho_p$ and
$R_1,\dots,R_{N+1}$ the parent radii, and put
\begin{gather*}
  A=\sum_{k=1}^Ny_ky_k^{\mathsf T},\qquad B=A+y_{N+1}y_{N+1}^{\mathsf T},\qquad
  M_\eps=(A+\eps NI_q)^{-1},\\ Z=y_{N+1}^{\mathsf T}M_\eps y_{N+1},\qquad f(t)=\frac t{1+t}.
\end{gather*}

\emph{Step (i): the exact identity and its limit.}
$q=\Tr[(B+\eps N)(B+\eps N)^{-1}]=\sum_{k=1}^{N+1}y_k^{\mathsf T}(B+\eps N)^{-1}y_k+\eps N\Tr(B+\eps N)^{-1}$,
and by the Sherman--Morrison formula $y_{N+1}^{\mathsf T}(B+\eps N)^{-1}y_{N+1}=f(Z)$.
The $N+1$ samples are exchangeable and every term is bounded ($f\le1$,
$\eps N\Tr(B+\eps N)^{-1}\le q$), so taking expectations
\[
  q=(N+1)\E f(Z)+\eps N\,\E\Tr(B+\eps NI)^{-1}.
\]
Also $0\le\Tr(A+\eps N)^{-1}-\Tr(B+\eps N)^{-1}=\frac{y_{N+1}^{\mathsf T}M_\eps^2y_{N+1}}{1+Z}\le\|M_\eps\|\le(\eps N)^{-1}$.
Dividing by $N$ and writing $\Tr(A+\eps N)^{-1}=\frac qN\,s_P(-\eps)$ with
$s_P(-\eps)=\frac1q\Tr(S_P+\eps)^{-1}\in(0,\eps^{-1}]$,
\[
  \frac qN=\E f(Z)+\eps\,\frac qN\,\E s_P(-\eps)+O(N^{-1}).
\]
By (RMP-$\Pi$), $s_P(-\eps)\to s(-\eps):=\int(t+\eps)^{-1}\mu_{c\alpha,\nu}(\dd t)$
in probability (the integrand is bounded and continuous on $[0,\infty)$), hence
in $L^1$. Therefore
\begin{equation}
  \E f(Z)\longrightarrow c\alpha\big(1-\eps s(-\eps)\big).
  \label{eq:EfZ}
\end{equation}
Equation \eqref{eq:RMP} for $\mu_{c\alpha,\nu}$ extends from $\C_+$ to
$z=-\eps<0$. Indeed $s$ is the Stieltjes transform of a probability measure on
$[0,\infty)$, so $s(\C_+)\subset\C_+$, $s(\C_-)\subset\C_-$ and $s(t)>0$ for
$t<0$; hence $d(z):=\operatorname{dist}\big(c\alpha s(z),(-\infty,0]\big)>0$ for
every $z\in\C\setminus[0,\infty)$, and
$|r/(1+c\alpha rs(z))|\le1/d(z)$ for all $r\ge0$, locally uniformly in $z$. So
both sides of \eqref{eq:RMP} are analytic on $\C\setminus[0,\infty)$, and they
agree on $\C_+$. So
$1-\eps s=s\int r(1+c\alpha rs)^{-1}\nu(\dd r)$ at $s=s(-\eps)$, that is
\begin{equation}
  c\alpha\big(1-\eps s(-\eps)\big)=\int\frac{c\alpha rs(-\eps)}{1+c\alpha rs(-\eps)}\nu(\dd r)
  =\E f\big(c\alpha R\,s(-\eps)\big),\qquad R\sim\nu.
  \label{eq:EfW-target}
\end{equation}

\emph{Step (ii): the conditional mean.} Let
$\mathcal G=\sigma(x_1,\dots,x_N)\vee\mathcal H_{N+1}$, where $\mathcal H_{N+1}$
is the sub-field of the radial reduction attached to the sample $x_{N+1}$.
The field $\mathcal H_{N+1}$ is independent of the first $N$ samples.
Since $M_\eps$ is bounded and $\mathcal G$-measurable,
\[
  W:=\E[Z\mid\mathcal G]=\Tr\big(M_\eps\,C_P\Sigma_{N+1}C_P^{\mathsf T}\big),\qquad
  \Sigma_{N+1}=\E[x_{N+1}x_{N+1}^{\mathsf T}\mid\mathcal H_{N+1}],
\]
which is finite because $\E[\|x_{N+1}\|^2\mid\mathcal H_{N+1}]<\infty$. Write
$W=W_0+\Delta$ with
$W_0=\rho_{N+1}\Tr M_\eps=\rho_{N+1}\frac qNs_P(-\eps)$.
Since $\|C_P\Gamma C_P^{\mathsf T}\|_1\le\|\Gamma\|_1$ for any $\Gamma$,
\[
  |\Delta|\le\|M_\eps\|\,\|\Sigma_{N+1}-\rho_{N+1}I\|_1\le\frac{p}{\eps N}\cdot\frac1p\|\Sigma_{N+1}-\rho_{N+1}I\|_1\xrightarrowp0
\]
by (ACI). Here $q/N=(q/p)(p/N)\to c\alpha$, $\rho_{N+1}\Rightarrow\nu$ and
$s_P(-\eps)\to s(-\eps)$, a constant, so
$W_0\Rightarrow c\alpha Rs(-\eps)$ with $R\sim\nu$ by Slutsky's theorem and
$\E f(W_0)\to\E f(c\alpha Rs(-\eps))$; as $|f(W)-f(W_0)|\le|\Delta|$ and $f\le1$,
also $\E f(W)\to\E f(c\alpha Rs(-\eps))$. With \eqref{eq:EfZ}--\eqref{eq:EfW-target},
\begin{equation}
  \E f(W)-\E f(Z)\longrightarrow0.
  \label{eq:jensen-gap-vanishes}
\end{equation}

\emph{Step (iii): the Jensen gap.} For $Z\ge0$ and $W=\E[Z\mid\mathcal G]$,
\[
  f(Z)-f(W)=\frac{Z-W}{(1+Z)(1+W)}=\frac{Z-W}{(1+W)^2}-\frac{(Z-W)^2}{(1+Z)(1+W)^2}.
\]
By the tower property for nonnegative variables,
$\E[Z/(1+W)^2]=\E[W/(1+W)^2]\le1/4$, so $\E[(Z-W)/(1+W)^2]=0$ and
\begin{equation}
  \E f(W)-\E f(Z)=\E\frac{(Z-W)^2}{(1+Z)(1+W)^2}\ge0.
  \label{eq:jensen-gap}
\end{equation}
By \eqref{eq:jensen-gap-vanishes} the right side tends to zero, so
$(Z-W)^2/((1+Z)(1+W)^2)\to0$ in probability. Both $Z$ and $W$ are tight, and
neither argument uses the radius: $W=W_0+\Delta$ with
$W_0=\rho_{N+1}\frac qNs_P(-\eps)$ convergent in law and $\Delta\to0$, so $W$ is
tight; and $Z\ge0$ has $\Pp(Z>L)\le\E[\min(1,W/L)]$ by conditional Markov, so
$Z$ is tight as well.
On $\{Z\le L,W\le L\}$, $(Z-W)^2\le(1+L)^3(Z-W)^2/((1+Z)(1+W)^2)$. Hence
$Z-W\to0$ and, with Step (ii), $Z-W_0\to0$ in probability. Multiplying by
$\eps N/q$ and writing $T_\eps=(S_P/\eps+I)^{-1}$,
\begin{equation}
  J(\eps):=\frac1q\Big(y^{\mathsf T}T_\eps y-\rho\Tr T_\eps\Big)\xrightarrowp0
  \qquad\text{for every }\eps>0,
  \label{eq:Jeps}
\end{equation}
where $y=y_{N+1}$, $\rho=\rho_{N+1}$, and $S_P=A/N$ is independent of
$(y,\rho)$.
Choose $\eps_n\to\infty$ so slowly that $J(\eps_n)\to0$ in probability (for
each $k$ pick $n_k$ with $\Pp(|J(k)|>1/k)<1/k$ for $n\ge n_k$ and set
$\eps_n=k$ on $[n_k,n_{k+1})$).

\emph{Step (iv): removing the resolvent.} Diagonalize
$S_P=\sum_{k=1}^q\lambda_ke_ke_k^{\mathsf T}$. Since $\frac1q\sum_k(y,e_k)^2=\|y\|^2/q$,
\[
  J(\eps_n)-\Big(\frac{\|y\|^2}q-\rho\Big)=-\frac1q\sum_{k=1}^q\big((y,e_k)^2-\rho\big)\frac{\lambda_k/\eps_n}{1+\lambda_k/\eps_n}.
\]
Both $\rho$ and $\|y\|^2/q$ are tight. The first converges in law. For the
second, $\E[\|y\|^2/q\mid\mathcal G]\le\rho+\xi_n$ with
$\xi_n=\frac1q\|\Sigma-\rho I\|_1=\frac pq\cdot\frac1p\|\Sigma-\rho I\|_1\to0$
in probability, so conditional Markov and, for $L\ge1$, the splitting
\[
  \E\Big[\min\Big(1,\frac{\rho+\xi_n}L\Big)\Big]
  \le\E\Big[\min\Big(1,\frac{2\rho}L\Big)\Big]+\E\big[\min(1,2\xi_n)\big]
\]
give it: the second term tends to zero by bounded convergence, and the first is
small uniformly in $n$ because $\rho$ is tight.
Put $K_n=\eps_n^{1/2}$ and split the sum according to $\lambda_k\le K_n$ or
$\lambda_k>K_n$. For $\lambda_k\le K_n$ the last factor is at most
$K_n/\eps_n=\eps_n^{-1/2}$, so that part is bounded by
$\eps_n^{-1/2}(\|y\|^2/q+\rho)\to0$ in probability. For
$\mathcal K=\{k:\lambda_k>K_n\}$, the factor is at most $1$ and, conditionally
on $\mathcal G$ (recall $y$ is independent of $S_P$),
\[
  \E\Big[\frac1q\sum_{k\in\mathcal K}\big((y,e_k)^2+\rho\big)\Bigm|\mathcal G\Big]
  =\frac1q\sum_{k\in\mathcal K}\big(e_k^{\mathsf T}C_P\Sigma C_P^{\mathsf T}e_k+\rho\big)
  \le\frac{2\rho|\mathcal K|}{q}+\frac1q\|\Sigma-\rho I\|_1 ,
\]
where the last step uses $e_k^{\mathsf T}C_P\Sigma C_P^{\mathsf T}e_k=\rho+e_k^{\mathsf T}Xe_k$
with $X=C_P(\Sigma-\rho I)C_P^{\mathsf T}$ and
$\sum_{k\in\mathcal K}|e_k^{\mathsf T}Xe_k|=\Tr(UX)\le\|X\|_1\le\|\Sigma-\rho I\|_1$
for the contraction $U=\sum_{k\in\mathcal K}\operatorname{sgn}(e_k^{\mathsf T}Xe_k)e_ke_k^{\mathsf T}$.
Here $|\mathcal K|/q=\ESD(S_P)((K_n,\infty))\to0$ in probability: given
$\delta>0$, pick a continuity point $K$ of $\mu_{c\alpha,\nu}$ with
$\mu_{c\alpha,\nu}((K,\infty))<\delta$; for $n$ large $K_n\ge K$ and
$\ESD(S_P)((K_n,\infty))\le\ESD(S_P)((K,\infty))\to\mu_{c\alpha,\nu}((K,\infty))<\delta$,
so no rate for $K_n$ is needed. By conditional Markov the large-eigenvalue part
tends to zero in probability, and
\begin{equation}
  \frac{\|C_Px\|^2}{q}-\rho_p\xrightarrowp0
  \quad\text{whenever }\rank P_p/p\to\alpha\in(0,1].
  \label{eq:N0}
\end{equation}

\emph{The radius.} Steps (i)--(iv) never used $\alpha<1$, so
\eqref{eq:N0} may be applied to $P_p=I_p$, where $q=p$ and $C_P=I$:
\[
  R_p-\rho_p=\frac{\|x\|^2}{p}-\rho_p\xrightarrowp0 .
\]
In particular $R_p\Rightarrow\nu$, and \eqref{eq:N0} becomes \eqref{eq:N1} for
every proportional projection sequence.

\emph{Other ranks.} If $\rank P_p/p\to0$, then
$R_p\rank(P_p)/p\to0$ by tightness. For fixed $\beta\in(0,1)$ choose
projections $P_p'\ge P_p$ of rank $\lceil\beta p\rceil$; by \eqref{eq:N1} for
$P'$ we get
$0\le x^{\mathsf T}P_px/p\le x^{\mathsf T}P'_px/p=\beta R_p+o_p(1)$, and
$\beta\downarrow0$ gives $Q_p(P_p)\to0$. If $\rank P_p/p\to1$ use $I-P_p$ and
$Q_p(P)=-Q_p(I-P)$. If $\rank P_p/p$ has no limit, take a subsequence along
which it converges to some $\alpha'$, complete it to a full sequence with
$\rank/p\to\alpha'$ by inserting arbitrary projections of rank
$\lfloor\alpha'p\rfloor$ at the remaining indices, apply the above to the
completed sequence, and use the subsequence criterion for convergence in
probability. This proves (RQC) in the form of \cref{def:RQC}, and
\cref{prop:rqc-equivalences} gives the uniform form.
\qed

\subsection{What the argument gives without conditional isotropy}

Step (i) and the gap identity \eqref{eq:jensen-gap} of Step (iii) use no
hypothesis on the conditional law of $x$; the passage from
\eqref{eq:jensen-gap} to $Z-W\to0$ does use one, through the tightness of $W$.
They give $\E f(W)-\E f(Z)=\E\frac{(Z-W)^2}{(1+Z)(1+W)^2}$ with
$W=\Tr(M_\eps C_P\Sigma_{N+1}C_P^{\mathsf T})$, and
$\E f(Z)\to\int f(c\alpha rs)\dd\nu$.
Without (ACI) nothing identifies $\lim\E f(W)$, and the gap identity yields
only $\liminf\E f(W)\ge\int f(c\alpha rs)\dd\nu$, which is Jensen's inequality
and carries no information. So (ACI) enters only through the deviation
$X=C_P(\Sigma_p-\rho_pI)C_P^{\mathsf T}$: in Step~(ii) to identify
$\lim\E f(W)$, and in Step~(iv) to control $\frac1q\Tr X$ and the selected
diagonal of $X$. It suffices for this proof to assume that, with
$\lambda_k,e_k$ the eigenvalues and eigenvectors of $S_P$,
\[
  \frac1q\Tr\big((S_P+\eps)^{-1}X\big)\to0,\qquad
  \frac1q\Tr X\to0,\qquad
  \frac1q\sum_{k:\lambda_k>K}\big|e_k^{\mathsf T}Xe_k\big|\to0
\]
in probability for every $\eps>0$ and every $K>0$; these are what steps (ii)
and (iv) need.

For the canonical reduction, the condition
\[
 \frac1p\|P(\Sigma_p-R_pI)P\|_1\to0
\]
for every deterministic proportional projection sequence is equivalent to
(ACI). By the argument of
\cref{lem:uniform}(b),(c), applied to the functional
$\frac1p\|PYP\|_1$ in place of $Q_p$ on the class of rank-$\lfloor\alpha p\rfloor$
projections, the deterministic statement is uniform and therefore holds for a
Haar distributed $P$ independent of $x$. Write $Y=\Sigma_p-R_pI$, so $\Tr Y=0$.
Averaging over $P$, $O(p)$-equivariance gives
$\E[PYP\mid Y]=a_pY+b_p(\Tr Y)I$ with
$a_p=(q/p)^2+O(1/p)\to\alpha^2>0$, and
$\|PYP\|_1\ge\Tr(\operatorname{sgn}(Y)PYP)$, so
\[
 \E[\|PYP\|_1\mid Y]\ge a_p\|Y\|_1.
\]
Put $V=\|Y\|_1/p$ and $X_P=\|PYP\|_1/p$. Since $0\le X_P\le V$,
the conditional mean bound gives
$V\Pp(X_P>\varepsilon\mid Y)\ge a_pV-\varepsilon$.
Consequently
\[
 \Pp(V>2\varepsilon/a_p)\le \frac2{a_p}\Pp(X_P>\varepsilon)\to0.
\]
As $a_p\to\alpha^2>0$, this proves \eqref{eq:ACI} without a radial
tightness assumption.

The characterization without (ACI) can be expressed in terms of radial laws. Put
$R_P=\|C_Px\|^2/\rank P$, the radius of the projected sample.

\begin{proposition}\label{prop:law-level}
Assume $R_p\Rightarrow\nu$ and fix any $\alpha\in(0,1)$.
Then (RQC) holds if and only if $R_{P_p}\Rightarrow\nu$ for every
deterministic projection sequence of rank $\lfloor\alpha p\rfloor$.
In that case the same statement holds at every proportional rank.
\end{proposition}
\begin{proof}
RQC gives $R_P-R_p=Q_p(P)\,p/\rank P\to0$ at every proportional rank.
Conversely, the asserted weak convergence at one rank gives
$\E\psi(R_{P_p})-\E\psi(R_p)\to0$ for $\psi(t)=t/(1+t)$.
Apply \cref{cor:one-transform}, whose geometric proof uses no spectral
hypothesis or conditional isotropy.
\end{proof}

Thus (RQC) is the statement that every proportional projection of $x$ has the
same radial law as $x$ itself. Theorem~\ref{thm:necessity-free} shows that the
projected spectra force this statement for arbitrary $\nu$ once
$R_p\Rightarrow\nu$ is known. That assumption cannot be dropped:
\cref{ex:spike} satisfies (RMP-$\Pi$) for $(c,\delta_1)$ while
$R_p\Rightarrow\law(1+Z^2)$ and (RQC) fails.

The single-projection version is false. Let $q=N$, let $D$ be deterministic
positive semidefinite with $\sup_q\|D\|<\infty$, $\Tr D/q\to1$ and
$\ESD(D)\Rightarrow\rho_D\ne\delta_1$, and let $y=D^{1/2}g$ with
$g\sim N(0,I_q)$. Its radial law is $\delta_1$: $\|y\|^2/q=g^{\mathsf T}Dg/q$
has mean $\Tr D/q\to1$ and variance $2\Tr D^2/q^2\to0$. Its spectral law is
$\mu_{1,\rho_D}$: at $q=N$ the matrices $\frac1ND^{1/2}\mathbf G\mathbf G^{\mathsf T}D^{1/2}$
and $\frac1N\mathbf G^{\mathsf T}D\mathbf G=\frac1N\sum_{i\le q}t_iw_iw_i^{\mathsf T}$,
with $t_i$ the eigenvalues of $D$ and $w_i$ i.i.d.\ $N(0,I_N)$, have the same
characteristic polynomial. The second is a weighted Wishart matrix with
deterministic weights, not a sample covariance matrix of i.i.d.\ vectors, so
\cref{thm:sufficient} does not apply to it; its limit is
\cite[Theorem~1.1]{SilversteinBai1995}, and comparing that equation with
\eqref{eq:RMP} at $c=1$ gives $\rho_D\boxtimes\MP_1=\mu_{1,\rho_D}$. Since
$\rho_D\ne\delta_1$ and $\nu\mapsto\mu_{1,\nu}$ is injective
(\cref{prop:limit-law}), $\mu_{1,\rho_D}\ne\mu_{1,\delta_1}$: the spectral law
of one projected model does not determine its radial law.
The characterization uses all deterministic orientations of one proportional
rank and the matching actual parent law. The bounded-transform criterion
in \cref{cor:one-transform} isolates the energy information recovered by
its spectral argument.

\section{Examples and counterexamples}
\label{sec:examples}

In this section (ACI) always refers to the canonical radial reduction of
\cref{def:ACI}: the sub-field is $\sigma(R_p)$, the scalar is $\rho_p=R_p$,
and $\Sigma_p=\E[xx^{\mathsf T}\mid R_p]$. The first four examples satisfy (RQC), with independent, functional,
or implicit radial-angular coupling.
The remaining ones fail (RQC) in different ways and separate the notions
involved.

\subsection{Models satisfying the radial condition}

\begin{example}[spherical direction, arbitrary radius]
\label{ex:spherical}
Let $x=\sqrt{pR}\,u$ with $u$ uniform on $S^{p-1}$ and $R\ge0$ any random
variable on the same space, with $R_p=R\Rightarrow\nu$. Nothing is assumed about
the joint law of $(R,u)$. For a real symmetric $A$ with $\|A\|\le1$,
$\E\,u^{\mathsf T}Au=\Tr A/p$ and
\[
  \Var(u^{\mathsf T}Au)=\frac{2}{p(p+2)}\Big(\Tr A^2-\frac{(\Tr A)^2}{p}\Big)\le\frac{2}{p+2},
\]
using $\E u_i^4=3/(p(p+2))$ and $\E u_i^2u_j^2=1/(p(p+2))$ for $i\ne j$. Since
$Q_p(A)=R\,(u^{\mathsf T}Au-\Tr A/p)$,
$\Pp(|Q_p(A)|>\eps)\le\Pp(R>M)+2M^2/(\eps^2(p+2))$ uniformly in $A$, and (RQC)
holds. The canonical reduction need not satisfy (ACI): for nonatomic $\nu$
there is a measure-space isomorphism $T$, modulo null sets, of
$S^{p-1}$ onto a Borel subset of $[0,\infty)$ carrying the uniform law to
$\nu$, since both are nonatomic standard probability spaces; with $R=T(u)$ one
has $\sigma(R)=\sigma(u)$ up to null sets, so $u$ is $\sigma(R)$-measurable,
$\Sigma_p=pR\,uu^{\mathsf T}$, and
$\frac1p\|\Sigma_p-R_pI\|_1=2R(1-1/p)\not\to0$. The conclusion of
\cref{thm:sufficient} is unaffected. When $R\perp u$ this is the elliptical
model of \cite{ElKaroui2009,PajorPastur2009}.
A concrete dependent instance is $R=\Phi(\sqrt p\,u_1)$
for a bounded continuous $\Phi>0$, so that $R$ is a function of the first
coordinate of the direction and $R_p\Rightarrow\law\Phi(Z)$, $Z\sim N(0,1)$.
\end{example}

\begin{example}[self-modulated Gaussian]
\label{ex:endogenous}
Let $g\sim N(0,I_p)$, let $h:\R\to\R$ satisfy $\E h(g_1)=0$, $\E h(g_1)^2=1$,
fix $\sigma>0$, and put
\[
  Y_p=\frac1{\sqrt p}\sum_{i=1}^ph(g_i),\qquad
  \Lambda_p=\exp\Big(\sigma Y_p-\frac{\sigma^2}2\Big),\qquad x=\sqrt{\Lambda_p}\,g .
\]
The radius is a nonlinear statistic of the same coordinates that determine the
direction. By the central limit theorem and the law of large numbers,
$R_p=\Lambda_p\|g\|^2/p\Rightarrow\nu=\law\exp(\sigma Z-\sigma^2/2)$, the
lognormal family of the tensor model with $\sigma^2$ in place of $\lambda v$.
For a real symmetric contraction $A$,
\[
  Q_p(A)=\Lambda_p\Big(\frac{g^{\mathsf T}Ag-\Tr A}{p}-\frac{\Tr A}{p}\Big(\frac{\|g\|^2}{p}-1\Big)\Big),
\]
and $\Var(g^{\mathsf T}Ag)=2\Tr A^2\le2p$, so both terms are $O_p(p^{-1/2})$
uniformly in $A$, $\Lambda_p$ is tight, and (RQC) holds. \Cref{thm:sufficient}
gives $\ESD(S)\Rightarrow\mu_{c,\nu}$.

The dependence between radius and direction depends on $h$. If
$h(t)=(t^2-1)/\sqrt2$ then $Y_p=(\|g\|^2-p)/\sqrt{2p}$ is a function of $\|g\|$,
which is independent of $g/\|g\|$, so $R_p\perp u$ and the model reduces
to the standard independent scale-mixture setting. If
$h(t)=t$ then $Y_p=\langle g,e\rangle$ with $e=\1/\sqrt p$, and
$\Lambda_p=\exp(\sigma\|g\|\langle u,e\rangle-\sigma^2/2)$ is asymptotically a
function of the direction $u=g/\|g\|$. In particular,
$\Cov(R_p,\,p\langle u,e\rangle^2)\to\E[e^{\sigma Z-\sigma^2/2}Z^2]-1=\sigma^2\ne0$,
whereas a scale mixture with the same marginals has covariance zero.

(ACI) holds for $h(t)=t$, and its verification shows what the conditional
second-moment matrix looks like when the radius is a function of the
direction. Write $g=Y_pe+g_\perp$ with $g_\perp\sim N(0,I-ee^{\mathsf T})$
independent of $Y_p\sim N(0,1)$. Given $(Y_p,\|g_\perp\|)$ the direction of
$g_\perp$ is uniform on the unit sphere of $e^\perp$, so
\[
  \E[xx^{\mathsf T}\mid Y_p,\|g_\perp\|]
  =\Lambda_p\Big(Y_p^2ee^{\mathsf T}+\frac{\|g_\perp\|^2}{p-1}(I-ee^{\mathsf T})\Big),
\]
and $R_p$ is $\sigma(Y_p,\|g_\perp\|)$-measurable, so the tower property gives
$\Sigma_p=a\,ee^{\mathsf T}+b\,(I-ee^{\mathsf T})$ with
$a=\E[\Lambda_pY_p^2\mid R_p]$ and $b=\E[\Lambda_p\|g_\perp\|^2\mid R_p]/(p-1)$.
From $\Tr\Sigma_p=pR_p$ one gets $(p-1)(b-R_p)=R_p-a$, whence
$\|\Sigma_p-R_pI\|_1=|a-R_p|+(p-1)|b-R_p|=2|a-R_p|$ and
\[
  \E\frac1p\|\Sigma_p-R_pI\|_1\le\frac2p\big(\E[\Lambda_pY_p^2]+\E R_p\big)
  =\frac2p\Big(2+\sigma^2+\frac{\sigma^2}p\Big)\longrightarrow0,
\]
using $\E[\Lambda_pY_p^2]=1+\sigma^2$ and $\E R_p=1+\sigma^2/p$. So
\cref{thm:necessity} applies as well. For $h(t)=(t^3-3t)/\sqrt6$ the radius
depends on the direction through $\sum_iu_i^3$, a genuinely nonlinear
functional, and the verification of $R_p\Rightarrow\nu$ and of (RQC) is the
same; we do not claim (ACI) there, since $\E\Lambda_p=\infty$ for every $p$ and
the argument above is unavailable.
\end{example}

\begin{example}[critical tensor features]
\label{ex:tensor}
Let $X$ be centered with unit variance and $\E X^4<\infty$, $v=\E X^4-1$, and
$x_I=\prod_{i\in I}X_i$ for $I\in\binom{[n]}{\ell}$, $p=\binom n\ell$, $\ell^2/n\to\lambda\in[0,\infty)$.
Reference \cite[Lemma~4.2]{Xie2026} proves
$R_p\Rightarrow\law\exp(\sqrt{\lambda v}Z-\lambda v/2)$, and
\cite[Theorem~6.2]{Xie2026} proves
$\sup_{\|A\|\le1}\E|Q_p(A)|^2\to0$, which implies (RQC) by Chebyshev. The
resulting critical tensor ESD is already proved in
\cite[Theorem~3.1]{Xie2026}, using the sufficient principle in its
Theorem~3.2; that tensor conclusion is not a result of the present paper. Here
the model serves only as an application of the converse theory. The coupling
between $R_p$ and $x/\|x\|$ is implicit: both are functions of
$X_1,\dots,X_n$, and we have no useful description of the conditional law of
the direction given the radius. When $X\overset d=-X$,
\cref{lem:tensor-ACI} gives exact conditional isotropy, and
\cref{thm:necessity} shows that for such base laws the radial condition
(RQC), which \cite{Xie2026} establishes in the stronger $L^2$ form, is not only
sufficient but necessary for the spectral conclusion under projections.
\end{example}

\begin{example}[random low-dimensional support]
\label{ex:subspace}
Let $K\subset\{1,\dots,p\}$ be a uniformly random $k$-subset, $u$ uniform on
the unit sphere of $\mathrm{span}\{e_i:i\in K\}$, $R\sim\nu$ independent, and
$x=\sqrt{pR}\,u$, with $k=k_p\to\infty$, $k/p\to0$. Then $R_p=R$ and
$\Sigma_p=RI$. For any orthogonal projection $P$ of rank $m$, conditioning on
$K$ and using the sphere moments $\E u_i^4=3/(k(k+2))$,
$\E u_i^2u_j^2=1/(k(k+2))$,
\[
  \E\Big[\Big(u^{\mathsf T}Pu-\frac mp\Big)^2\Big]
  =\E\Var(u^{\mathsf T}Pu\mid K)+\Var\Big(\frac1k\sum_{i\in K}P_{ii}\Big)
  \le\frac{2}{k+2}+\frac1{4k},
\]
the second term being the variance of a sample mean of $k$ values in $[0,1]$
drawn without replacement. Thus (RQC) holds, uniformly over projections, as
soon as $k\to\infty$, however slowly, and $\ESD(S)\Rightarrow\mu_{c,\nu}$. For
$k=1$ this is \cref{ex:sparse-radius} with the classical compound Poisson limit.
The passage from the classical to the free law is governed by the dimension of
the support of one sample, not by its radius. For fixed $k$ the bound above no
longer tends to zero and (RQC) fails whenever $\nu\ne\delta_0$: for $P$ the
projection onto the first $\lfloor p/2\rfloor$ coordinates,
$u^{\mathsf T}Pu=\sum_{i\in K\cap[p/2]}u_i^2$ converges in law to a
nondegenerate limit, so $Q_p(P)$ does not vanish. For fixed $k\ge2$ we have not
identified the limit of $\ESD(S)$.
\end{example}

\subsection{Models failing the radial condition}

The following examples separate radial convergence, isotropy, and RQC.
Some have a nonradial spectral limit; the block example has the correct
full radial law but fails under projection. A recurring theme is the
contrast between the \emph{free} compound Poisson law $\mu_{c,\nu}$ and the
\emph{classical} compound Poisson law with the same rate and jump law: when
$\nu$ has finite moments through the order under discussion, the free
cumulants of $\mu_{c,\nu}$ are $\kappa_k=c^{k-1}a_k$, where
$a_k=\int r^k\,\nu(\dd r)$. This follows first for bounded radii from
the $R$-transform and then by truncation; no convergent power series is
asserted for an unbounded law. These are also the
classical cumulants of the law of $c\sum_{i\le K}Y_i$ with $K\sim\mathrm{Poi}(1/c)$
and $Y_i\sim\nu$ i.i.d. Since noncrossing and all partitions agree up to
$k=3$, the two laws share their first three moments and differ at the fourth by
$\kappa_2^2=c^2a_2^2$. A moment-free distinction is the atom at zero:
$e^{-\rho/c}$ for the classical law and $(1-\rho/c)_+$ for the free one, where
$\rho=\nu((0,\infty))$. Second-moment tests cannot separate them.

\begin{example}[sparse vector: unit radius, Poisson spectrum]
\label{ex:sparse}
Let $x=\sqrt p\,e_J$ with $J$ uniform on $\{1,\dots,p\}$. Then $R_p=1$
exactly, $\nu=\delta_1$, and $\Sigma_p=\E[pe_Je_J^{\mathsf T}]=I$, so (ACI)
holds exactly. For the coordinate
projection $P$ onto the first $\lfloor p/2\rfloor$ coordinates,
$Q_p(P)=\1_{\{J\le p/2\}}-\lfloor p/2\rfloor/p\Rightarrow\frac12(\delta_{1/2}+\delta_{-1/2})$,
so (RQC) fails. The matrix $S=c_n\diag(n_1,\dots,n_p)$ with
$n_i=\#\{j:J_j=i\}\sim\mathrm{Bin}(N,1/p)$, and for bounded continuous $f$,
\[
 \E\frac1p\sum_if(c_nn_i)\to\E f(c\,\mathrm{Poi}(1/c)),\qquad
 \Var\left(\frac1p\sum_if(c_nn_i)\right)\to0.
\]
The variance is at most
$\|f\|_\infty^2/p+|\Cov(f(c_nn_1),f(c_nn_2))|$, which tends to zero
because $(n_1,n_2)$ converges to a pair of independent Poisson variables.
Hence $\ESD(S)\Rightarrow\law(c\,\mathrm{Poi}(1/c))$ in probability, a purely
atomic law with atom $e^{-1/c}$ at $0$, whereas $\mu_{c,\delta_1}=\MP_c$ has an
absolutely continuous part and atom $(1-1/c)_+$. Radial convergence, isotropy,
exchangeability of coordinates and conditional isotropy all hold; only (RQC)
fails, and the spectrum is not MP.
\end{example}

\begin{example}[sparse direction, random radius]
\label{ex:sparse-radius}
Let $x=\sqrt{pR}\,e_J$ with $R\sim\nu$ independent of $J$. Then $R_p=R$,
(ACI) holds exactly, (RQC) fails as above whenever $\nu\ne\delta_0$, and
$S=c_n\diag\big(\sum_{j:J_j=i}R_j\big)_i$, so by the same law of large numbers
$\ESD(S)$ converges to the law of $c\sum_{i\le K}Y_i$, $K\sim\mathrm{Poi}(1/c)$,
$Y_i\sim\nu$: the classical compound Poisson law with rate $1/c$ and jump law
$\law(cR)$. The radial law $\mu_{c,\nu}$ is the free compound Poisson law with
the same parameters. They differ whenever $\nu\ne\delta_0$ (atoms at zero
$e^{-\rho/c}\ne(1-\rho/c)_+$ for $\rho>0$). Radial convergence to $\nu$ leaves
the spectrum undetermined between these two laws, and \cref{ex:subspace} shows
how the effective dimension of the support decides between them.
\end{example}

\begin{example}[a spike: MP spectrum with a nondegenerate radius]
\label{ex:spike}
Let $x=g+\sqrt p\,g_1e_1$ with $g\sim N(0,I_p)$. Then
$R_p=\|g\|^2/p+((1+\sqrt p)^2-1)g_1^2/p\Rightarrow\nu=\law(1+Z^2)$, $Z\sim N(0,1)$,
a nondegenerate radial law. But $S=\frac1N\sum_jg_jg_j^{\mathsf T}+E$ with
$\rank E\le2$ (the cross terms span $e_1$ and $\frac1N\sum_jg_{j1}g_j$), so by
\cref{lem:rank-one} $\ESD(S)\Rightarrow\MP_c=\mu_{c,\delta_1}\ne\mu_{c,\nu}$.
(RQC) fails: for $P=e_1e_1^{\mathsf T}$, $Q_p(P)\Rightarrow Z^2$. (ACI) fails
for the canonical reduction: $\Sigma_p$ is diagonal by sign symmetry, and
$\frac1p\|\Sigma_p-R_pI\|_1\Rightarrow2Z^2$.
Indeed, the last $p-1$ diagonal entries are equal and their normalized
sum tends to $1$ in $L^1$, even after conditioning on $R_p$; the trace
identity $\Tr\Sigma_p=pR_p$ then gives the assertion. In fact the same rank-two argument
applied to $C_Px$ shows that (RMP-$\Pi$) for $(c,\delta_1)$ \emph{holds}, so
\cref{thm:necessity} forces more: since (RQC) fails, $(x_p)$ admits no radial
reduction at all with $\rho_p\Rightarrow\delta_1$. Thus proportional projected spectra do not identify the parent radius
without an assumption excluding energy in small subspaces.
\end{example}

\begin{example}[radial law for the full matrix, failure under projection]
\label{ex:block}
Let $p=2q$, $g^{(1)},g^{(2)}\sim N(0,I_q)$, $\xi\sim\mathrm{Bernoulli}(1/2)$,
$R\sim\nu$, all independent, and
$x=\sqrt{2R}\,(\xi g^{(1)},(1-\xi)g^{(2)})\in\R^p$. Then
$R_p=R\|g^{(2-\xi)}\|^2/q\Rightarrow\nu$. Given $(R,\xi,\|g^{(1)}\|,\|g^{(2)}\|)$
the direction of $g^{(2-\xi)}$ is uniform on $S^{q-1}$, so
$\E[xx^{\mathsf T}\mid R,\xi,\|g^{(1)}\|,\|g^{(2)}\|]=2R_p(\xi I_q\oplus(1-\xi)I_q)$;
and $\xi$ is independent of $R_p$ because $\|g^{(1)}\|\overset d=\|g^{(2)}\|$.
Hence $\Sigma_p=\E[xx^{\mathsf T}\mid R_p]=R_pI$: (ACI) holds exactly. $S$ is block
diagonal; the first block is $\frac{N_1}N\cdot\frac1{N_1}\sum_{j\in J_1}2R_jg_j^{(1)}g_j^{(1)\mathsf T}$
with $J_1=\{j:\xi_j=1\}$, $N_1=|J_1|$, $N_1/N\to1/2$, hence
$(1+o(1))\frac1{N_1}\sum_{j\in J_1}R_jg_j^{(1)}g_j^{(1)\mathsf T}$, a weighted
Wishart matrix in dimension $q$ with $q/N_1\to c$ and radii $R_j\Rightarrow\nu$;
by \cref{thm:sufficient} its ESD converges to $\mu_{c,\nu}$, and so does the
second block's. Therefore $\ESD(S)\Rightarrow\mu_{c,\nu}$: the radial law holds
for the full matrix with the correct $\nu$. Yet for $P$ the projection onto the
first block,
\[
 x^{\mathsf T}Px/p-R_p\Tr P/p\Rightarrow R(\xi-\tfrac12),
\]
which is
nondegenerate unless $\nu=\delta_0$, so (RQC) fails, and the projected covariance has ESD $\Rightarrow\mu_{c/2,\frac12(\delta_0+\law(2R))}\ne\mu_{c/2,\nu}$
by the injectivity in \cref{prop:limit-law}, unless $\nu=\delta_0$. This is the
radial form of the block example in \cite[\S2]{Yaskov2016}, due to
\cite{Adamczak2011}; it shows that the
hypothesis of \cref{thm:necessity} must involve projections.
\end{example}

\begin{example}[Gaussian branch and sparse branch]
\label{ex:mixture}
Let $x=\sqrt p\,e_J$ with probability $1/2$ and $x=g$ with probability
$1/2$. Then $R_p\to1$, $\Sigma_p=R_pI$ exactly, and (RQC) fails. Write
$S=D+W$ with $D=c_n\diag(n_i)$ from the sparse samples and $W$ the Wishart
matrix of the Gaussian samples; $W$ is orthogonally invariant and, conditional on the branch indicators,
independent of $D$. Both conditional spectral limits are deterministic
because the branch proportions converge to $1/2$. Conditioning first, $\ESD(S)\Rightarrow\mu_D\boxplus\mu_W$ with
$\mu_D=\law(c\,\mathrm{Poi}(1/(2c)))$ and $\mu_W$ the law of $\frac12Y$,
$Y\sim\MP_{2c}$ \cite[Theorem~2.1]{PasturVasilchuk2000}, whose hypotheses ask
only for weak convergence of the two spectral measures and a uniformly bounded
first absolute moment, both of which hold here. Since $S\succeq D$, Weyl's
inequality gives $\ESD(S)([L,\infty))\ge\ESD(D)((L,\infty))$, so the support of
the limit is unbounded and the limit is not $\MP_c$. (The first three moments
agree with those of $\MP_c$; the fourth exceeds it by $c^2/4$.)
\end{example}

\begin{example}[heavy-tailed coordinates]
\label{ex:heavy}
Let $x=p^{1/2-1/\alpha}(X_1,\dots,X_p)$ with $X_i$ i.i.d.\ symmetric and
$\Pp(|X|>t)\sim t^{-\alpha}$, $0<\alpha<2$. Then
$R_p=p^{-2/\alpha}\sum_iX_i^2$ converges to a positive $(\alpha/2)$-stable law
$\nu$ (Laplace transform $\exp(-\Gamma(1-\alpha/2)\theta^{\alpha/2})$), which
has no mean. The law is signed-permutation invariant, so (ACI) holds exactly.
For the coordinate projection $P$ of rank $\lfloor p/2\rfloor$, the two halves
are independent and
$Q_p(P)\Rightarrow2^{-1-2/\alpha}(S^{(1)}-S^{(2)})$ with $S^{(1)},S^{(2)}$
i.i.d.\ $\sim\nu$, which is nondegenerate: (RQC) fails. By
\cref{thm:necessity}, (RMP-$\Pi$) fails as well. We make no claim about the
ESD of $S$ itself; see \cite{BelinschiDemboGuionnet2009} for heavy-tailed
sample covariance matrices.
\end{example}

\section{Deterministic population covariance}
\label{sec:population-statement}

\begin{theorem}
\label{thm:population}
Let $T=T_p$ be deterministic and positive semidefinite, with
$\sup_p\|T_p\|<\infty$ and $\ESD(T_p)\Rightarrow H\ne\delta_0$. Under the hypotheses of
\cref{thm:sufficient}, the ESD of
\[
  S_T=\frac1N\sum_{j=1}^NT^{1/2}x_jx_j^{\mathsf T}T^{1/2}
\]
converges weakly in probability to the unique probability measure $\mu_{c,\nu,H}$ on $[0,\infty)$
whose Stieltjes transform satisfies, for every $z\in\C_+$,
\begin{equation}
  s(z)=\int\frac{H(\dd t)}{\kappa t-z},\qquad
  m=\int\frac{t\,H(\dd t)}{\kappa t-z},\qquad
  \kappa=\int\frac{r\,\nu(\dd r)}{1+crm},
  \label{eq:T-RMP}
\end{equation}
for some $m=m(z)$ with $\Imc m\ge0$ and $\int r|1+crm|^{-1}\nu(\dd r)<\infty$,
and $\kappa=\kappa(z)\in\C$; the pair $(m,\kappa)$ solving \eqref{eq:T-RMP} is
unique, and $\Imc m>0$.
\end{theorem}

The limit is the free multiplicative convolution $H\boxtimes\mu_{c,\nu}$
(see \cite{Voiculescu1991,HiaiPetz2000} for this operation and for asymptotic
freeness), with the convention $H\boxtimes\delta_0=\delta_0$; this identification (\cref{prop:free-mult}) uses freeness results
that we quote rather than prove.

For $H=\delta_1$ one has $m=s$ and \eqref{eq:T-RMP} reduces to \eqref{eq:RMP};
for $\nu=\delta_1$ it is Silverstein's equation \cite{Silverstein1995} in the
variable $m=\lim\frac1p\Tr(TG)$.

The proof, including the free-probability identification, is in
\cref{sec:population}. The analytic fixed-point argument does not use
asymptotic freeness.

\section{Limits of global spectral identification}
\label{sec:extremes}

RQC and weak radial convergence impose no upper bound on the spectral edge.

\begin{proposition}
\label{prop:lambda-max-lower}
For every $n$, $\lambda_{\max}(S)\ge c_n\max_{j\le N}R_j$. Consequently, if
$\supp\nu$ is unbounded, $\lambda_{\max}(S)\to\infty$ in probability under the
hypotheses of \cref{thm:sufficient} (indeed under $R_p\Rightarrow\nu$ alone).
\end{proposition}

\begin{proof}
The bound follows from $S\succeq N^{-1}x_jx_j^T$, including when $x_j=0$.
For the second statement let $t>0$ satisfy
$\nu((t,\infty))>0$; the Portmanteau theorem for the open set $(t,\infty)$
gives $\liminf_p\Pp(R_p>t)\ge\nu((t,\infty))>0$, so
$\Pp(\max_{j\le N}R_j\le t)=\Pp(R_p\le t)^N\to0$. When $\supp\nu$ is unbounded
such $t$ exist arbitrarily large.
\end{proof}

\begin{example}[Same global law, different top eigenvalues]
\label{ex:extremes}
Let $\nu$ be supported on $[0,B]$. Take independent spherical directions
$u_j$ and radii with either law $\nu$ or
$(1-p^{-1})\nu+p^{-1}\delta_p$, and form columns $\sqrt{pR_j}\,u_j$.
Both models satisfy RQC and have limiting law $\mu_{c,\nu}$.
In the first model, write $\sqrt p\,u_j=g_j\sqrt p/\|g_j\|$
with independent standard Gaussian $g_j$.
Gaussian norm and singular-value bounds \cite{DavidsonSzarek2001} give
$\min_{j\le N}\|g_j\|^2/p\to1$ and
$\|G\|^2/N=O_{\Pp}(1)$, so the top eigenvalue is bounded in probability.
In the second model, a radius equal to $p$ occurs with probability
$1-(1-p^{-1})^N\to1-e^{-1/c}$; on that event
$\lambda_{\max}(S_p)\ge c_np$.
\end{example}

Quantitative local laws require additional assumptions on angular
fluctuations and on radial tails. For comparison, see the separable
local laws of \cite{KnowlesYin2017} and the non-separable results of
\cite{FanMaPaquetteWang2026}. Such conclusions do not follow from
qualitative RQC.

The remaining inverse question concerns the radius, rather than the
characterization conditional on its limit.

\begin{question}\label{q:radius-recovery}
Assume \eqref{eq:no-small-subspace-energy}, and suppose every deterministic
rank-$\lfloor\alpha p\rfloor$ projected ESD converges in probability to
$\mu_{c\alpha,\nu}$ for a specified probability law $\nu$ with
$\int r^2\,\nu(\dd r)=\infty$.
Must $R_p\Rightarrow\nu$?
\end{question}

The geometric condition already gives radial tightness.
\Cref{thm:radius-identification} answers the finite-second-moment case;
its bootstrap bound \eqref{eq:inverse-radius-moment-bootstrap} does not
extend to the case posed here. Once $R_p\Rightarrow\nu$ is known,
\cref{thm:necessity-free} supplies RQC for arbitrary tails.

\section{Relation to earlier results}
\label{sec:related}

\emph{Classical and Yaskov-type theorems.} \Cref{thm:sufficient} contains the
MP theorem under \eqref{eq:yaskov-A} as the case $\nu=\delta_1$, and its proof
retains the radius in the resolvent argument. Its qualitative conclusion
is also obtained from the more general random-profile replacement theorem
\cite[Theorem~2.2]{Yaskov2014}: set the profile equal to $R_pI_p$.
That theorem's assumptions become RQC and $R_p^2/p\to0$ in probability,
the latter automatic under tightness. It does not require the profile
to be a conditional second-moment matrix. Our direct sufficient proof
makes the radial equation and moment-free tightness explicit.
In the
characterization, \cref{thm:necessity} applied through the trivial radial
reduction ($\mathcal H_p$ trivial, $\rho_p=1$, $\nu=\delta_1$) contains
\cite[Theorem~3.3]{Yaskov2015b}: the hypothesis there is almost sure
convergence of the projected spectra, which implies ours, the conclusion there
is the quadratic-form condition for real symmetric positive semidefinite
matrices, which ours implies by \cref{prop:rqc-equivalences}, and $R_p\to1$ is
recovered as a conclusion rather than assumed. The converse half of
\cite[Theorem~3.3]{Yaskov2015b} is recovered with almost sure convergence under
the summability hypothesis of \cref{thm:sufficient}, and with convergence in
probability otherwise. That characterization appears in
the preprint version \cite[Theorem~3.3]{Yaskov2015b}, whose numbering
we use. The published note \cite{Yaskov2016} states a characterization
in terms of resolvent quadratic forms; its Theorem~2.1 is the
radius-concentration result used here. The radial
conditional version uses a reduction, whereas \cref{thm:necessity-free}
assumes the actual radius law and tests one rank fraction, with no moments.
\Cref{ex:block}
is the radial form of the block example of \cite[\S2]{Yaskov2016} and
\cite{Adamczak2011}, showing that projections are needed.

\emph{Projection-energy rigidity.}
The scalar Jensen identity in our bounded-transform argument is already
present in \cite[Lemma~4.1 and Appendix~B]{Yaskov2015b}; that paper also
uses a single bounded transform to test concentration in Lemma~4.4.
Projection reductions and rank-one resolvent replacement are standard
ingredients. The convex hull of fixed-rank projections, used in
\cref{lem:energy-L2,thm:bounded-energy}, is the standard fantope
\cite{VuChoLeiRohe2013}. The resulting one-sided, dimension-free bounds
apply to arbitrary random positive semidefinite operators, with no radial
moments or rank restriction on the operator. The criterion recovers energy
concentration; it does not assert rotational invariance of the vector law.

\emph{Logarithmic spectral comparison.}
Shannon-transform formulas for covariance models are classical; see, for
example, \cite[Theorem~4.1]{HachemLoubatonNajim2007}. Our converse uses the
finite-sample concavity of expected log-determinants in the common i.i.d.\
column law, proved here by a negative tensor-square expansion. Combining
angular symmetrization with complementary projection laws and a strict
Jensen gap gives the characterization for arbitrary radial laws.
The determinant comparison of \cref{lem:frame-trimming} and Gaussian
singular-value bounds control the truncation error by tail probability alone,
as in \eqref{eq:truncated-gap-bound}. The variational formula and Gaussian
singular-value estimates are established ingredients.
Proposition~\ref{prop:bounded-spectral-obstruction} precludes a universal
bounded scalar replacement of the logarithm; the proof instead compares
finite logarithmic quantities before passing to the cutoff limit.

\emph{Elliptical laws and scale mixtures.} For $x=\sqrt{pR}\,u$ with $u$
uniform and $R\perp u$ the limit is in \cite{ElKaroui2009,PajorPastur2009};
\cref{ex:spherical} shows that independence plays no role: any $R$ with
$R_p\Rightarrow\nu$, however it depends on $u$, gives the same limit.

\emph{Weighted and separable covariance.} For
$\frac1N\sum_jw_jz_jz_j^{\mathsf T}$ with $(w_j)$ independent of $(z_j)$ the
equation \eqref{eq:RMP} is \cite[Theorem~1.1]{SilversteinBai1995} for
coordinates with independent entries and
\cite[Theorem~4.2]{ChengMikulincer2026} under a quadratic-form condition on the
directions. For the separable model $T^{1/2}ZD^{1/2}$ and its limiting
spectral measure, see \cite{CouilletHachem2014} and the references therein,
including \cite{Zhang2007}; see \cite{PaulSilverstein2009} for the absence
of eigenvalues outside the support.
Conditionally on the weights these are sample covariance matrices with
independent columns of different variances, and the proofs use that structure.
Under (RQC) there is no such conditional structure: in the tensor model we
have no description of the conditional law of $x/\sqrt{R_p}$ given $R_p$ beyond
the second-moment identity of \cref{lem:tensor-ACI}. \Cref{prop:self-normalized} explains the
relationship to weighted models. The absence of radial-angular independence
is already allowed by the random-profile theorem \cite{Yaskov2014}.
The survey \cite{Yaskov2025} develops this random-profile formulation for
partially dependent Gram matrices. A converse for a general auxiliary profile
would also have to specify its coupling to the sample; its marginal law alone
cannot identify that coupling.
The new converse and energy-rigidity estimates address necessity and
projected-energy identification.

\emph{Free compound Poisson limits.} Boedihardjo \cite{Boedihardjo2015} shows
that free compound Poisson laws are the natural limit class for sample
covariance matrices under a different set of structural hypotheses on the
sample vectors. The limit class here is the same; the hypotheses are not
comparable, and we do not use his results.

\emph{Division from the tensor paper.} Reference \cite{Xie2026} proves the
critical tensor results: the lognormal radius limit (Lemma~4.2), radial $L^2$
concentration (Theorem~6.2), the abstract sufficient principle under $L^2$ and
uniform-integrability hypotheses (Theorem~3.2), and the tensor spectral limit
(Theorem~3.1), together with its atoms, moments, and support. Those statements
are inputs here and are not claimed again.

The present paper begins where that analysis ends. It weakens the sufficient
principle to moment-free convergence in probability, but its main new direction
is converse: projected radial laws force angular concentration. Applied to the
tensor example, \cref{thm:necessity-free} gives a one-rank characterization
using the quoted convergence of its actual radius. This is a
new implication about projected spectra, not a second proof of the critical
tensor ESD. It requires neither a symmetric base law nor (ACI); the conditional
result \cref{thm:necessity} gives a separate route for symmetric base laws. The
population-covariance and extreme-eigenvalue results are likewise general
radial statements, not tensor asymptotics.

\appendix

\section{Proof of the population-covariance theorem}
\label{sec:population}

\subsection{What changes with a population covariance}

For $y=T^{1/2}x$ the radial condition gives
$y^{\mathsf T}Ay/p-R_p\Tr(TA)/p\to0$: the sample sees $T$ through the trace
$\Tr(TA)$, so the leave-one-out quadratic form $a_j=\frac1py_j^{\mathsf T}G_jy_j$
is close to $R_j\,m_n$ with $m_n=\frac1p\Tr(TG)$ rather than to $R_js_n$. The
exact identity \eqref{eq:exact} then involves $m_n$ on the right and $s_n$ on
the left and does not close. The standard remedy, due to Silverstein
\cite{Silverstein1995,SilversteinBai1995}, is to compare $G$ not with a scalar
but with the deterministic-equivalent resolvent $(\kappa T-z)^{-1}$ for a
suitable scalar $\kappa$, and to derive a closed system for $(s,m,\kappa)$. The
radial version has $\kappa=\frac1N\sum_jR_je_j$ with $e_j=(1+c_na_j)^{-1}$,
which is where the radii enter. Two features are specific to the radial
setting. First, without a first moment on $\nu$ the quantities $R_je_j$ are not
bounded, and the bounded-summand device of \cref{sec:sufficient} is not
available for them; we therefore prove the theorem for bounded radii and pass
to general $\nu$ by removing the samples with large radius, exactly as in
Step~1 of \cref{sec:sufficient}. \Cref{rem:truncation-needed}
gives a model satisfying all hypotheses in which the term-by-term estimate of
step~(4) breaks down for the untruncated matrix, two normalized column contributions to $\kappa_n$
being bounded below by positive constants while the matching contributions
to $\hat\kappa_n$ are $O_{\Pp}(1/p)$. Second, for unbounded $\nu$ the
integral defining $\kappa$ need not converge absolutely at the limit point of
the truncated solutions; a lower bound $\Imc m\ge\gamma_0(z)>0$, uniform in the
truncation level and obtained from the tightness of the ESD, is what makes it
converge and what makes the passage to the limit uniform in $r$.
\Cref{lem:T-uniqueness} derives $\Imc m>0$ for a solution by itself; it is the
uniformity in the truncation level that has to be proved.

\subsection{Uniqueness}

\begin{lemma}
\label{lem:T-uniqueness}
Fix $z=E+i\eta\in\C_+$, $c>0$, a probability measure $\nu$ on $[0,\infty)$ and a
probability measure $H\ne\delta_0$ on $[0,\tau]$. Call $m\in\C$ a solution if
$\Imc m\ge0$, $\int r|1+crm|^{-1}\nu(\dd r)<\infty$, and
\begin{equation}
  m=\int\frac{t\,H(\dd t)}{\kappa(m)t-z},\qquad \kappa(m)=\int\frac{r\,\nu(\dd r)}{1+crm}.
  \label{eq:T-fixed}
\end{equation}
Every solution has $\Imc m>0$, and there is at most one solution.
\end{lemma}

\begin{proof}
If $\Imc m>0$ then $|1+crm|\ge cr\Imc m$, so $|r/(1+crm)|\le(c\Imc m)^{-1}$ and
$r^2/|1+crm|^2\le(c\Imc m)^{-2}$: $\kappa(m)$ and
$\alpha(m):=c\int r^2|1+crm|^{-2}\nu(\dd r)$ are finite for every $\nu$. From
$\Imc\frac r{1+crm}=-\frac{cr^2\Imc m}{|1+crm|^2}$ we get
$\Imc\kappa=-\alpha\Imc m\le0$, hence $\Imc(\kappa t-z)\le-\eta$ and
$|\kappa t-z|\ge\eta$ for $t\ge0$, and
\begin{multline}
  \Imc m=\int t\,\frac{-\Imc(\kappa t-z)}{|\kappa t-z|^2}H(\dd t)
  =\alpha\beta\,\Imc m+\gamma,\\
  \beta=\int\frac{t^2H(\dd t)}{|\kappa t-z|^2},\qquad
  \gamma=\eta\int\frac{t\,H(\dd t)}{|\kappa t-z|^2}>0 .
  \label{eq:Im-m}
\end{multline}
If $\Imc m=0$, then $\alpha$ may be infinite, but $\kappa$ is real and finite
by the absolute-convergence convention, so $\Imc(\kappa t-z)=-\eta$ and the
first equality in \eqref{eq:Im-m} reads $0=\gamma>0$: the contradiction uses
only $\kappa\in\R$. So $\Imc m>0$, and then
$(1-\alpha\beta)\Imc m=\gamma>0$ gives $\alpha\beta<1$. For two solutions
$m_1\ne m_2$, direct algebra gives
$m_1-m_2=(\kappa_2-\kappa_1)\int\frac{t^2H(\dd t)}{(\kappa_1t-z)(\kappa_2t-z)}$ and
$\kappa_2-\kappa_1=c(m_1-m_2)\int\frac{r^2\nu(\dd r)}{(1+crm_1)(1+crm_2)}$, so
\[
  1=c\int\frac{r^2\,\nu(\dd r)}{(1+crm_1)(1+crm_2)}\int\frac{t^2\,H(\dd t)}{(\kappa_1t-z)(\kappa_2t-z)},
\]
and Cauchy--Schwarz bounds the right side in modulus by
$\sqrt{\alpha_1\alpha_2}\,\sqrt{\beta_1\beta_2}$, which equals
$\sqrt{(\alpha_1\beta_1)(\alpha_2\beta_2)}<1$.
\end{proof}

Given $m$, \eqref{eq:T-RMP} determines $\kappa$ and $s$; also
$\kappa m=\int(1+\frac z{\kappa t-z})H(\dd t)=1+zs$. For $H=\delta_1$,
$s=m=(\kappa-z)^{-1}$ and \eqref{eq:T-RMP} is \eqref{eq:RMP}. For $\nu=\delta_1$,
$\kappa=(1+cm)^{-1}$ and $\kappa m=1+zs$ give $\kappa=1-c-czs$, so
$s=\int H(\dd t)/((1-c-czs)t-z)$, Silverstein's equation.

\subsection{\texorpdfstring{Proof of \cref{thm:population} for bounded radii}{Proof of the population theorem for bounded radii}}

Assume $R_p\le M$ almost surely for all $p$. Fix $z=E+i\eta$, and put
$\tau:=\sup_p\|T_p\|<\infty$; since $\supp\ESD(T_p)\subseteq[0,\tau]$ and
$\ESD(T_p)\Rightarrow H$, also $\supp H\subseteq[0,\tau]$, and $\|T/\tau\|\le1$.
Notation:
$y_j=T^{1/2}x_j$, $G_j$ the resolvent of $S_T-\frac1Ny_jy_j^{\mathsf T}$,
$a_j=\frac1py_j^{\mathsf T}G_jy_j$, $e_j=(1+c_na_j)^{-1}$,
$m_n=\frac1p\Tr(TG)$, $m_n^{(j)}=\frac1p\Tr(TG_j)$, and
$\phi_p(\eps)=\sup_{\|A\|\le1}\Pp(|Q_p(A)|>\eps)\to0$ from
\cref{prop:rqc-equivalences}.

\emph{(1) Denominators.} In the eigenbasis $(t_l,u_l)$ of the positive
semidefinite matrix $S_T-\frac1Ny_jy_j^{\mathsf T}$,
$c_na_j=\frac1N\sum_l|u_l^{\mathsf T}y_j|^2/(t_l-z)$, and in the eigenbasis
$(\lambda_l,v_l)$ of $S_T$, $c_nrm_n=\frac{c_nr}p\sum_lv_l^{\mathsf T}Tv_l/(\lambda_l-z)$
with $v_l^{\mathsf T}Tv_l\ge0$. \Cref{lem:denominator} gives
$|e_j|\le|z|/\eta$ and $|1+c_nrm_n|^{-1},|1+c_nrm_n^{(j)}|^{-1}\le|z|/\eta$ for
$r\ge0$; also $\Imc(c_na_j)=\frac\eta N\|G_jy_j\|^2\ge0$, so $\Imc e_j\le0$,
and $\Imc m_n,\Imc m_n^{(j)}\ge0$, $|m_n|\le\tau/\eta$.

\emph{(2) Rank-one bounds.} By \cite[Lemma~2.6]{SilversteinBai1995},
\[
 |\Tr[((B-z)^{-1}-(B+vv^*/N-z)^{-1})A]|\le\|A\|/\eta
\]
for Hermitian $B$, $v\in\C^p$, and any $A$. Hence
$|m_n-m_n^{(j)}|\le\tau/(p\eta)$ and $\frac1p|\Tr(A'(G-G_j))|\le\|A'\|/(p\eta)$.

\emph{(3) Leave-one-out concentration.} $(\eta/\tau)T^{1/2}G_jT^{1/2}$ is
independent of $x_j$ with norm $\le1$, and
$a_j-R_jm_n^{(j)}=Q_p(T^{1/2}G_jT^{1/2})$, so by \cref{lem:uniform}(c)
$\Pp(|a_j-R_jm_n^{(j)}|>\eps)\le\phi_p(\eps\eta/\tau)$ for every $j$.

\emph{(4) The scalar $\kappa$.} Let
\[
 \begin{gathered}
 \kappa_n=\frac1N\sum_jR_je_j,\qquad
 \hat\kappa_n=\frac1N\sum_j\frac{R_j}{1+c_nR_jm_n},\\
 \hat\kappa^{(j)}=\frac1N\sum_{k\ne j}\frac{R_k}{1+c_nR_km_n^{(j)}} .
 \end{gathered}
\]
$\hat\kappa^{(j)}$ is a function of $\{x_k\}_{k\ne j}$. All three have
nonpositive imaginary part by (1), so $D=\kappa_nT-z$ and
$D_j=\hat\kappa^{(j)}T-z$ are normal with eigenvalues of imaginary part
$\le-\eta$: $\|D^{-1}\|,\|D_j^{-1}\|\le1/\eta$. By (1)--(3) and the identity
$\frac r{1+crw_1}-\frac r{1+crw_2}=\frac{-cr^2(w_1-w_2)}{(1+crw_1)(1+crw_2)}$,
\begin{gather*}
  |\hat\kappa_n-\hat\kappa^{(j)}|\le\frac{M|z|}{N\eta}+c_nM^2\Big(\frac{|z|}\eta\Big)^2\frac{\tau}{p\eta}
  \quad\text{deterministically, and}\\
  \Big|e_j-\frac1{1+c_nR_jm_n}\Big|\le\min\Big\{\frac{2|z|}{\eta},\,c_n\Big(\frac{|z|}\eta\Big)^2|a_j-R_jm_n|\Big\},
\end{gather*}
so $\E|\kappa_n-\hat\kappa_n|\le M[c_n(|z|/\eta)^2(\eps+M\tau/(p\eta))+(2|z|/\eta)\phi_p(\eps\eta/\tau)]$
for every $\eps>0$, and $\kappa_n-\hat\kappa_n\to0$ in $L^1$. Hence
$\sup_j|\kappa_n-\hat\kappa^{(j)}|\to0$ in probability.

\emph{(5) Master identity.} Since $D-(S_T-z)=\kappa_nT-S_T$ and
$Gy_j=e_jG_jy_j$ (Sherman--Morrison),
\[
  G-D^{-1}=G(\kappa_nT-S_T)D^{-1}
  =\kappa_nGTD^{-1}-\frac1N\sum_je_jG_jy_jy_j^{\mathsf T}D^{-1}.
\]
For deterministic $A$ with $\|A\|\le1$, taking $\frac1p\Tr(A\,\cdot)$,
\begin{equation}
\begin{split}
  \frac1p\Tr(AG)-\frac1p\Tr(AD^{-1})
  &=\kappa_n\frac1p\Tr(AGTD^{-1})-\frac1N\sum_je_jq_j,\\
  q_j&=\frac1py_j^{\mathsf T}D^{-1}AG_jy_j .
\end{split}
  \label{eq:master}
\end{equation}

\emph{(6) The quadratic forms $q_j$.} Set
\[
 \hat q_j=\frac1py_j^{\mathsf T}D_j^{-1}AG_jy_j,\qquad
 \theta_n=\frac1p\Tr(TD^{-1}AG).
\]
Then $|\theta_n|\le\tau/\eta^2$. Since
$\|y_j\|^2\le\tau pM$ and
$D^{-1}-D_j^{-1}=D^{-1}(\hat\kappa^{(j)}-\kappa_n)TD_j^{-1}$, we get
$|q_j-\hat q_j|\le\tau^2M|\kappa_n-\hat\kappa^{(j)}|/\eta^3$. The matrix
$T^{1/2}D_j^{-1}AG_jT^{1/2}$ is a function of $\{x_k\}_{k\ne j}$ with norm
$\le\tau/\eta^2$, so \cref{lem:uniform}(c) gives
\[
 \Pp\left(\left|\hat q_j-\frac{R_j}{p}\Tr(TD_j^{-1}AG_j)\right|>\eps\right)
 \le\phi_p(\eps\eta^2/\tau).
\]
By (2) with $A'=TD_j^{-1}A$, $\frac1p|\Tr(A'(G_j-G))|\le\tau/(p\eta^2)$, and
$\frac1p|\Tr(T(D_j^{-1}-D^{-1})AG)|\le\tau^2|\kappa_n-\hat\kappa^{(j)}|/\eta^3$.
Altogether $q_j=R_j\theta_n+\epsilon_j$ with $|\epsilon_j|\le2\tau M/\eta^2$
and $\sup_j\Pp(|\epsilon_j|>\eps)\to0$ for every $\eps$.

\emph{(7) Closing the identity.} By cyclicity
$\Tr(AGTD^{-1})=\Tr(TD^{-1}AG)=p\theta_n$, so \eqref{eq:master} becomes
\[
 \begin{aligned}
 \frac1p\Tr(AG)-\frac1p\Tr(AD^{-1})
 &=\theta_n\left(\kappa_n-\frac1N\sum_je_jR_j\right)
       -\frac1N\sum_je_j\epsilon_j\\
 &=-\frac1N\sum_je_j\epsilon_j.
 \end{aligned}
\]
Its expected absolute value is bounded by
\[
 \frac{|z|}{\eta}\left(\eps+\frac{2\tau M}{\eta^2}
               \sup_j\Pp(|\epsilon_j|>\eps)\right)
 \longrightarrow \frac{|z|}{\eta}\eps.
\]
Hence, for $A\in\{I,T/\tau\}$,
\begin{equation}
  s_n(z)-\int\frac{\ESD(T_p)(\dd t)}{\kappa_nt-z}\xrightarrowp0,\qquad
  m_n(z)-\int\frac{t\,\ESD(T_p)(\dd t)}{\kappa_nt-z}\xrightarrowp0 .
  \label{eq:det-equiv}
\end{equation}

\emph{(8) Lower bound on $\Imc m_n$.} With $\Pi_L$ the spectral projection of
$S_T$ on $[0,L]$,
\begin{multline*}
  \Imc m_n=\frac\eta p\sum_l\frac{v_l^{\mathsf T}Tv_l}{|\lambda_l-z|^2}\ge\frac{\eta}{(L+|z|)^2}\frac1p\Tr(T\Pi_L)\\
  \ge\frac{\eta}{(L+|z|)^2}\Big(\frac1p\Tr T-\tau\,\ESD(S_T)((L,\infty))\Big).
\end{multline*}
Here $R_p\le M$, so $p^{-1}\Tr S_T\le\tau M$ and
$\ESD(S_T)((L,\infty))\le\tau M/L$ deterministically.
Since $p^{-1}\Tr T\to\bar t_H:=\int t\,H(\dd t)>0$,
choosing $L$ with $\tau^2M/L<\bar t_H/4$ gives
$\Imc m_n(z)\ge\gamma_0(z)>0$ for all sufficiently large $n$.

\emph{(9) Limit.} Fix a countable dense $\mathcal Z\subset\C_+$; along a
subsequence, almost surely, $\hat\nu_N\Rightarrow\nu$, the tightness bounds
hold, \eqref{eq:det-equiv} and $\kappa_n-\hat\kappa_n\to0$ hold for all
$z\in\mathcal Z$, and $\Imc m_n(z)\ge\gamma_0(z)$ eventually. Take a weak
cluster point $\mu$ of $\ESD(S_T)$ (a probability measure by tightness) and,
by compactness of $\{|m|\le\tau/\eta\}$ and a diagonal argument, a
subsubsequence along which $m_n(z)\to m_\infty(z)$ for all $z\in\mathcal Z$,
with $\Imc m_\infty\ge\gamma_0>0$. Then
$\hat\kappa_n\to\kappa_\infty:=\int r(1+crm_\infty)^{-1}\nu(\dd r)$, because the
integrand $r/(1+c_nrm_n)$ converges uniformly on $[0,M]$ (denominators
$\ge\eta/|z|$) and $\hat\nu_N\Rightarrow\nu$; $\kappa_n\to\kappa_\infty$ as
well, and $\Imc\kappa_\infty\le0$. The integrands $1/(\kappa_nt-z)$ and
$t/(\kappa_nt-z)$ converge uniformly on $[0,\tau]$ (denominators $\ge\eta$),
so \eqref{eq:det-equiv} passes to the limit and gives \eqref{eq:T-RMP} with
$(s_\mu,m_\infty,\kappa_\infty)$. By \cref{lem:T-uniqueness}, $m_\infty(z)$,
hence $s_\mu(z)$, is determined for $z\in\mathcal Z$, hence on $\C_+$, and all
cluster points coincide. Convergence in probability follows as in
\cref{sec:sufficient}. \qed

\subsection{General radial laws}

Let $M$ be a continuity point of $\nu$ and $x^{(M)}=x\1_{\{R_p\le M\}}$. Its
radius $R_p\1_{\{R_p\le M\}}$ converges weakly to
$\nu_M=\nu|_{[0,M]}+\nu((M,\infty))\delta_0$, and
$Q^{(M)}_p(A)=\1_{\{R_p\le M\}}Q_p(A)$, so (RQC) holds and the bounded case
gives $\ESD(S_T^{(M)})\Rightarrow\mu_M:=\mu_{c,\nu_M,H}$ in probability, where
$S_T^{(M)}$ is $S_T$ with the samples in $J_M=\{j:R_j>M\}$ removed. By
\cref{lem:rank-one} and \eqref{eq:JM}, for continuity points $M<M'$,
$\dK(\ESD(S_T^{(M)}),\ESD(S_T^{(M')}))\le2|J_M|/p\le2\nu((M,\infty))/c+o_p(1)$.
Passing to the limit at common continuity points of $F_{\mu_M},F_{\mu_{M'}}$
and using right continuity gives
\[
 \dK(\mu_M,\mu_{M'})\le2\nu((M,\infty))/c.
\]
Thus
$(F_{\mu_M})_M$ is uniformly Cauchy; its uniform limit $F$ is nondecreasing,
right-continuous, vanishes on $(-\infty,0)$ and has
$F(+\infty)\ge1-2\nu((M,\infty))/c$ for every $M$, hence is the distribution
function of a probability measure $\mu$, and $\dK(\mu_M,\mu)\le2\nu((M,\infty))/c$.
With the L\'evy metric $d_{\mathrm L}\le\dK$, which metrizes weak convergence,
\begin{multline*}
  d_{\mathrm L}(\ESD(S_T),\mu)\le\dK\big(\ESD(S_T),\ESD(S_T^{(M)})\big)
  +d_{\mathrm L}\big(\ESD(S_T^{(M)}),\mu_M\big)\\+\dK(\mu_M,\mu)
  \le\frac{4\nu((M,\infty))}{c}+o_p(1),
\end{multline*}
and $M\to\infty$ gives $\ESD(S_T)\Rightarrow\mu$ in probability. (Kolmogorov
distance cannot be used for the middle term: $\mu_M$ may have an atom at $0$.)

It remains to identify $\mu$ through \eqref{eq:T-RMP}. Let $(s_M,m_M,\kappa_M)$
solve the system for $\nu_M$. The lower bound (8) is uniform in the truncation
level: fix a continuity point $M_0$ with $\tau\nu((M_0,\infty))/c<\bar t_H/8$
and $L$ with $\tau^2M_0/L<\bar t_H/8$; for every $M\ge M_0$, removing the
samples in $J_{M_0}$ from $S_T^{(M)}$ leaves a matrix of trace $\le\tau M_0p$,
so $\ESD(S_T^{(M)})((L,\infty))\le\tau M_0/L+|J_{M_0}|/p$ simultaneously for all
$M\ge M_0$, and (8) gives $\Imc m_n^{(M)}(z)\ge\gamma_0(z)$ with probability
$\to1$, uniformly in $M\ge M_0$. Since $m_n^{(M)}(z)\to m_M(z)$ in probability
(every cluster point is the unique solution), $\Imc m_M(z)\ge\gamma_0(z)$ for
all $M\ge M_0$. Take $M_k\to\infty$ with $m_{M_k}(z)\to m_\infty(z)$,
$\Imc m_\infty\ge\gamma_0$. Then
$|\frac r{1+crm_{M_k}}-\frac r{1+crm_\infty}|\le|m_{M_k}-m_\infty|/(c\gamma_0^2)$
uniformly in $r>0$, and
$|\int\frac{r\,\nu_{M_k}(\dd r)}{1+crm_\infty}-\int\frac{r\,\nu(\dd r)}{1+crm_\infty}|\le2\nu((M_k,\infty))/(c\gamma_0)$,
so $\kappa_{M_k}\to\kappa_\infty=\int r(1+crm_\infty)^{-1}\nu(\dd r)$, with the
integral absolutely convergent; the $H$-integrals pass to the limit as before,
$s_{\mu_{M_k}}(z)\to s_\mu(z)$, and $(s_\mu,m_\infty,\kappa_\infty)$ solves
\eqref{eq:T-RMP}. \Cref{lem:T-uniqueness} identifies $m_\infty$ and $s_\mu(z)$
for every $z\in\C_+$, so $\mu=\mu_{c,\nu,H}$. \qed

\begin{remark}[Failure of an untruncated termwise estimate]
\label{rem:truncation-needed}
The bounded-radius argument cannot be applied directly by using (8) to bound
$R_je_j$ through $|c_na_j/(1+c_na_j)|\le1+|z|/\eta$ and $a_j\approx R_jm_n^{(j)}$.
For example, take $T=I$, $\Pp(R>t)=1/t$ for
$t\ge1$, and $x=\sqrt{pR}\,u$ with $u$ uniform on the sphere when $R\le p$ and
$u=e_1$ when $R>p$. Then $R_p=R\Rightarrow\nu$, and (RQC) holds: for any projection $P$ and any
$K<p$, the direction is uniform on $\{R\le K\}$, so
$\Pp(|Q_p(P)|>\eps)\le\Pp(R>K)+2K^2/(\eps^2(p+2))=1/K+2K^2/(\eps^2(p+2))$;
let $p\to\infty$, then $K\to\infty$. With
probability tending to $\frac1{2c^2}e^{-1/c}>0$, exactly two samples $j,j'$
have $R_j,R_{j'}>p$, and both then have direction $e_1$; by symmetry the two
contributions add rather than cancel. Evaluate at $z=-t<0$ and let
$B=\frac1N\sum_{k\ne j,j'}x_kx_k^{\mathsf T}$,
$g=e_1^{\mathsf T}(B+tI)^{-1}e_1>0$ and $\zeta=1/(c_ng)>0$. Two applications of the Sherman--Morrison formula give
\[
  R_je_j=\frac{R_j(\zeta+R_{j'})}{\zeta+R_j+R_{j'}},
\]
Since $\zeta>0$,
\[
 R_je_j\ge\frac{R_jR_{j'}}{R_j+R_{j'}}\ge p/2,
\]
and the same holds for $j'$. Their contribution to $\kappa_n$ is at
least $c_n$. On the other hand, convergence of the ESD to a deterministic
probability law gives
$m_n(-t)=p^{-1}\Tr(B_{\rm full}+tI)^{-1}\ge a_t>0$ with probability
tending to one, for some constant $a_t$; here $B_{\rm full}$ is the
full sample covariance. Their contribution to $\hat\kappa_n$ is
therefore at most $2/(Nc_nm_n(-t))=O_{\Pp}(p^{-1})$.
Two normalized column contributions to $\kappa_n$ are
therefore bounded below by positive constants while the corresponding summands of $\hat\kappa_n$
vanish: the estimate of step~(4), which controlled $\kappa_n-\hat\kappa_n$
term by term through $R_j\le M$, has no analogue for the untruncated matrix,
and these two samples must be removed as a rank-two perturbation, which is
what the truncation does.
\end{remark}

\begin{proposition}[free multiplicative convolution]
\label{prop:free-mult}
Under the hypotheses of \cref{thm:population}, $\mu_{c,\nu,H}=H\boxtimes\mu_{c,\nu}$.
Moreover, let $\mathbf G$ be a $p\times N$ standard Gaussian matrix,
$D_R=\diag(R_1,\dots,R_N)$ with $R_j$ i.i.d.\ of law $\nu$ independent of
$\mathbf G$, and $\underline W=\frac1ND_R^{1/2}\mathbf G^{\mathsf T}\mathbf GD_R^{1/2}$
the $N\times N$ companion of the Gaussian surrogate; then
\begin{equation}
  c\,\mu_{c,\nu}+(1-c)\delta_0=\nu\boxtimes\mathrm D_c\MP_{1/c},
  \label{eq:companion}
\end{equation}
where $\mathrm D_c\rho(A)=\rho(A/c)$ is the dilation by $c$ and $\MP_{1/c}$ the
MP law of ratio $1/c$ and mean one.
\end{proposition}

\begin{proof}
Since $\mu_{c,\nu,H}$ depends only on $(c,\nu,H)$, evaluate it on the Gaussian
surrogate $\tilde x=\sqrt R\,g$, for which $S_T=T^{1/2}W_\nu T^{1/2}$ with
$W_\nu=\frac1N\mathbf GD_R\mathbf G^{\mathsf T}$. For bounded radii $R\le M$,
$W_\nu\preceq M\cdot\frac1N\mathbf G\mathbf G^{\mathsf T}$, so
$\limsup_n\|W_\nu\|\le M(1+\sqrt c)^2$ almost surely, by the Gaussian norm bounds
\cite[Theorem~II.13]{DavidsonSzarek2001} and Borel--Cantelli, and
$\ESD(W_\nu)\Rightarrow\mu_{c,\nu}$ almost surely by the last clause of
\cref{thm:sufficient}: since $\mu_{c,\nu,H}$ and $\mu_{c,\nu}$ depend only on
$(c,\nu,H)$, we are free to evaluate them along dimensions $p_n=n$,
$N_n=\lceil n/c\rceil$ and any deterministic $T_n\succeq0$ with
$\|T_n\|\le\tau$ and $\ESD(T_n)\Rightarrow H$, for which
$\sum_ne^{-ap_n}<\infty$ for every $a>0$. The law of $W_\nu$ is invariant under
conjugation by a fixed orthogonal matrix, so if $O$ is Haar distributed and
independent of $W_\nu=U\Lambda U^{\mathsf T}$ then
$W_\nu\overset d=(OU)\Lambda(OU)^{\mathsf T}$ with $OU$ Haar and independent of
$\Lambda$: $W_\nu$ is a Haar conjugation of a spectrum independent of $T$.
Apply \cite[Theorem~5.2]{CollinsSniady2006} conditionally on the spectra.
If needed, pass to a subsequence on which all normalized mixed traces of
the bounded pair $(T_p,\Lambda_p)$ converge; a diagonal extraction supplies
such a subsequence. The matrices are real symmetric, so the joint limit
also includes their transposes, as required by that theorem.
Freeness with the independent Haar matrix implies that the limiting
moments of $T_pO_p\Lambda_pO_p^T$ depend only on the two marginal limits.
Every extracted limit is consequently $H\boxtimes\mu_{c,\nu}$.
Since $\spec(T^{1/2}W_\nu T^{1/2})=\spec(W_\nu T)$ and
both limits are compactly supported, $\mu_{c,\nu,H}=H\boxtimes\mu_{c,\nu}$ for
compactly supported $\nu$. For general
$\nu$, $\dK(\mu_{c,\nu_M},\mu_{c,\nu})\le2\nu((M,\infty))/c$ and
$\dK(\mu_{c,\nu_M,H},\mu_{c,\nu,H})\le2\nu((M,\infty))/c$ by the truncation
argument above. Corollary~6.7 of \cite{BercoviciVoiculescu1993} gives joint
weak continuity of $\boxtimes$ on probability measures on $[0,\infty)$
when both limiting factors differ from $\delta_0$. Here $H\ne\delta_0$ by
hypothesis, and $\mu_{c,\nu}\ne\delta_0$ when $\nu\ne\delta_0$ by
\cref{prop:limit-law}. Thus the identity passes to the limit; for
$\nu=\delta_0$ both sides are $\delta_0$. For \eqref{eq:companion}: $p\,\ESD(W_\nu)$ and $N\,\ESD(\underline W)$
differ by $(p-N)\delta_0$, so $\ESD(\underline W)=c_n\ESD(W_\nu)+(1-c_n)\delta_0$;
and $\underline W=D_R^{1/2}(\frac1N\mathbf G^{\mathsf T}\mathbf G)D_R^{1/2}$ with
$\frac1N\mathbf G^{\mathsf T}\mathbf G=c_n\cdot\frac1p\mathbf G^{\mathsf T}\mathbf G$,
whose ESD converges to $\mathrm D_c\MP_{1/c}$; asymptotic freeness of $D_R$ and
the orthogonally invariant Wishart factor, applied conditionally on $D_R$
(which is independent of $\mathbf G$, has $\|D_R\|\le M$ and
$\ESD(D_R)\Rightarrow\nu$ almost surely), gives the right side. The left side
of \eqref{eq:companion} is a probability measure for every $c>0$: by
\cref{prop:limit-law}, $\mu_{c,\nu}(\{0\})\ge1-1/c$, so for $c>1$ the atom of
$c\,\mu_{c,\nu}$ at $0$ absorbs the negative mass $1-c$, while the total mass
is $c+(1-c)=1$. For unbounded $\nu$, apply the identity to $\nu_M$ and let
$M\to\infty$: $\dK(\mu_{c,\nu_M},\mu_{c,\nu})\le2\nu((M,\infty))/c$ and
$\nu_M\Rightarrow\nu$, so both sides converge weakly. For the right side,
Corollary~6.7 applies when $\nu\ne\delta_0$, since the other limiting
factor is the nonzero dilation of $\MP_{1/c}$; for $\nu=\delta_0$ both
sides are $\delta_0$.
\end{proof}

\section{An elementary Gram bound with inverse-aspect error}
\label{app:multi-aspect}

The following Gram inequality bounds the energy deficit directly, with
error $1/d$ and no determinants.

\begin{lemma}[A Gram-matrix bound]\label{lem:gram-radius-bound}
Let $Y=(y_1,\ldots,y_N)\in\R^{q\times N}$ be arbitrary, put
$S=YY^T/N$, $d_q=q/N$, and $r_j=\|y_j\|^2/q$.
Then
\begin{equation}\label{eq:gram-radius-bound}
 \frac1N\sum_{j=1}^N\psi(r_j)
 \ge d_q\int\psi(\lambda/d_q)\,\ESD(S)(\dd\lambda).
\end{equation}
\end{lemma}
\begin{proof}
The diagonal of $Y^TY/q$ is $(r_j)$, so the concave Schur--Horn
inequality gives $\sum_j\psi(r_j)\ge\Tr\psi(Y^TY/q)$.
Its nonzero eigenvalues are those of $S/d_q$; use $\psi(0)=0$ and divide by $N$.
\end{proof}

\begin{theorem}[Moment-free converse at unbounded aspect ratios]
\label{thm:multi-aspect-inverse}
Assume $R_p\Rightarrow\nu$, where $\nu$ is a probability measure on $[0,\infty)$,
fix $\alpha\in(0,1)$, and put $q_p=\lfloor\alpha p\rfloor$.
Let $d_k\to\infty$. For each $k$, suppose there are sample sizes
$N_{p,k}$ with $q_p/N_{p,k}\to d_k$ such that, for every deterministic
rank-$q_p$ sequence $P_p$, the projected covariance formed from $N_{p,k}$
independent copies of $x_p$ has ESD converging in probability to
$\mu_{d_k,\nu}$. Then (RQC) holds, and conversely RQC implies these
spectral limits. No coupling between the ensembles for different $k$ is required.
\end{theorem}

\begin{proof}
Fix $d=d_k>1$. Apply \cref{lem:gram-radius-bound} to the projected
columns and take expectations. Boundedness of $\psi$, weak spectral
convergence, and sequence selection imply, uniformly in $P$,
\[
 \liminf_p\E\psi(\|Px_p\|^2/q_p)
 \ge d\int\psi(\lambda/d)\,\mu_{d,\nu}(\dd\lambda).
\]
Put $a_d=d\,s_{d,\nu}(-d)$. The fixed-point equation gives
\[
 1=a_d+\frac1d\int\psi(a_dr)\,\nu(\dd r),
 \qquad 1-1/d\le a_d\le1,
\]
and
\[
 d\int\psi(\lambda/d)\,\mu_{d,\nu}(\dd\lambda)
 =d(1-a_d)=\int\psi(a_dr)\,\nu(\dd r).
\]
Since
\[
 0\le\psi(r)-\psi(a_dr)
 =\frac{(1-a_d)r}{(1+r)(1+a_dr)}\le\frac1d,
\]
and $\E\psi(R_p)\to\int\psi(r)\nu(\dd r)$, we obtain
\begin{equation}\label{eq:multi-aspect-deficit}
 \limsup_p\Delta_\psi(x_p,q_p)\le\frac1{d_k}.
\end{equation}
Letting $k\to\infty$ forces $\Delta_\psi(x_p,q_p)\to0$.
Tightness and \cref{thm:bounded-energy} give RQC. The converse is
\cref{thm:sufficient,rem:general-radius} for each fixed $k$.
\end{proof}

The $1/d$ error cannot be removed by assuming that a full spectral limit
determines its independently thinned limit. At $p/N\to1$, compare
$x=D_p^{1/2}g$ with $x=\sqrt R\,g$, where $g$ is standard Gaussian,
$D_p\ge0$ is uniformly bounded with $\ESD(D_p)\Rightarrow\eta$
nondegenerate, and
$R\sim\eta$ is independent of $g$. Gaussian transpose duality gives
the same original law $\eta\boxtimes\MP_1$.
Retaining columns with probability $\vartheta$ and normalization $1/N$
gives respective second limiting moments
\[
 \vartheta^2m_2+\vartheta m_1^2,
 \qquad \vartheta^2m_1^2+\vartheta m_2,
 \qquad m_j=\int r^j\eta(\dd r).
\]
These differ for $0<\vartheta<1$. The counterexample concerns full ESDs,
not matched-parent projected spectra or thinning under freeness
\cite{Mukherjee2022}.

\section{An obstruction to bounded spectral concavity}
\label{app:bounded-obstruction}

The logarithmic tail issue cannot be removed by an arbitrary bounded scalar
substitute while keeping a universal finite-sample concavity argument.

\begin{proposition}\label{prop:bounded-spectral-obstruction}
Let $f:[0,\infty)\to\R$ be continuous and nonconstant, with a finite limit at
infinity. For a probability law $\sigma$ on $\R^2$, put
\[
 \mathcal F_f(\sigma)=\frac12\E\Tr
 f\left(\frac{y_1y_1^T+y_2y_2^T}{2}\right),
 \qquad y_1,y_2\overset{\mathrm{iid}}\sim\sigma.
\]
There exist compactly supported laws $\sigma_+,\sigma_-$ with the same
two-point distribution of $r=\|y\|^2/2$, satisfying
$\E[y\mid r]=0$ and $\E[yy^T\mid r]=rI_2$, such that
$\sigma_0=(\sigma_++\sigma_-)/2$ is spherical and
\[
 \mathcal F_f(\sigma_+)=\mathcal F_f(\sigma_-)
 >\mathcal F_f(\sigma_0).
\]
Their separate angular symmetrizations are both $\sigma_0$.
Thus column-law concavity and universal maximization by angular
symmetrization both fail, even with a fixed radial marginal and exact
conditional isotropy.
\end{proposition}

\begin{proof}
Write $e_\phi=(\cos\phi,\sin\phi)^T$, $\ell=f(\infty)$, and
\[
 K_{a,b}(\phi)=\Tr f(ae_0e_0^T+be_\phi e_\phi^T),\qquad a,b>0.
\]
This kernel is even and $\pi$-periodic. Its two eigenvalues have sum $a+b$
and product $ab\sin^2\phi$. For fixed $a$ and $b\to\infty$, one tends to
infinity and the other to $a\sin^2\phi$. For the pair $(b,b)$, both tend to
infinity except on the null set $\{\phi=0,\pi\}$ modulo $2\pi$.
Since $f$ is bounded, dominated convergence gives
\[
 K_{a,b}\longrightarrow\ell+f(a\sin^2\phi),\qquad
 K_{b,b}\longrightarrow2\ell
 \quad\text{in }L^1([0,2\pi]).
\]
Choose $a$ such that $f$ is not affine on $[0,a]$. Such an $a$ exists
because a bounded affine function on $[0,\infty)$ is constant.
Fourier uniqueness gives an integer $m\ge2$ with
\[
 B=\frac1{2\pi}\int_0^{2\pi}f(a\sin^2\phi)\cos(2m\phi)\,\dd\phi\ne0:
\]
otherwise the continuous even $\pi$-periodic function would be a linear
combination of $1$ and $\cos2\phi$, making $f$ affine on $[0,a]$.
Put
\[
 A_m(r,s)=\frac1{2\pi}\int_0^{2\pi}
 K_{r,s}(\phi)\cos(2m\phi)\,\dd\phi,\qquad
 M_b=\begin{pmatrix}A_m(a,a)&A_m(a,b)\\
                    A_m(a,b)&A_m(b,b)\end{pmatrix}.
\]
As $b\to\infty$, the determinant tends to $-B^2<0$. For a sufficiently
large finite $b$, choose real $u,v$ with $\max(|u|,|v|)\le1$ and
$(u,v)M_b(u,v)^T>0$.

For $\sigma_\pm$, choose $r=a,b$ with equal probabilities and set
$y=\sqrt{2r}\,e_\phi$. Relative to uniform circle measure, use angular
densities $1\pm u\cos(2m\phi)$ at $r=a$ and $1\pm v\cos(2m\phi)$ at $r=b$.
These are nonnegative probability densities. As $m\ge2$, their first and
second angular moments agree with uniform measure, proving the stated
conditional mean and covariance identities. Their average, as well as either
angular symmetrization, is the spherical law $\sigma_0$.

Rotational invariance cancels the linear terms in the two angular densities.
The quadratic term uses
\[
 \iint K_{r,s}(\phi-\vartheta)\cos(2m\phi)\cos(2m\vartheta)
 \frac{\dd\phi\,\dd\vartheta}{(2\pi)^2}=\frac12A_m(r,s).
\]
The equal radial weights and the normalized trace therefore give exactly
\[
 \mathcal F_f(\sigma_\pm)-\mathcal F_f(\sigma_0)
   =\frac1{16}(u,v)M_b(u,v)^T>0.
\]
This proves the proposition.
\end{proof}

This obstruction includes $f(\lambda)=\lambda/(t+\lambda)$, all nonconstant
bounded Bernstein functions, clipped logarithms, and bounded differences of
logarithms. It concerns a universal finite-sample comparison. It does not
exclude an asymptotic inequality using the projected spectral hypothesis,
or a nonlinear functional of several spectral observations.

\begin{corollary}\label{cor:sublog-spectral-obstruction}
The conclusions of \cref{prop:bounded-spectral-obstruction} also hold for
every continuous nonconstant increasing concave $f$ with
$f(x)=o(\log x)$ as $x\to\infty$, including unbounded such functions.
\end{corollary}

\begin{proof}
Subtract $f(0)$ so that $f\ge0$ and $f(0)=0$.
Use the preceding notation. For fixed $a$ the larger eigenvalue of
$ae_0e_0^T+be_\phi e_\phi^T$ belongs to $[b,b+a]$, and
\[
 0\le f(b+a)-f(b)\le a f(b)/b\to0.
\]
The smaller eigenvalue tends to $a\sin^2\phi$ and belongs to $[0,a]$.
Thus $A_m(a,b)$ has the same nonzero limit $B$ as before for a suitable
$a$ and $m\ge2$.

It remains to make $A_m(b,b)\to0$ along a sequence. Put
$h(\phi)=1-|\cos\phi|$ and
\[
 D_b(\phi)=2f(b)-f(b(1+|\cos\phi|))-f(bh(\phi)).
\]
Concavity and monotonicity give
$0\le D_b(\phi)\le f(b)-f(bh(\phi))$.
Writing $g(u)=f(e^u)$, a shift of integration variables yields, for $v\ge0$,
\[
 \int_0^T\{g(u)-g(u-v)\}\,\dd u\le v g(T).
\]
Since $\int_0^{2\pi}-\log h(\phi)\,\dd\phi<\infty$, Tonelli's theorem gives
\[
 \frac1T\int_0^T\int_0^{2\pi}D_{e^u}(\phi)\,
       \frac{\dd\phi}{2\pi}\,\dd u
 \le \frac{C f(e^T)}T\longrightarrow0.
\]
There is therefore a sequence $b\to\infty$ along which the angular mean of
$D_b$ tends to zero. Since $K_{b,b}=2f(b)-D_b$, every nonconstant Fourier
coefficient tends to zero along that sequence. The same two-radius matrix
and angular-density construction proves the result.
\end{proof}

\section*{Acknowledgment}

Language-model tools were used to assist with editing the manuscript.
The author is solely responsible for its mathematical content.

\end{document}